\documentclass[11pt]{article}
\usepackage{amsthm}
\usepackage{amsmath}
\usepackage{amssymb}
\usepackage[toc,page]{appendix}

\allowdisplaybreaks

\makeatletter
\@addtoreset{equation}{section}

\makeatother

\makeatletter

\newdimen\mainfontsize \mainfontsize=1\@ptsize pt
\topskip=\mainfontsize
   \@maxdepth=\maxdepth

\makeatother

\theoremstyle{plain}
\newtheorem{thm}{Theorem}[section]
\newtheorem{lem}[thm]{Lemma}

\newtheorem{cor}[thm]{Corollary}
\theoremstyle{definition}
\newtheorem{defn}[thm]{Definition}

\theoremstyle{remark}
\newtheorem{rem}[thm]{Remark}

\title{Stochastic evolution equations for large portfolios of stochastic volatility models}
\author{Ben Hambly\footnote{hambly@maths.ox.ac.uk} $\,$ and Nikolaos Kolliopoulos\footnote{kolliopoulos@maths.ox.ac.uk (corresponding author)} \\
Mathematical Institute, University of Oxford}
\date{\today} 

\begin{document}
\maketitle

%
%
%
%
%
%

\begin{abstract}
We consider a large market model of defaultable assets in which the asset price processes are modelled as Heston-type stochastic volatility models with default upon hitting a lower boundary. We assume that both the asset prices and their volatilities are correlated through 
systemic Brownian motions. We are interested in the loss process that arises in this setting and we prove the existence of a large 
portfolio limit for the empirical measure process of this system. This limit evolves as a measure valued process and we show that it will have a
density given in terms of a solution to a stochastic partial differential equation of filtering type in the two-dimensional half-space, with a Dirichlet boundary condition. We employ Malliavin calculus to establish the existence of a regular density for the volatility component, and an approximation by models of piecewise constant volatilities combined with a kernel smoothing technique to obtain existence and regularity for the full two-dimensional filtering problem. We are able to establish good regularity properties for solutions, however uniqueness remains an open problem.
\end{abstract}

\section{Introduction}

In the study of large portfolios of assets it is common to model correlation through factor models. In this setting the random 
drivers of individual asset prices come from two independent sources. Firstly there is an idiosyncratic component that reflects the 
movements due to the asset's individual circumstances. Secondly there are systemic components that reflect the impact of 
macroscopic events at the whole market or sector level.  The motivations for this paper come from developing 
such models firstly for credit derivatives such as CDOs which are functions of large portfolios of credit risky assets, but also for 
the evolution of large portfolios which have exposure to a significant proportion of the whole market. 
The financial crisis of 2008 showed that the correlation between credit risky assets was not adequately 
modelled and in this work we will examine the behaviour of a large market when all the individual assets follow
classical stochastic volatility models but are correlated through market factors.

Our starting point is a simple structural model for default in a large portfolio, studied in \cite{BHHJR}. In this setting there is a 
market of $N$ credit risky assets in which the $i$-th asset price $A^i$ for $i=1,\dots,N$ is modelled by a geometric Brownian motion 
with a single systemic risk factor, in that under a risk neutral measure
\begin{eqnarray*}
dA_{t}^{i} &=& rA_{t}^{i}dt+\sigma A_{t}^{i}\left(\sqrt{1-\rho^{2}} dW_{t}^{i}+\rho dW^0_{t}\right),\;\;0\leq t\leq T_{i}
\\
A_{t}^{i} &=& b^{i},\;t>T_{i}\\
A_{0}^{i} &=& a^{i},
\end{eqnarray*}
where $T_{i}=\inf\{t\geq0:\,A_{t}^{i}=b^{i}\}$ for some constant default barrier $b^i$ and the parameters $r,\sigma,\rho, a^i$ are constants. 
Here the Brownian motions $W^0,W^1,\dots$ 
are all independent and we see that it is $W^0$ which captures the macroscopic effects felt by the whole market. We note that
the parameters of the geometric Brownian motions, are the same for each asset, it is just the starting point and idiosyncratic noise which cause the differences in asset 
prices. By rewriting this in terms of a distance to default process and considering the empirical measure
it was shown in \cite{BHHJR} that the limit empirical measure process of the model has a density which is the unique strong solution to
an SPDE on the positive half-line. The density takes values in a weighted Sobolev space as the derivatives of the density may not be
well behaved at the origin. The exact regularity of the density at the origin
was the subject of \cite{Ledger14}, where it was shown that the regularity is a function of the parameter $\rho$. 

This is a naive model and has the problems that would be expected from such a simple structural default model. The short term credit
spreads go to 0 and we see correlation skew when using the model to price the tranches of CDOs. Thus we
wish to investigate a model which incorporates more realistic features. In particular we take stochastic volatility models for
the underlying assets and allow there to be global volatility factors driving the market volatility as well as idiosyncratic
factors for the volatilities of the individual assets. It is also the case that we would like to allow the parameters that describe
the volatility and correlation between assets to vary.

In this paper we consider a large portfolio of $N$ credit risky assets,
where now stochastic volatility models are used instead of Black-Scholes models to describe the evolution of the asset values. The CIR process is used to model the volatility as it is non-negative and mean reverting.
We assume the $i$-th value process $A_i$ satisfies the following system of SDEs
\begin{equation}
\begin{array}{rcl}
dA_{t}^{i} &=& A_{t}^{i}{\mu}_{i}dt+A_{t}^{i}h\left(\sigma_{t}^{i}\right)\left(\sqrt{1-\rho_{1,i}^{2}}dW_{t}^{i}+\rho_{1,i}dW_{t}^{0}\right),\;0\leq t\leq T_{i}\\
d\sigma_{t}^{i} &=& k_{i}(\theta_{i}-\sigma_{t}^{i})dt+\xi_{i}\sqrt{\sigma_{t}^{i}}\left(\sqrt{1-\rho_{2,i}^{2}}dB_{t}^{i}+\rho_{2,i}dB_{t}^{0}\right),\;t\geq 0\\
A_{t}^{i}&=&b^{i},\;t>T_{i}\\
(A_{0}^{i},\,\sigma_{0}^{i}) &=& (a^{i},\,\sigma^{i}),
\end{array}
\label{eq:modelA}
\end{equation}
for all $i\in\{1,2,\dots,N\}$, where $T_{i}=\inf\{t\geq0:\,A_{t}^{i}=b^{i}\}$.
Here, $a^{1},\,a^{2},\,...,\,a^{N}$ and $\sigma^{1},\,\sigma^{2},\,...,\,\sigma^{N}$
are the initial values of the asset prices and the volatilities respectively,
$b^{i}$ is the constant default barrier for the value of the $i$-th asset,
$C_{i}=(k_{i},\,\theta_{i},\,\xi_{i},\,r_{i},\,\rho_{1,i},\,\rho_{2,i})$
for $i\in\{1,\,2,\,...,\,N\}$ are vectors for the various parameters of the model, $h$ is a
function with enough regularity, and $W_{t}^{1},\,B_{t}^{1},\,...,\,W_{t}^{N},\,B_{t}^{N}$
are standard Brownian motions. We will assume that $(a^{i},\,\sigma^{i})$ and $C_{i}$ are drawn independently from some distribution and the Brownian motions are independent
from each other and from each $a^{i},\,\sigma^{i}$ and $C_{i}$. Finally $\left(W_{t}^{0}, B_{t}^{0}\right)$ is a pair of correlated Brownian motions, independent of both $W^i$ and $B^i$ for all $i=1,\dots, N$ as well as $a^{i},\,\sigma^{i}$ and $C_{i}$, which represents the impact of macroscopic factors on each asset and each volatility respectively. 
As is usual in a credit setting we consider the distance to default, or the log asset prices, by 
setting $X_{t}^{i}=\left(\ln A_{t}^{i}-\ln b^{i}\right)$ in \eqref{eq:modelA}.
Applying Ito's formula, our model becomes
\begin{equation}
\begin{array}{rcl}
dX_{t}^{i} &=& \left(r_{i}-\frac{h^{2}(\sigma_{t}^{i})}{2}\right)dt+h(\sigma_{t}^{i})\left(\sqrt{1-\rho_{1,i}^{2}}dW_{t}^{i}+\rho_{1,i}dW_{t}^{0}\right),\;\;0\leq t\leq T_{i} \\
d\sigma_{t}^{i} &=& k_{i}(\theta_{i}-\sigma_{t}^{i})dt+\xi_{i}\sqrt{\sigma_{t}^{i}}\sqrt{1-\rho_{2,i}^{2}}dB_{t}^{i}+\xi\sqrt{\sigma_{t}^{i}}\rho_{2,i}dB_{t}^{0},\;\;t\geq 0 \\
X_{t}^{i} &=& 0,\;t>T_{i} \\
(X_{0}^{i},\,\sigma_{0}^{i}) &=& (x^{i},\sigma^{i}),
\end{array}
\label{eq:model}
\end{equation}
for $i\in\{1,\,2,\,...,\,N\}$, where $x^{i}=\left(\ln a^{i}-\ln b^{i}\right)$
and $T_{i}=\inf\{t\geq0:\,X_{t}^{i}=0\}$, $\forall\,1\leq i\leq N$.

\vspace*{.1in}

An important output from such large portfolio models is the loss process, which gives the proportion 
of assets that have defaulted by any time $t$. This can be used to capture some key quantities 
in risk management, such as the probability of loss from a portfolio and the expected loss given 
default. In credit derivative pricing, the payoffs of CDO tranches
are piecewise linear functions of this loss process. 

In our set up the loss process is given by the mass of the two-dimensional empirical measure
\begin{equation}\label{eq:1.5}
v_{t}^{N}=\frac{1}{N} \sum_{i=1}^{N}\delta_{X_{t}^{i},\sigma_{t}^{i}},
\end{equation}
on $\{0\}\times\mathbb{R}$, while the restriction of $v_{t}^{N}$ to $(0,\infty)\times\mathbb{R}$ for $t \geq 0$ is given by 
\begin{equation}\label{eq:1.6}
v_{1,t}^{N}=\frac{1}{N} \sum_{i=1}^{N}\delta_{X_{t}^{i},\sigma_{t}^{i}}\mathbb{I}_{\{T_{i}>t\}}.
\end{equation}
Section 2 establishes the following convergence result: almost surely and for all positive $t$ we have both
\begin{eqnarray*}
v_{t}^{N}\rightarrow v_{t} = \mathbb{P}\left(\left(X_{t}^{1}, \, \sigma_{t}^{1}\right)\in\cdot\,|\,W_{\cdot}^{0},\,B_{\cdot}^{0},\,\mathcal{G}\right)
\end{eqnarray*}
and 
\begin{eqnarray*}
v_{1,t}^{N}\rightarrow v_{1,t} &=& \mathbb{P}\left(\left(X_{t}^{1}, \, \sigma_{t}^{1}\right)\in\cdot,T_{1}>t\,|\,W_{\cdot}^{0},\,B_{\cdot}^{0},\,\mathcal{G}\right) \\
&=& \mathbb{E}\left[v_{t, \, C_{1}}\left(\cdot\right)\,|\,W_{\cdot}^{0},\,B_{\cdot}^{0},\,\mathcal{G}\right]
\end{eqnarray*}
weakly as $N\rightarrow\infty$, for some $\sigma$-algebra $\mathcal{G}$ containing the initial 
data, where we denote by $v_{t, \, C_{1}}\left( \cdot \right)$ the measure-valued process $\mathbb{P}\left(\left(X_{t}^{1}, \, \sigma_{t}^{1}\right)\in\cdot,T_{1}>t\,|\,W_{\cdot}^{0},\,B_{\cdot}^{0}, \, C_{1}, \, \mathcal{G} \right)$. In Sections 3 and 4, we prove that $v_{t,C_{1}}$ - depending on the information contained in $(W_{.}^{0}$, $B_{.}^{0})$, $\mathcal{G}$ and the coefficient vector $C_{1}=(k_{1},\theta_{1},\xi_{1},r_{1},\rho_{1,1},\rho_{2,1})$
- has a density $u_{t,C_{1}}$ in a weighted Sobolev-Lebesgue space of the two-dimensional positive half-space, but with no differentiability in the second spatial variable $y$. Moreover, it is shown in these sections that given $C_{1}$, $u_{t, \, C_{1}}$ satisfies an SPDE in that function space, along with a Dirichlet boundary condition at $x=0$. In Section 5 we improve the 
regularity by obtaining (weak) differentiability also in $y$, along with some good integrability for the derivative. Our SPDE for $u_{t,C_{1}}$ has the 
form
\newpage{}
\begin{eqnarray}
u_{t,C_{1}}(x,\,y)&=&u_{0}(x,\,y)-\int_{0}^{t}\left(r_1-\frac12 h^{2}(y)\right)\left(u_{s,C_{1}}(x,\,y)\right)_{x}ds \nonumber \\
& & -\int_{0}^{t}k_{1}\left(\theta_{1}-y\right)\left(u_{s,C_{1}}(x,\,y)\right)_{y}ds
 +\frac{1}{2}\int_{0}^{t}h^{2}(y)\left(u_{s,C_{1}}(x,\,y)\right)_{xx}ds \nonumber \\ 
 &  & +\frac{\xi_{1}^{2}}{2}\int_{0}^{t}\left(yu_{s,C_{1}}(x,\,y)\right)_{yy}ds  +\xi_{1}\rho_{3}\rho_{1,1}\rho_{2,1}\int_{0}^{t}\left(h(y)\sqrt{y}u_{s,C_{1}}(x,\,y)\right)_{xy}ds 
\nonumber \\
& &  -\rho_{1,1}\int_{0}^{t}h(y)\left(u_{s,C_{1}}(x,\,y)\right)_{x}dW_{s}^{0}-\xi_{1}\rho_{2,1}\int_{0}^{t}\left(\sqrt{y}u_{t,C_{1}}(x,\,y)\right)_{y}dB_{s}^{0} \qquad \quad \label{eq:spde1}
\end{eqnarray}
where $u_0$ is the initial density, $\rho_{3}$ is the correlation coefficient between $W_{t}^{0}$ and $B_{t}^{0}$ (i.e $dW^0_t \cdot dB^0_t = \rho_3dt$), and the boundary condition $u_{t,C_{1}}(0,y)=0$ is satisfied for all $y\in \mathbb{R}$ and $t \geq 0$. Our result for the case where each parameter vector $C_{i}$ is the same constant vector for 
all $i$ will lead to a limiting empirical process whose density is precisely the solution to the above initial-boundary value problem.

In order to implement the model we could solve the initial-boundary value problem for the SPDE 
numerically for samples $C_{i}$ of the parameters. Then, we can approximate the loss process from
\begin{eqnarray}
\lim_{N\rightarrow\infty}v_{t}^{N}(\{0\}\times\mathbb{R})&=&1-\lim_{N\rightarrow\infty}v_{1,t}^{N}(\mathbb{R}^{2})=
1-\int\int_{\mathbb{R}^{2}}u_{t}(x,y)dxdy \nonumber \\
&=& 1-\mathbb{E}\left[\int_{0}^{\infty}\int_{\mathbb{R}}u_{t,C_{1}}(x,\,y)dxdy,\,|\,W_{\cdot}^{0},\,B_{\cdot}^{0},\,\mathcal{G}\right] \nonumber
\\
&\thickapprox & 1-\frac{1}{n}\sum_{i=1}^{n}\int_{0}^{\infty}\int_{\mathbb{R}}u_{t,c_{i}}(x,\,y)dxdy \label{eq:1.8}
\end{eqnarray}
where $\{c_{1},\,c_{2},\,...,\,c_{n}\}$ is a random sample from the
distribution of $C_{1}$. As the SPDE satisfied by
each $u_{t,c_{i}}$ is driven by the two-dimensional Brownian path
$(W_{.}^{0},\,B_{.}^{0})$, we only need to simulate $(W_{.}^{0},\,B_{.}^{0})$ and solve 
the corresponding SPDEs. This approach is quite efficient when the number of assets $N$ is 
large, since we do not have to simulate the $2N$ idiosyncratic Brownian paths. 

\vspace*{.1in}

There are other approaches to the modelling of credit risk in large portfolios which lead to stochastic partial differential equations. For example in a reduced form setting, see \cite{GSS,SG, STG, SSG}. However, this is the first structural large portfolio model to incorporate stochastic volatility and 
also the first to introduce random coefficients in the SDEs describing the evolution of the 
asset prices. This provides a level of generality which captures many features of asset prices, 
and by taking a large portfolio limit reduces the complexity of the numerical calculations 
arising in risk management and in derivatives pricing applications. Of course, a disadvantage of 
the model is the introduction of a large number of parameters that need to be simulated or 
estimated in order to implement the model. Moreover, the random coefficients need assumptions on their 
joint distributions and, when we use Monte Carlo techniques to estimate expectations in \eqref{eq:1.8}, a very 
large number of simulations may still be required due to the random parameters.
The constant coefficient case is just a special case of the model we have considered, in which the weak limit $v_{1,t}$ of the empirical process coincides with the measure-valued process $v_{t,C_1}$ whose density $u_{t,C_1}$ satisfies our SPDE. Our main aim in this paper is to establish the theoretical background for the general case.  

The calibration of the model for its use in the pricing of CDOs would follow a similar approach 
to that used in \cite{BHHJR}. In its simplest form we take the version where the parameters of 
the model are assumed to be constants. The initial condition would be fitted to the CDS prices of 
the underlying constituents of the portfolio. 
The parameters of the model are then determined from the market tranche prices of the CDOs with
different maturities. This is done by solving the model forward from different parameter settings 
to find model tranche prices and then minimizing the least squares distance between model and 
market to locate the best fit parameters.

The approach to solving the model forward must be done numerically. This type of model is more
computationally intensive than that considered in \cite{BHHJR} as the SPDE is in two dimensions. 
The technique is to generate the two dimensional Brownian path and then solve the SPDE using a 
finite element approach. Speed up could be achieved by extending the work of \cite{GR} where the
multilevel Monte Carlo approach was used for the model of \cite{BHHJR}. We will not discuss the 
numerical analysis for the model as, even in the one-dimensional case, this is challenging.

\vspace*{.1in}

There are significant mathematical challenges in extending large portfolio models to the stochastic volatility setting. A key point is to estimate the boundary behaviour of the empirical measure and, with a non-constant volatility
path, this needs a novel approach. The kernel smoothing technique used by \cite{BHHJR} also 
needs alteration to cope with this volatility process, to enable us to obtain the best possible regularity for our two-dimensional density.

In Section 2, we assume that the initial data satisfies some
reasonable exchangeability conditions in order to obtain the convergence result for the empirical
measure process as $N \rightarrow \infty$. This is not just a two-dimensional version of the 
corresponding result in \cite{BHHJR}, since it gives the convergence of the restriction of the 
empirical measure process to $(0, \infty)\times\mathbb{R}$, while it also gives the form of the limiting
measure-valued process. It includes thus a law of large numbers which is particularly important for dealing with this two-dimensional version of the large portfolio analysis problem. 

In Section~3 we extend some existing Malliavin calculus results and techniques, in order to obtain a strong norm estimate for the density of a CIR process when a component of the driving Brownian Motion (the market factor) is given. We are only able to do this under a condition on the parameters which is stronger than the Feller condition for the CIR process to not hit 0 at any positive time. This is due to the fact that the CIR process does not have Lipschitz coefficients, which means that standard Malliavin calculus techniques for proving the existence of a density of an Ito process are not directly applicable and approximations with processes having better coefficients are needed. In Section~4 we prove a convergence result for a sequence of stopped Ito processes when the sequence of volatility paths decreases pointwise to a continuous and positive path, in order to extend the results of \cite{BHHJR} to the case when the volatility path is non-constant. Combining this with the results of Section~3 in a divide-and-conquer approach, we 
obtain the existence of a regular density for the measure-valued process $v_{t,C_1}$, for any good enough value of $C_1$, and also the SPDE and the boundary condition satisfied by that density. 

In Section~5, we extend the kernel smoothing method developed in \cite{BHHJR, KX99, HL16}, by proving that the standard heat kernel maintains its smoothing and convergence properties, when it is composed with a square root function, and also in certain weighted $L^2$ spaces. This allows us to obtain differentiability of our density in the $y$-direction, and also weighted $L^2$ integrability of the derivative. This improved kernel smoothing method does not work in distribution spaces for our SPDE and thus, the regularity results of the previous two sections are crucial. Finally, in Section~6, we discuss the question of uniqueness of the solution.

\begin{rem}
(1) We will not discuss the issue of asymptotic arbitrage which can arise when there is a large portfolio limit of assets (see \cite{KKD1, KKD2}). As we are using 
the limiting model as an approximation to a large finite model, which will not admit arbitrage, the question is only of theoretical interest.

(2) When calibrating the model for pricing credit derivatives the drift term in the asset's value process is replaced by a known 
interest rate. Including more parameters than that used in \cite{BHHJR} should improve the calibration of the model and may 
allow observed features such as correlation skew in CDOs to be captured. Though we should note that even including jumps in the 
basic model of \cite{BHHJR} still makes it difficult to capture all the observed features of CDO tranche prices \cite{BR12}.

(3) One could view the empirical mean of such a model as a natural model for an index, see \cite{HV15} for the simple case. Here we would
produce a stochastic volatility model for the index and this could be used to price volatility dependent derivatives. 

(4) It would be natural to develop central limit theorems and a large deviation analysis in further work, potentially by adapting and extending appropriately the ideas of \cite{SSO1, SSO2}. For applications in systemic risk it would also be interesting to add a mean field interaction.

(5) The popular Heston model is just a simple case of the model used to describe the evolution of the asset values in our setting, which is obtained when the function $h$ is just a square root function.
\end{rem}

\section{Connection to the probabilistic solution of an SPDE}

In order to study the asymptotic behaviour of our system
of particles, some assumptions have to be made. We assume that $(\Omega,\,\mathcal{F},\,\{\mathcal{F}_{t}\}_{t\geq0},\,\mathbb{P})$
is a filtered probability space with a complete and right-continuous filtration $\{\mathcal{F}_{t}\}_{t\geq0}$,
$\left\{ \left(X_{0}^{1},\,\sigma_{0}^{1}\right),\,\left(X_{0}^{2},\,\sigma_{0}^{2}\right),\,...\right\} $
is an exchangeable sequence of $\mathcal{F}_{0}$-measurable
two-dimensional random vectors (see \cite{Aldous} for more on exchangeability), and $C_{i}=\left(k_{i},\,\theta_{i},\,\xi_{i},\,r_{i},\,\rho_{1,i},\,\rho_{2,i}\right)$
for $i\in\mathbb{N}$ are i.i.d $\mathcal{F}_{0}$-measurable random
vectors in $\mathbb{R}_{+}^{6}$, independent from each $(X_{0}^{j},\,\sigma_{0}^{j})$,
such that $\mathbb{P}$- almost surely we have both $k_{i}\theta_{i}>\frac{3}{4}\xi_{i}^{2}$ and $\rho_{1,i},\,\rho_{2,i} \in (-1, 1)$. We note that the condition on $k_i, \theta_i, \xi_i$ is stronger than the usual Feller condition that ensures that 0 is not reached by a CIR process in finite time.
We also consider an infinite sequence $\{W_{t}^{0},\,B_{t}^{0},\,W_{t}^{1},\,B_{t}^{1},\,W_{t}^{2},\,B_{t}^{2},\,...\}$
of $\mathcal{F}_{t}$ - adapted standard Brownian
motions, in which only $W_{t}^{0}$ and $B_{t}^{0}$ are correlated and their correlation coefficient is denoted by $\rho_{3}$. Under these assumptions and for each $N\in\mathbb{N}$, we
consider the interacting particle system described by equations \eqref{eq:model}
and the corresponding empirical measure processes $v_{t}^{N}$ and
$v_{1,t}^{N}$ given by \eqref{eq:1.5} and \eqref{eq:1.6} respectively. We also
define $v_{2,t}^{N}=v_{t}^{N}-v_{1,t}^{N}$, the restriction of $v_{t}^{N}$
to $\{0\}\times\mathbb{R}$, for all $t\geq0$.

We start with the following convergence theorem, the proof of which is a simple modification of the convergence theorem for the one-dimensional empirical measure process in \cite{BHHJR} and can be found in the Appendix. It is stronger than a convergence result for the empirical measure process $v_{t}^{N}$ but we need it for proving Theorem~\ref{thm:2.3}, a crucial result for establishing the convergence of the $\{0\}\times\mathbb{R}$ - supported component $v_{2,t}^{N}$ (Theorem~\ref{thm:2.5}).

\begin{thm}\label{thm:2.1}
  For each $N\in\mathbb{N}$ and any $t,\,s\geq0$, consider the random measure given by
\[
v_{3,t,s}^{N}=\frac{1}{N}\sum_{i=1}^{N}\delta_{X_{t}^{i},\sigma_{t}^{i},\sigma_{s}^{i}}.
\]
The sequence $v_{3,t,s}^{N}$ of three-dimensional empirical measures
converges weakly to some measure $v_{3,t,s}$ for all $t,\,s\geq0$,
$\mathbb{P}$-almost surely. Moreover, the measure-valued process
$\{v_{3,t,s}:\,t,\,s\geq0\}$ is $\mathbb{P}$-almost surely continuous
in both $t$ and $s$ under the weak topology.
\end{thm}

The convergence result for $v_{t}^{N}$ is a direct consequence of the above theorem and it is given in the following corollary.

\begin{cor}\label{cor:2.2}
The sequence $v_{t}^{N}$ of two-dimensional empirical measures given
by $(1.3)$ converges weakly to some measure $v_{t}$ for all $t\geq0$,
$\mathbb{P}$-almost surely. Moreover, the path $\{v_{t}:\,t\geq0\}$
is $\mathbb{P}$-almost surely continuous under the weak topology.
The measure-valued process $v_{t}$ is the restriction of
$v_{3,t,s}$ to the space of functions which are constant in the third
variable, for any $t\geq0$.
\end{cor}

\begin{proof}
Since $v_{t}^{N}$ is the restriction of $v_{3,t,s}^{N}$ to the space
of functions which are constant in the third variable, the result follows
by testing the measure against such functions and by taking $N\rightarrow\infty$.
\end{proof}
 
Next, we prove a theorem which gives us the form of the
weak limits of the empirical measures $v_{3,t,s}$.

\begin{thm}\label{thm:2.3}
There exists an $\Omega_{0}\subset\Omega$ with $\mathbb{P}(\Omega_{0})=1$
such that for any $\omega\in\Omega_{0}$, we have $\int_{\mathbb{R}^{2}}f dv_{3,t,s}=\mathbb{E}\left[f\left(X_{t}^{1},\,\sigma_{t}^{1},\,\sigma_{s}^{1}\right)|\,W_{.}^{0},\,B_{.}^{0},\,\mathcal{G}\right]$
for any $t,\,s\geq0$ and any $f\in C_{b}(\mathbb{R}^{3};\,\mathbb{R})$,
where $\mathcal{G}$ is some $\sigma$-algebra contained in $\mathcal{F}_{0}$.
\end{thm}

\begin{proof}
By the exchangeability of the initial data, we know that
there exists a $\sigma$-algebra $\mathcal{G}$ contained in $\mathcal{F}_{0}$,
such that the two-dimensional vectors: $\left(X_{0}^{1},\,\sigma_{0}^{1}\right),\,\left(X_{0}^{2},\,\sigma_{0}^{2}\right),\,...$
are i.i.d given $\mathcal{G}$. Moreover, $\left(B_{\cdot}^{k},\,W_{\cdot}^{k},\,C_{k}\right)$
for $k\in\mathbb{N}$ are i.i.d and since they are also independent
from $\left(B_{\cdot}^{0},\,W_{\cdot}^{0},\,\mathcal{G}\right)$,
they are also i.i.d. under the probability measure $\mathbb{P}(\,\cdot\,|\,W^{0},\,B^{0},\,\mathcal{G})$.
The same holds for the two-dimensional vectors $\left(X_{0}^{1},\,\sigma_{0}^{1}\right),\,\left(X_{0}^{2},\,\sigma_{0}^{2}\right),\,...$,
since they are i.i.d given $\mathcal{G}$ and measurable with respect
to the bigger $\sigma$-algebra $\mathcal{F}_{0}$, with $\left(W_{\cdot}^{0},\,B_{\cdot}^{0}\right)$
being independent from $\mathcal{F}_{0}$. Thus, noting that there is a function $g$ such that
\[
\left(X_{t}^{k},\,\sigma_{t}^{k},\,\sigma_{s}^{k}\right)=g\left(t,\,s,\,B_{\cdot}^{k},\,W_{\cdot}^{k},\,B_{\cdot}^{0},\,W_{\cdot}^{0},\,C_{k},\,X_{0}^{k},\,\sigma_{0}^{k}\right)
\]
it follows that $\left(X_{t}^{k},\sigma_{t}^{k},\sigma_{s}^{k}\right)$ for $k\in\mathbb{N}$ are also i.i.d. random vectors under $\mathbb{P}(\,\cdot\,|W^{0},B^{0},\mathcal{G})$.

Thus, for any $f\in C_b(\mathbb{R}^3;\mathbb{R})$ we have
\begin{eqnarray*}
1&\geq &\mathbb{P}\left(\int_{\mathbb{R}^{2}}f dv_{3,t,s}^{N}\rightarrow\mathbb{E}\left[f\left(X_{t}^{1},\sigma_{t}^{1},\,\sigma_{s}^{1}\right)|W_{\cdot}^{0},B_{\cdot}^{0},\,\mathcal{G}\right]\:\forall\,t,\,s\in\mathbb{Q}^{+}\right) \\
&=& \mathbb{E}\left[\mathbb{P}\left(\int_{\mathbb{R}^{2}}f dv_{3,t,s}^{N}\rightarrow\mathbb{E}\left[f\left(X_{t}^{1},\sigma_{t}^{1},\sigma_{s}^{1}\right)|\,W_{\cdot}^{0},\,B_{\cdot}^{0},\,\mathcal{G}\right]\:\forall\,t,\,s\in\mathbb{Q}^{+}|W_{\cdot}^{0},B_{\cdot}^{0},\mathcal{G}\right)\right] \\
&=& \mathbb{E}\left[1-\mathbb{P}\left({\displaystyle \cup_{t,s\in\mathbb{Q}^{+}}}\left\{ \int_{\mathbb{R}^{2}}f dv_{3,t,s}^{N}\nrightarrow\mathbb{E}\left[f\left(X_{t}^{1},\sigma_{t}^{1},\sigma_{s}^{1}\right)|W_{\cdot}^{0},B_{\cdot}^{0},\mathcal{G}\right]\right\} |W_{\cdot}^{0},B_{\cdot}^{0},\mathcal{G}\right)\right] \\
&\geq &\mathbb{E}\left[1-\sum_{t,s\in\mathbb{Q}^{+}}\mathbb{P}\left(\int_{\mathbb{R}^{2}}f dv_{3,t,s}^{N}\nrightarrow\mathbb{E}\left[f\left(X_{t}^{1},\sigma_{t}^{1},\,\sigma_{s}^{1}\right)|W_{\cdot}^{0},B_{\cdot}^{0},\mathcal{G}\right]|W_{\cdot}^{0},B_{\cdot}^{0},\mathcal{G}\right)\right]=1,
\end{eqnarray*}
where, in the last expectation, by the strong law of large numbers, for each $t,s$ the probability that there is no convergence is zero.
Hence, there is an $\Omega^{f}\subset\Omega$ (depending on $f$)
with $\mathbb{P}\left(\Omega^{f}\right)=1$, such that
\[
\int_{\mathbb{R}^{2}}f dv_{3,t,s}^{N}\rightarrow\mathbb{E}\left[f\left(X_{t}^{1},\,\sigma_{t}^{1},\,\sigma_{s}^{1}\right)|\,W_{\cdot}^{0},\,B_{\cdot}^{0},\,\mathcal{G}\right]\:\forall\,t,\,s\in\mathbb{Q}^{+}
\]
as $N\rightarrow\infty$, for all $\omega\in\Omega^{f}$.

If we denote by $\Omega_{0}^{f}$ the intersection of $\Omega^{f}$
with the set of events for which the results of Theorem \ref{thm:2.1} hold,
we see that $\mathbb{P}\left(\Omega_{0}^{f}\right)=1$ and that for
all $\omega\in\Omega_{0}^{f}$ we have 
\begin{equation}\label{eq:2.1}
\int_{\mathbb{R}^{2}}f dv_{3,t,s}=\mathbb{E}\left[f\left(X_{t}^{1},\,\sigma_{t}^{1},\,\sigma_{s}^{1}\right)|\,W_{.}^{0},\,B_{.}^{0},\,\mathcal{G}\right]
\end{equation}
for any $t,\,s\in\mathbb{Q}^{+}$. Since both quantities in $\eqref{eq:2.1}$ are continuous
in $(t,\,s)$ (this follows from Theorem 2.1 for the LHS, and by using
the dominated convergence theorem for the RHS) and since they coincide
for any $t,\,s\in\mathbb{Q}^{+}$, we conclude that they coincide
for all $t,\,s\geq0$ in $\Omega_{0}^{f}$. 

Finally, taking the intersection of all $\Omega_{0}^{p}$
for all $p$ belonging to a countable and dense subset $D$ of $C_{b}\left(\mathbb{R}^{3};\,\mathbb{R}\right)$,
we obtain the desired set $\Omega_{0}$. This follows from the fact
that both quantities in $\eqref{eq:2.1}$ are bounded functionals in $C_{b}\left(\mathbb{R}^{3};\,\mathbb{R}\right)$
with the supremum norm, where for the LHS this follows by taking limits
in the obvious inequality, $\int_{\mathbb{R}^{2}}f dv_{3,t,s}^{n}\leq||f||_{\infty},\:\forall\,n\in\mathbb{N}$.
Our proof is now complete.
\end{proof}

\begin{cor}\label{cor:2.4}
Let $\{v_{t}:\,t\geq0 \}$ be the measure-valued process defined in Corollary~\ref{cor:2.2}. 
There exists an $\Omega_{0}\subset\Omega$ with $\mathbb{P}(\Omega_{0})=1$
such that for any $\omega\in\Omega_{0}$, we have $\int_{\mathbb{R}^{2}}f dv_{t}=\mathbb{E}\left[f\left(X_{t}^{1},\,\sigma_{t}^{1}\right)|\,W_{.}^{0},\,B_{.}^{0},\,\mathcal{G}\right]$
for any $t\geq0$ and for any $f\in C_{b}\left(\mathbb{R}^{3};\,\mathbb{R}\right)$,
where $\mathcal{G}$ is the $\sigma$-algebra defined in Theorem~\ref{thm:2.3}.
\end{cor}

\begin{proof}
This result follows by testing the measure against functions which are constant in the
third variable and by recalling Corollary~\ref{cor:2.2}.
\end{proof}

The above corollary completes the convergence result for
$v_{t}^{N}$ which was given in Corollary~\ref{cor:2.2}. However, what we need is a similar result for its restriction to $\{0\}\times\mathbb{R}$, that is $v_{2,t}^{N}$. This is given in the following Theorem, the proof of which is based on the more general convergence result given in Theorem~\ref{thm:2.1} and Theorem~\ref{thm:2.3}.
\begin{thm}\label{thm:2.5}
There exists a measure-valued process $\{v_{2,t}:\,t\geq0\}$
and an $\Omega_{1}\subset\Omega_{0}$ with $\mathbb{P}(\Omega_{1})=1$,
such that for any $\omega\in\Omega_{0}$ we have that $v_{2,t}^{N}\rightarrow v_{2,t}$
weakly for all $t\geq0$. Moreover, we have $\int_{\mathbb{R}^{2}}f dv_{2,t}=\mathbb{E}\left[f\left(X_{t}^{1},\,\sigma_{t}^{1}\right)\mathbb{I}_{\{T_{1}<t\}}|\,W_{.}^{0},\,B_{.}^{0},\,\mathcal{G}\right]$
for all $t\geq0$ and for all $f\in C_{b}\left(\mathbb{R}^{2};\,\mathbb{R}\right)$,
where $\mathcal{G}$ is the $\sigma$-algebra defined in Theorem~\ref{thm:2.3}.
\end{thm}

\begin{proof}
First, observe that when $\left(W^{0},\,B^{0},\,\mathcal{G}\right)$
is given, $T_{1}$ has a continuous distribution, since it is a stopping
time for the Ito process $X_{t}^{1}$. Moreover, observe that
\[
v_{2,t}^{N}=\delta_{0}\sum_{i=1}^{N}\delta_{\sigma_{t}^{i}}\mathbb{I}_{\{T_{i}<t\}}
\]
and that
\[
\mathbb{E}\left[f\left(X_{t}^{1},\,\sigma_{t}^{1}\right)\mathbb{I}_{\{T_{1}<t\}}\,|\,W_{.}^{0},\,B_{.}^{0},\,G\right]=\mathbb{E}\left[f\left(0,\,\sigma_{t}^{1}\right)\mathbb{I}_{\{T_{1}<t\}}\,|\,W_{\cdot}^{0},\,B_{\cdot}^{0},\,\mathcal{G}\right],
\]
for any $t\geq0$ and $f\in C_{b}\left(\mathbb{R}^{2};\,\mathbb{R}\right)$,
which means that we only need to work with functions $f$ which are
constant in the first variable.

Let now $v_{2,t}(\cdot)$ be the probability law $\mathbb{P}\left(\left(X_{t}^{1},\,\sigma_{t}^{1}\right)\in\cdot
\,;\,T_{1}\leq t\,|\,W^{0},\,B^{0},\,\mathcal{G}\right)$
and fix a function $f$ in $C_{b}\left(\mathbb{R};\,\mathbb{R}\right)$
with positive values. Since $T_{i}$ is adapted to $\left(X_{t}^{i},\,\sigma_{t}^{i}\right)$
for any $i\in\mathbb{N}$, by the independence obtained in the first paragraph of the proof
of Theorem~\ref{thm:2.3} and a Law of Large Numbers argument similar to the
one in that proof, we have that the desired convergence holds for
the chosen $f$, for all $t\in\mathbb{Q}^{+}$ and for all $\omega$
in some $\Omega_{1,f}\subset\Omega$ with $\mathbb{P}\left(\Omega_{1,f}\right)=1$. By intersecting with the full-probability set $\Omega_0$ given in Theorem~\ref{thm:2.3}, we can take $\Omega_{1,f}\subset\Omega_{0}$. 
Now for a $t\geq0$ and an $\omega\in\Omega_{1,f}$, we pick any two
rational numbers $t_{1},\,t_{2}\geq0$ such that $t_{1}\leq t\leq t_{2}$.
Then we have
\begin{eqnarray*}
& & \liminf_{N\rightarrow\infty}\frac{1}{N}\sum_{i=1}^{N}f\left(\sigma_{t}^{i}\right)\mathbb{I}_{\{T_{i}\leq t\}} \\
& & \qquad =  \liminf_{N\rightarrow\infty}\frac{1}{N}\left[\sum_{i=1}^{N}f\left(\sigma_{t_{1}}^{i}\right)\mathbb{I}_{\{T_{i}<t\}}+
\sum_{i=1}^{N}\left(f\left(\sigma_{t}^{i}\right)-f\left(\sigma_{t_{1}}^{i}\right)\right)\mathbb{I}_{\{T_{i}<t\}}\right] \\
& & \qquad \geq \liminf_{N\rightarrow\infty}\frac{1}{N}\left[\sum_{i=1}^{N}f\left(\sigma_{t_{1}}^{i}\right)\mathbb{I}_{\{T_{i}<t\}}-\sum_{i=1}^{N}\left|f\left(\sigma_{t}^{i}\right)-f\left(\sigma_{t_{1}}^{i}\right)\right|\right] \\
& & \qquad \geq \liminf_{N\rightarrow\infty}\frac{1}{N}\sum_{i=1}^{N}f\left(\sigma_{t_{1}}^{i}\right)\mathbb{I}_{\{T_{i}<t_{1}\}}-\liminf_{N\rightarrow\infty}\frac{1}{N}\sum_{i=1}^{N}\left|f\left(\sigma_{t}^{i}\right)-f\left(\sigma_{t_{1}}^{i}\right)\right|,
\end{eqnarray*}
where the first term equals $\mathbb{E}\left[f\left(\sigma_{t_1}^{1}\right)\mathbb{I}_{\{T_{1}<t_{1}\}}\,|\,W_{.}^{0},\,B_{.}^{0},\,\mathcal{G}\right]$
for each rational time $t_1$.
Next, by recalling Theorem~\ref{thm:2.3} for $f(x,\,y,\,z)=|f(y)-f(z)|$, and $s=t_{1}$ we find that the second term equals $\mathbb{E}\left[|f\left(\sigma_{t}^{1}\right)-f\left(\sigma_{t_{1}}^{1}\right)|\,|\,W_{.}^{0},\,B_{.}^{0},\,\mathcal{G}\right]$.
Now taking $t_{1}\rightarrow t$ and using the Dominated Convergence
Theorem and the fact that the random variable $T_{1}$ has a continuous distribution, we obtain
\begin{equation}\label{eq:2.4}
\liminf_{N\rightarrow\infty}\frac{1}{N}\sum_{i=1}^{N}f(\sigma_{t}^{i})\mathbb{I}_{\{T_{i}\leq t\}}\geq\mathbb{E}\left[f(\sigma_{t}^{1})\mathbb{I}_{\{T_{1}<t\}}\,|\,W_{\cdot}^{0},\,B_{\cdot}^{0},\,\mathcal{G}\right].
\end{equation}

Similarly, we have 
\begin{eqnarray*}
& & \limsup_{N\rightarrow\infty}\frac{1}{N}\sum_{i=1}^{N}f\left(\sigma_{t}^{i}\right)\mathbb{I}_{\{T_{i}\leq t\}} \\
& & \qquad =\limsup_{N\rightarrow\infty}\frac{1}{N}\left[\sum_{i=1}^{N}f\left(\sigma_{t_{2}}^{i}\right)\mathbb{I}_{\{T_{i}<t\}}+\sum_{i=1}^{N}\left(f\left(\sigma_{t}^{i}\right)-f\left(\sigma_{t_{2}}^{i}\right)\right)\mathbb{I}_{\{T_{i}<t\}}\right] \\
& & \qquad \leq \limsup_{N\rightarrow\infty}\frac{1}{N}\left[\sum_{i=1}^{N}f\left(\sigma_{t_{2}}^{i}\right)\mathbb{I}_{\{T_{i}<t\}}+\sum_{i=1}^{N}\left|f\left(\sigma_{t}^{i}\right)-f\left(\sigma_{t_{2}}^{i}\right)\right|\right] \\
& & \qquad\leq \limsup_{N\rightarrow\infty}\frac{1}{N}\sum_{i=1}^{N}f\left(\sigma_{t_{2}}^{i}\right)\mathbb{I}_{\{T_{i}<t_{2}\}}+\limsup_{N\rightarrow\infty}\frac{1}{N}\sum_{i=1}^{N}\left|f\left(\sigma_{t}^{i}\right)-f\left(\sigma_{t_{2}}^{i}\right)\right|
\end{eqnarray*}
and by the same argument for the rational number $t_{2}\rightarrow t$, we find
\begin{equation}\label{eq:2.5}
\limsup_{N\rightarrow\infty}\frac{1}{N}\sum_{i=1}^{N}f\left(\sigma_{t}^{i}\right)\mathbb{I}_{\{T_{i}\leq t\}}\leq\mathbb{E}\left[f\left(\sigma_{t}^{1}\right)\mathbb{I}_{\{T_{1}<t\}}\,|\,W_{\cdot}^{0},\,B_{\cdot}^{0},\,\mathcal{G}\right].
\end{equation}

Hence, by \eqref{eq:2.4} and \eqref{eq:2.5}, the desired
convergence holds in $\Omega_{1,f}\subset\Omega_{0}$ for all $t\geq0$
and any $f$ in $C_{b}(\mathbb{R};\,\mathbb{R})$ with positive values.
By linearity, and since every continuous function can be decomposed
into its continuous positive part and its continuous negative part,
we can have our convergence result for any $f\in C_{b}(\mathbb{R};\,\mathbb{R})$.
Let now $S=\{p_{i}:\,i\in\mathbb{N}\}$ be a countable basis of $C_{b}(\mathbb{R};\,\mathbb{R})$.
Then, by linearity, the desired convergence holds for all $t\geq0$,
all $p\in\left[S\right]$ and all $\omega\in\Omega_{1}=\cap_{i\in\mathbb{N}}\Omega_{1,p_{i}}$,
with $\mathbb{P}(\Omega_{1})=1$. Now for any $f\in C_{b}(\mathbb{R};\,\mathbb{R})$
and $\epsilon>0$, we can pick $p\in\left[S\right]$ such that $||f-p||_{\infty}<\frac{\epsilon}{3}$,
so we have
\begin{eqnarray*}
& & \left|\int_{\mathbb{R}^{2}}f(y) dv_{2,t}^{N}(x,y)-\int_{\mathbb{R}^{2}}f(y) 
dv_{2,t}(x,y)\right| \\
& & \qquad \leq\left|\int_{\mathbb{R}^{2}}(f(y)-p(y)) dv_{2,t}^{N}(x,y)\right| +\left|\int_{\mathbb{R}^{2}}p(y)
 dv_{2,t}^{N}(x,y)-\int_{\mathbb{R}^{2}}p(y) dv_{2,t}(x,y)\right| \\
& & \qquad\qquad\qquad +\left|\int_{\mathbb{R}^{2}}(p(y)-f(y)) dv_{2,t}(x,y)\right| \\
& & \qquad \leq ||f-p||_{\infty}+\left|\int_{\mathbb{R}^{2}}p(y) dv_{2,t}^{N}(x,y)-\int_{\mathbb{R}^{2}}p(y) dv_{2,t}(x,y)\right|+||f-p||_{\infty} \\
& & \qquad \leq3\times\frac{\epsilon}{3}=\epsilon,
\end{eqnarray*}
for all $N$ sufficiently large. Thus, we have our convergence result
for all $t\geq0$ and all $\omega\in\Omega_{1}=\cap_{i\in\mathbb{N}}\Omega_{1,p_{i}}$
with $\mathbb{P}(\Omega_{1})=1$, so we are done.
\end{proof}

Next, by Corollary~\ref{cor:2.2} and Theorem~\ref{thm:2.5}, we have that 
$v_{1,t}^{N}=v_{t}^{N}-v_{2,t}^{N}\rightarrow v_{t}-v_{2,t}=:v_{1,t}$
weakly, for all $t\geq0$, $\mathbb{P}$-almost surely. Also, it follows
from Corollary~\ref{cor:2.4} and Theorem~\ref{thm:2.5} that
\begin{eqnarray*}
\int_{\mathbb{R}^{2}}f dv_{1,t} &=& \mathbb{E}\left[f\left(X_{t}^{1},\,\sigma_{t}^{1}\right)\mathbb{I}_{\{T_{1}>t\}}\,|\,W_{\cdot}^{0},\,B_{\cdot}^{0},\,\mathcal{G}\right]  \\
&=& \mathbb{E}\left[\mathbb{E}\left[f\left(X_{t}^{1},\,\sigma_{t}^{1}\right)\mathbb{I}_{\{T_{1}>t\}}\,|\,W_{\cdot}^{0},\,B_{\cdot}^{0},\,C_{1},\,\mathcal{G}\right]\,|\,W_{\cdot}^{0},\,B_{\cdot}^{0},\,\mathcal{G}\right],
\end{eqnarray*}
for any $f\in C_{b}\left(\mathbb{R}^{2};\,\mathbb{R}\right)$ and
$t\geq0$. It is therefore reasonable to study the behaviour of the
process of measures $v_{t,C_{1}}(\cdot)$ defined as
\[
v_{t,C_{1}}\left(\cdot \right)=\mathbb{P}\left[\left(X_{t}^{1},\sigma_{t}^{1}\right) \in \cdot , T_{1}>t \,|W_{\cdot}^{0},B_{\cdot}^{0},C_{1},\mathcal{G}\right],
\]
for a given value of $C_{1}=\left(k_{1},\,\theta_{1},\,\xi_{1},\,r_{1},\,\rho_{1,1},\,\rho_{2,1}\right)$.
The behaviour of this process of measures is given in the following
Theorem.

\begin{thm}\label{thm:2.6}
Let $A$ be the two-dimensional differential operator
mapping any smooth function $f:\mathbb{R}^{+}\times\mathbb{R}\rightarrow\mathbb{R}$
to
\begin{eqnarray*}
Af\left(x,\,y\right) &=& \left(r_{1}-\frac{h^{2}\left(y\right)}{2}\right) f_{x}\left(x,\,y\right)+k_{1}\left(\theta_{1}-y\right)f_{y}
\left(x,\,y\right)+\frac{1}{2}h^{2}\left(y\right)f_{xx}\left(x,\,y\right) \\
& & \qquad +\frac{1}{2}\xi_{1}^{2}yf_{yy}\left(x,\,y\right)+\xi_{1}\rho_{3}\rho_{1,1}\rho_{2,1}h(y)\sqrt{y}f_{xy}\left(x,\,y\right)
\end{eqnarray*}
for all $\left(x,\,y\right)\in\mathbb{R}^{+}\times\mathbb{R}$. Then,
the measure-valued stochastic process $v_{t,C_{1}}$ satisfies the
following weak form SPDE
\begin{eqnarray*}
\int_{\mathbb{R}^{2}}f\left(x,\,y\right) dv_{t,C_{1}}\left(x,\,y\right)&=&\int_{\mathbb{R}^{2}}f\left(x,\,y\right) dv_{0,C_{1}}\left(x,\,y\right) \\
& & +\int_{0}^{t}\int_{\mathbb{R}^{2}}Af\left(x,\,y\right) dv_{s,C_{1}}\left(x,\,y\right)ds \\
& & +\rho_{1,1}\int_{0}^{t}\int_{\mathbb{R}^{2}}h\left(y\right) f_{x}\left(x,\,y\right) dv_{s,C_{1}}\left(x,\,y\right)dW_{s}^{0} \\
& & +\xi_{1}\rho_{2,1}\int_{0}^{t}\int_{\mathbb{R}^{2}}\sqrt{y} f_{y}\left(x,\,y\right) dv_{s,C_{1}}\left(x,\,y\right)dB_{s}^{0},
\end{eqnarray*}
for all $t\geq0$ and any $f\in C_{0}^{test}=\left\{ g\in C_{b}^{2}\left(\mathbb{R}^{+}\times\mathbb{R}\right):\,g\left(0,\,y\right)=0,\:\forall\,y\in\mathbb{R}\right\} $.
\end{thm}

\begin{proof}
By using Ito's formula for the stopped two-dimensional process $\left\{ \left(X_{t}^{1},\,\sigma_{t}^{1}\right):\,t\geq0\right\} $
given by \eqref{eq:model} and by recalling that $f\left(0,\,y\right)=0$ for
all $y$, we obtain
\begin{eqnarray*}
& & f\left(X_{t\wedge T_{1}}^{1},\,\sigma_{t}^{1}\right) \\
& & \qquad = f\left(X_{0}^{1},\,\sigma_{0}^{1}\right) \\
& & \qquad\qquad +\int_{0}^{t}\left[f_{x}\left(X_{s}^{1},\,\sigma_{s}^{1}\right)\left(r_{1}-\frac{h^{2}\left(\sigma_{s}^{1}\right)}{2}\right)+k_{1}f_{y}\left(X_{s}^{1},\,\sigma_{s}^{1}\right)\left(\theta_{1}-\sigma_{s}^{1}\right)\right]\mathbb{I}_{\{T_{1}>s\}}ds \\
& &  \qquad\qquad+\frac{1}{2}\int_{0}^{t}\left[f_{xx}\left(X_{s}^{1},\,\sigma_{s}^{1}\right)h^{2}\left(\sigma_{s}^{1}\right)+\xi_{1}^{2}f_{yy}\left(X_{s}^{1},\,\sigma_{s}^{1}\right)\sigma_{s}^{1}\right]\mathbb{I}_{\{T_{1}>s\}}ds \\
& &  \qquad\qquad+\xi_{1}\rho_{3}\rho_{1,1}\rho_{2,1}\int_{0}^{t}f_{xy}\left(X_{s}^{1},\,\sigma_{s}^{1}\right)h\left(\sigma_{s}^{1}\right)\sqrt{\sigma_{s}^{1}}\mathbb{I}_{\{T_{1}>s\}}ds \\
& &  \qquad\qquad+\int_{0}^{t}f_{x}\left(X_{s}^{1},\,\sigma_{s}^{1}\right)\mathbb{I}_{\{T_{1}>s\}}h\left(\sigma_{s}^{1}\right)\rho_{1,1}dW_{s}^{0} \\
& & \qquad\qquad+\xi_{1}\int_{0}^{t}f_{y}\left(X_{s}^{1},\,\sigma_{s}^{1}\right)\mathbb{I}_{\{T_{1}>s\}}\sqrt{\sigma_{s}^{1}}\rho_{2,1}dB_{s}^{0}
\\
& &  \qquad\qquad+\int_{0}^{t}f_{x}\left(X_{s}^{1},\,\sigma_{s}^{1}\right)\mathbb{I}_{\{T_{1}>s\}}h(\sigma_{s}^{1})\sqrt{1-\rho_{1,1}^{2}}dW_{s}^{1}
\\
& &  \qquad\qquad+\xi_{1}\int_{0}^{t}f_{y}\left(X_{s}^{1},\,\sigma_{s}^{1}\right)\mathbb{I}_{\{T_{1}>s\}}\sqrt{\sigma_{s}^{1}}\sqrt{1-\rho_{2,1}^{2}}dB_{s}^{1}
\end{eqnarray*}
and the desired result follows by taking conditional expectations
given $\left(W_{\cdot}^{0},\,B_{\cdot}^{0}\right)$, $C_{1}$ and
$\mathcal{G}$, by noticing that Ito integrals with respect to $B_{\cdot}^{1}$
and $W_{\cdot}^{1}$ vanish due to the pairwise independence of the
Brownian Motions, and by taking the given coefficients out of the
conditional expectations.
\end{proof}

\section{Volatility Analysis - A Malliavin Calculus approach}

Now that we have connected our problem to the study of the probabilistic
solution of an SPDE, we need to establish the best possible regularity result for that solution. Before showing that the measure-valued process
$\mathbb{E}\left[\cdot(X_{t}^{1},\,\sigma_{t}^{1}) \mathbb{I}_{T_1 > t}\,|\,B_{\cdot}^{0},\,W_{\cdot}^{0},\, C_1, \, \mathcal{G}\right]$
does indeed have a density for almost all paths
of $\left(B_{\cdot}^{0},\,W_{\cdot}^{0}\right)$ with some good regularity, it is natural (and important as we will see) to ask whether the same holds for the
1-dimensional measure-valued process describing the evolution of $\mathbb{E}\left[f(\sigma_{t})\,|\,B_{\cdot}^{0},\,\mathcal{G}\right]$, for suitable $f$, where $\sigma$ is a CIR process driven by a combination of $B_{\cdot}^{0}$ and $B_{\cdot}^{1}$, that is a process satisfying
\begin{equation}\label{eq:3.1}
d\sigma_{t}=k(\theta-\sigma_{t})dt+\xi\sqrt{\sigma_{t}}\sqrt{1-\rho_{2}^{2}}dB_{t}^{1}+\xi\sqrt{\sigma_{t}}\rho_{2}dB_{t}^{0},\;t\geq0.
\end{equation}
We assume that the coefficients of equation \eqref{eq:3.1} satisfy: $k\theta>\frac{3}{4}\xi^{2}$,
which is stronger than the standard Feller boundary condition for
a CIR process, and also $\rho_2 \in (-1, \, 1)$. Then, the answer to our question is given in the next
theorem.
\begin{thm}\label{thm:3.1}
Assume that $\sigma_{0}$ is a random variable in $L^{p}\left(\Omega,\,\mathcal{F}_{0},\,\mathbb{P}\right)$
for all $p>p_{0}=1-\frac{2k\theta}{\xi^{2}}$, such that given $\mathcal{G}$,
$\sigma_{0}$ has a continuous density $p_{0}(\cdot\,|\,\mathcal{G})$
which is supported in $[0,\infty)$ and which satisfies
\[
\mathbb{E}\left[||p_{0}(\cdot\,|\,\mathcal{G})||_{\infty}^{\gamma}\right]<\infty,
\]
for all $\gamma\in\left[-\frac{2k\theta}{\xi^{2}},\,1\right]$. Then,
for every path of $B_{\cdot}^{0}$ and $t\geq0$, the conditional probability
measure $\mathbb{P}(\sigma_{t}\in A\,|\,B_{\cdot}^{0},\,\mathcal{G})$
posseses a continuous probability density $p_{t}(\cdot\,|\,B_{\cdot}^{0},\,\mathcal{G})$
which is supported in $[0,\infty)$. Moreover, for any $T>0$,
any $1<q<\frac{4k\theta}{3\xi^{2}}$ and any $\alpha\geq0$, we have
\[
M_{B^{0},\alpha}=\sup_{t\leq T}\left(\sup_{y\geq0}\left(y^{\alpha}p_{t}(y\,|\,B_{\cdot}^{0},\,\mathcal{G})\right)\right)\in L^{q}\left(\Omega\right)
\]
\end{thm}

To prove this Theorem we need a few lemmas. First, we will need the following finiteness result for the moments of the supremum of a CIR / Ornstein-Uhlenbeck process up to some finite time. The proof of this technical lemma can be found in the Appendix.
\begin{lem}\label{lem:3.2}
Under the assumptions of Theorem~\ref{thm:3.1}, for any $p\geq0$ and $T>0$
we have
\[
\mathbb{E}\left[\sup_{0\leq t\leq T}\sigma_{t}^{p}\right]<\infty.
\]
Moreover, if $\left\{ u_{t}:\,t\geq0\right\} $ is the Ornstein-Uhlenbeck process which solves the SDE
\[
du_{t}=-\frac{k}{2}u_{t}dt+\frac{\xi}{2}\left(\sqrt{1-\rho_{2}^{2}}dB_{t}^{1}+\rho_{2}dB_{t}^{0}\right),
\]
under the initial condition $u_{0}=\sqrt{\sigma_{0}}$, then we have also
\[
\mathbb{E}\left[\sup_{0\leq t\leq T}u_{t}^{2}\right]<\infty.
\]
\end{lem}

Next, we need a few results that involve the notion of Malliavin differentiability. The Malliavin
derivative of a random variable adapted to a Brownian path is a stochastic process measuring, in 
some sense, the rate of change of the random variable when the Brownian path changes at any time 
$t$. Extending this to random variables taking values in Banach spaces, we can define the $k$-th 
Malliavin derivative as a random function of $k$ time variables (provided that it exists). The 
existence and behaviour of these derivatives are inextricably connected to the existence of a 
regular density for the random variable. We refer to \cite{Nualart} for the basics of Malliavin 
calculus and, as in \cite{Nualart}, we denote by $\mathbb{D}^{n,p}(V)$ the space of $n$-times 
Malliavin differentiable random variables taking values in the Banach space $V$, whose $k$-th 
Malliavin derivative has an $L^2$ norm (as a function in $k$ time variables taking values 
in $V$) of a finite $L^p$ norm as a random variable, for all $0 \leq k \leq n$.  

In \cite{AE1} and \cite{AE2} it is proven that the CIR process
has a Malliavin derivative under the probability measure $\mathbb{P}$, which is given by a quite 
similar formula. In \cite{AE1} it is also proven that under our strong assumption $k\theta > 
\frac{3}{4}{\xi}^2$, a second Malliaving derivative with some regularity also exists. The next two 
lemmas extend these results to the case where the path of the market noise $B_{\cdot}$ is given. 
This is exactly what we need in order to prove Theorem~\ref{thm:3.1}. The proofs of these extensions are 
more or less based on the same ideas as the corresponding initial results (except that 
Lemma~\ref{lem:3.2} is also needed at some points) and can be found in the Appendix. 

\begin{lem}\label{lem:3.3}
There exists an $\Omega_{1}\subset\Omega$ with $\mathbb{P}\left(\Omega_{1}\right)=1$,
such that for all $\omega\in\Omega_{1}$ the random probability measure
$\mathbb{P}(\cdot\,|\,B_{\cdot}^{0},\,\mathcal{G})$ has the following
property: Under $\mathbb{P}(\cdot\,|\,B_{\cdot}^{0},\,\mathcal{G})$,
the process $\{\sigma_{t}:\,t\geq0\}$ has a Malliavin derivative
with respect to the Brownian Motion $B_{\cdot}^{1}$ which is given
by
\begin{equation}\label{eq:3.2}
D_{t'}\sigma_{t}=\xi\sqrt{1-\rho^{2}_{2}}e^{-\int_{t'}^{t}\left[\left(\frac{k\theta}{2}-\frac{\xi^{2}}{8}\right)\frac{1}{\sigma_{s}}+\frac{k}{2}\right]ds}\sqrt{\sigma_{t}},
\end{equation}
for all $t>0$ and $0\leq t'\leq t$. This is a process in $t'$
which belongs to $L_{B_{\cdot}^{0},\,\mathcal{G}}^{2}\left([0,\,t]\times\Omega\right)$
for any fixed $t\geq0$, where the notation $L_{B_{\cdot}^{0},\,\mathcal{G}}^{q}$
is used for any $L^{q}$ space when the probability measure $\mathbb{P}$
is replaced by $\mathbb{P}(\cdot\,|\,B_{\cdot}^{0},\,\mathcal{G})$. 
\end{lem}

\begin{lem}\label{lem:3.4}
For any $1\leq q'<\frac{4k\theta}{3\xi^{2}}$ and $T>0$,
there exists an $\Omega_{2}\subset\Omega$ with $\mathbb{P}\left(\Omega_{2}\right)=1$,
such that for all $\omega\in\Omega_{2}$ the random probability measure
$\mathbb{P}(\cdot\,|\,B_{\cdot}^{0},\,\mathcal{G})$ has the following
property: Under $\mathbb{P}(\cdot\,|\,B_{\cdot}^{0},\,\mathcal{G})$,
the process $\{\sigma_{t}:\,0<t\le T\}$ belongs to the space $\mathbb{D}^{2,q'}$
with respect to the Brownian Motion $B_{\cdot}^{1}$, and the second
order Malliavin derivative is given by
\begin{equation}\label{eq:3.11}
D_{t',t''}^{2}\sigma_{t}=D_{t'}\sigma_{t}\times\left[\int_{t'}^{t}\left(\frac{k\theta}{2}-\frac{\xi^{2}}{8}\right)\frac{1}{\left(\sigma_{s}\right)^{2}}D_{t''}\sigma_{s}ds+\frac{1}{2\sigma_{t}}D_{t''}\sigma_{t}\right],
\end{equation}
for all $0<t\leq T$ and $0\leq t',\,t''\leq t$, where the first
order derivatives are given by Lemma 3.3. Furthermore we have
\begin{equation}\label{eq:3.12}
\mathbb{E}\left[\sup_{0\leq t',t''\leq t\leq T}|D_{t',t''}^{2}\sigma_{t}|^{q'}\right]<\infty.
\end{equation}
The same holds for the process $v_{t}=\sqrt{\sigma_{t}}$, but this time the second Malliavin Derivative is given by
\begin{equation}\label{eq:3.13}
D_{t',t''}^{2}v_{t}=D_{t'}v_{t}\times\left[\int_{t'}^{t}\left(\frac{k\theta}{2}-\frac{\xi^{2}}{8}\right)\frac{1}{\sigma_{s}^{2}}D_{t''}\sigma_{s}ds\right].
\end{equation}
\end{lem}

Finally, we need a lemma that connects the existence of a regular density to the existence of a regular non-vanishing Malliavin Derivative. There are many results of this kind in the literature (many of them can be found in \cite{Nualart}), but in our case we need the following.

\begin{lem}\label{lem:3.5}
Let $F$ be a random variable in the space $\mathbb{D}^{1,2}\cap L^{p}\left(\Omega\right)$
for all $p>1$. Assume that for some process $\{u_{t}:\,0\leq t\leq T\}$
of $L^{2}$-integrable paths, we have $\left\langle u_{.},\,D_{.}F\right\rangle _{L^{2}}\neq0$
almost surely and also $\frac{u_{t}}{\left\langle u_{.},\,D_{.}F\right\rangle _{L^{2}}}\in\mathbb{D}^{1,q'}
\left(L^{2}\left(\left[0,\,T\right]\right)\right)$ for some $q'>1$. Then $F$ has a continuous density $p(\cdot)$ for
which it holds
\[
\sup_{y\geq0}\left(y^{\alpha}p(y)\right)\leq C\mathbb{E}^{\frac{q'-1}{q'}}\left[F^{\frac{\alpha q'}{q'-1}}\right]\mathbb{E}^{\frac{1}{q'}}\left[\left\Vert D_{.}\frac{u_{.}}{\left\langle u_{.},\,D_{.}F\right\rangle _{L^{2}}}\right\Vert _{L^{2}\left(\left[0,\,T\right]^{2}\right)}^{q'}\right]<\infty,
\]
 for some $C>0$ and for all $\alpha\geq0$ for which $\mathbb{E}\left[F^{\frac{\alpha q'}{q'-1}}\right]$
is finite.
\end{lem}

The proof of the above lemma has also been put in the Appendix, since it is almost identical to that of Proposition~2.1.1 in \cite{Nualart} (page 78), except that in the end we need to recall Meyer's inequality in order to obtain the estimate for the supremum. We are now ready to prove the main result of this section.

\begin{proof}[Proof of Theorem 3.1]
Lemma~\ref{lem:3.3} implies that for almost all $\omega\in\Omega$,
$\sigma_{t}\in\mathbb{D}^{1,2}$ with respect to $B_{\cdot}^{1}$ under the
probability measure $\mathbb{P}\left(\cdot\,|\,B_{\cdot}^{0},\,\mathcal{G}\right)$.
We would like to apply Lemma~\ref{lem:3.5} on $\sigma_{t}$ for an appropriate
process $\{u_{t'}:\,0\leq t'\leq t\}$. Let $u_{.}$ be the unique
pathwise solution to the linear integral equation
\[
u_{t'}=\int_{0}^{t'}u_{s}e^{-\int_{s}^{t}\left[\left(\frac{k\theta}{2}-\frac{\xi^{2}}{8}\right)+\frac{k}{2}\right]\frac{1}{\sigma_{s'}}ds'}\sqrt{\sigma_{t}}ds,\:\:\forall\,t'<t.
\]
Then, $u_{t'}=u_0e^{\int_{0}^{t'}e^{-\int_{s}^{t}\left[\left(\frac{k\theta}{2}-\frac{\xi^{2}}{8}\right)\frac{1}{\sigma_{s'}}+\frac{k}{2}\right]ds'}\sqrt{\sigma_{t}}ds}$
for any $t'\leq t$, which is almost surely a differentiable and strictly
increasing function on $[0,\,t],$ always bounded by $u_0e^{t\sqrt{\sigma_{t}}}>0$.
Then it is easy to check that
\[
\frac{u_{t'}}{\left\langle u_{.},\,D_{.}\sigma_{t}\right\rangle _{L^{2}\left[0,\,t\right]}}=\frac{u_{t'}}{u_{t}}\leq1,
\]
for any $t'<t$. Thus, we have that
\begin{equation}\label{eq:3.22}
U_{.}=\frac{u_{.}}{\left\langle u_{.},\,D_{.}\sigma_{t}\right\rangle _{L^{2}\left[0,\,t\right]}}\in L^{q'}\left(\Omega;\,
L^{2}\left(\left[0,\,t\right]\right)\right).
\end{equation}

Next, we want to show that $U_{t'}$ is Malliavin differentiable
and compute its derivative for any $0<t'<t$. By Lemma~\ref{lem:3.3} we have
that $D_{.}\sigma_{t}$ lies in the space
\[
\mathbb{D}^{1,q'}\left(L^{2}\left[0,\,t\right]\right)\subset\mathbb{D}^{1,q'}\left(L^{2}\left[t',\,t\right]\right),
\]
for any $t'<t$ and any $1<q'\leq\frac{4k\theta}{3\xi^{2}}$, which
implies that
\begin{eqnarray*}
D_{.}\int_{t'}^{t}D_{s}\sigma_{t}ds &=& \int_{t'}^{t}D_{s,.}^{2}\sigma_{t}ds
\\
&\leq & \int_{0}^{t}D_{s,.}^{2}\sigma_{t}ds\in L^{q'}\left(\Omega;\,L^{2}\left[0,\,t\right]\right),
\end{eqnarray*}
for all $0\leq t'\leq t$ and any $1<q'\leq\frac{4k\theta}{3\xi^{2}}$.
Since the LHS of the above is positive (follows easily from Lemma~\ref{lem:3.4} and our assumptions for the coefficients), the above implies also that $\int_{t'}^{t}D_{s}\sigma_{t}ds\in\mathbb{D}^{1,q'}$
for any $t'<t$ and that
\begin{equation}\label{eq:3.23}
\int_{.}^{t}D_{s}\sigma_{t}ds\in\mathbb{D}^{1,q'}\left(L^{2}\left[0,\,t\right]\right).
\end{equation}
Consider now a smooth function $F$ satisfying
\[
F(x)=\begin{cases}
e^{-x},\; & x\geq0,\\
0, & x<-1,
\end{cases}
\]
and which has bounded derivatives. Then we can easily check that $U_{t'}=F\left(\int_{t'}^{t}D_{s}\sigma_{t}ds\right)$,
and by the standard Malliavin Chain rule we obtain $U_{t'}\in\mathbb{D}^{1,q'}$
for any $1<q'<\frac{4k\theta}{3\xi^{2}}$, with
\begin{equation}\label{eq:3.24}
D_{.}U_{t'}=F'\left(\int_{t'}^{t}D_{s}\sigma_{t}ds\right)\int_{t'}^{t}D_{s,.}^{2}\sigma_{t}ds,
\end{equation}
 for all $0\leq t'\leq t$. Finally, from \eqref{eq:3.23} and \eqref{eq:3.24}
we have that $D_{.}U_{.}$ belongs to the space $L^{q'}\left(\Omega;\,L^{2}\left(\left[0,\,t\right]^{2}\right)\right)$
and thus, by \eqref{eq:3.22} we deduce that $U_{.}\in\mathbb{D}^{1,q'}\left(L^{2}\left[0,\,t\right]\right)$.

Recall now that $\sigma_{t}^{p}$
has a finite expectation under $\mathbb{P}\left(\cdot\,|\,B_{\cdot}^{0},\,\mathcal{G}\right)$,
for any exponent $p>0$ and any $t\leq T$, for all $\omega\in\Omega_{2}$
with $\mathbb{P}\left(\Omega_{2}\right)=1$ (this follows easily from Lemma~\ref{lem:3.2} and the law of total expectation). Thus, by Lemma~\ref{lem:3.5}, $\sigma_{t}$
possesses for all $\omega\in\Omega_{2}$ a continuous density $p_{t}(\cdot\,|\,B_{\cdot}^{0},\,\mathcal{G})$
under $\mathbb{P}\left(\cdot\,|\,B_{\cdot}^{0},\,\mathcal{G}\right)$,
such that for all $t\geq0$
\begin{eqnarray}
\sup_{y\geq0} y^{\alpha}p_{t}(y|B_{\cdot}^{0},\mathcal{G})& \leq & M_{B_{\cdot}^{0},\alpha,t}\nonumber \\
&:=&C\mathbb{E}^{\frac{q'-1}{q'}}\left[\left(\sigma_{t}\right)^{\frac{\alpha q'}{q'-1}}\,|B_{\cdot}^{0},\mathcal{G}\right] \nonumber \\
& & \; \times \mathbb{E}^{\frac{1}{q'}}\left[\left(\left\Vert F'\left(\int_{.}^{t}D_{s}\sigma_{t}ds\right)\int_{.}^{t}D_{s,.}^{2}\sigma_{t}ds\right\Vert _{L^{2}\left(\left[0,\,t\right]^{2}\right)}\right)^{q'}|B_{\cdot}^{0},\mathcal{G}\right] 
\nonumber \\
&\leq & C'\mathbb{E}^{\frac{q'-1}{q'}}\left[\sup_{0\leq t\leq T}\left(\left(\sigma_{t}\right)^{\frac{\alpha q'}{q'-1}}\right)|B_{\cdot}^{0},\mathcal{G}\right] \nonumber \\
& & \; \times \mathbb{E}^{\frac{1}{q'}}\left[\left(\left\Vert \int_{.}^{t}D_{s,.}^{2}\sigma_{t}ds\right\Vert _{L^{2}\left(\left[0,t\right]^{2}\right)}\right)^{q'}|B_{\cdot}^{0},\mathcal{G}\right], \label{eq:3.25}
\end{eqnarray}
for any $\alpha\geq0$.

It is not hard to see that the constant $C'>0$ does not
depend on the fixed path $B_{\cdot}^{0}$ of the Market factor, since
it depends on the maximum of the derivative of $F$ and the universal
constant of Proposition~1.5.4 in \cite{Nualart} (changing the measure here
is the same as changing the process $\sigma_{t}$ by changing its
Market factor, under the law of the idiosyncratic factor). Therefore,
for any $1<q<q'$, by Holder's inequality, the estimate \eqref{eq:3.25}
and the law of total expectation, we have
\begin{eqnarray}
\mathbb{E}\left[\left(\sup_{0<t\le T}M_{B_{\cdot}^{0},\alpha,t}\right)^{q}\right] &\leq & C^{q}\mathbb{E}^{\frac{q'-q}{q'}}\left[\mathbb{E}^{\frac{(q'-1)q}{q'-q}}\left[\sup_{0\leq t\leq T}\left(\left(\sigma_{t}\right)^{\frac{\alpha q'}{q'-1}}\right)\,|\,B_{\cdot}^{0},\,\mathcal{G}\right]\right] \nonumber \\
& & \qquad \times \mathbb{E}^{\frac{q}{q'}}\left[\sup_{0<t\leq T}\left(\int_{0}^{T}\int_{0}^{T}\left(\int_{t'}^{T}D_{s,t''}^{2}\sigma_{t}ds\right)^{2}dt'dt''\right)^{\frac{q}{2}}\right]. \nonumber \\
\label{eq:3.26}
\end{eqnarray}

Since $\frac{(q'-1)q}{q'-q}>1$, by applying Holder's inequality
and the law of total expectation once more, we find that the first
factor on the RHS of \eqref{eq:3.26} is bounded by
\[
C^{q}\mathbb{E}^{\frac{q'-q}{q'}}\left[\sup_{0\leq t\leq T}\left(\left(\sigma_{t}\right)^{\frac{\alpha qq'}{q'-q}}\right)\right]
\]
which is finite by Lemma~\ref{lem:3.2}. On the other hand, the
second factor on the RHS of \eqref{eq:3.26} is bounded by
\begin{eqnarray*}
& & \mathbb{E}^{\frac{q}{q'}}\left[\left(\int_{0}^{T}\int_{0}^{T}\int_{t'}^{T}\left(\sup_{0<t\leq T}D_{s,t''}^{2}\sigma_{t}\right)^{2}dsdt'dt''\right)^{\frac{q'}{2}}\right] \\
& & \qquad\qquad \leq \mathbb{E}^{\frac{q}{q'}}\left(\left(\int_{0}^{T}\int_{0}^{T}\int_{0}^{T}\left(\sup_{0<s,t''\leq t\leq T}D_{s,t''}^{2}\sigma_{t}\right)^{2}dsdt'dt''\right)^{\frac{q'}{2}}\right) \\
& & \qquad\qquad \leq T^{\frac{3q}{2}}\mathbb{E}^{\frac{q}{q'}}\left(\left(\sup_{0<s,t''\leq t\leq T}D_{s,t''}^{2}\sigma_{t}\right)^{q'}\right)
\end{eqnarray*}
which is finite by Lemma~\ref{lem:3.4} for any $T>0$, so the desired result follows.
\end{proof}

\section{Existence of a regular two-dimensional density}

In this section, we combine the results we have obtained in the previous
section for the volatility process with the regularity results
we have for the constant volatility model of \cite{BHHJR}, in order to
obtain a regular density for the probabilistic solution of the SPDE
obtained in Section~2, when the value of $C_1$ is given.
First, for any Hilbert space $H$, we denote by $L^{2}\left(\Omega\times\left[0,\,T\right];\,H\right)$
the space of $H$- valued stochastic processes, which are adapted
to the Brownian path $(W_{\cdot}^{0},\,B_{\cdot}^{0})$. For our purpose we will need the following useful Theorem which extends the results of \cite{BHHJR} to the case where the volatility path is non-constant.

\begin{thm}\label{thm:4.1}
Let $\{X_{t}:\,t\geq0\}$ satisfy the stopped SDE 
\begin{equation}\label{eq:4.1}
\begin{cases}
dX_{t}=\left(r-\frac{\sigma_{t}}{2}\right)dt+\sqrt{\sigma_{t}}\sqrt{1-\rho_{1}^{2}}dW_{t}^{1} 
+\sqrt{\sigma_{t}}\rho_{1}dW_{t}^{0}, & 0\leq t\leq\tau,\\
\\
X_{t}=0, & t\ge\tau,
\end{cases}
\end{equation}
for $\tau=\inf\{t\geq0:\,X_{t}=0\}$, under the initial
condition $X_{0}=x_{0}$, where $x_{0}$ is a continuous random variable
with an $L^{2}$ density $u_{0}=u_{0}(\cdot\,|\,\mathcal{G})$ given
$\mathcal{G}\subset\sigma(x_{0})$, $W_{\cdot}^{0}$ and $W_{\cdot}^{1}$ are pairwise independent
standard Brownian motions which are also independent of $x_{0}$,
$r>0$ and $\rho_{1}\in(-1,\,1)$ are given constants, and $\{\sigma_{t}:\,t\geq0\}$
is just a deterministic path which is continuous and positive. Let
$\{V_{t}:\,t\geq0\}$ be the measure-valued process given by
\[
V_{t}(A)=\mathbb{P}\left(X_{t}\in A,\tau\geq t\,|\,W_{\cdot}^{0},\,\mathcal{G}\right)
\]
for any Borel set $A\subset\left(0,\,\infty\right)$. Then almost
surely, the following are true for all $T>0$;
\begin{enumerate}
\item  $V_{t}$ possesses a density $u(t, \, \cdot) = u\left(t,\,\cdot\,;\,W_{\cdot}^{0},\,\mathcal{G}\right)$
for all $0\leq t\leq T$, which is the unique solution to the SPDE
\begin{equation}\label{eq:4.2}
du(t,\,x)=-\left(r-\frac{\sigma_{t}}{2}\right)u_{x}(t,\,x)dt+\frac{\sigma_{t}}{2}
u_{xx}(t,\,x)dt-\sqrt{\sigma_{t}}\rho_{1}u_{x}(t,\,x)dW_{t}^{0}
\end{equation}
in $L^{2}\left(\Omega\times\left[0,\,T\right];\,H_{0}^{1}\left(\mathbb{R}^{+}\right)\right)$
under the initial condition $u(0,\cdot)=u_{0}$, where $u_{0}$ is
the density of $x_{0}$ given $\mathcal{G}$. 

\item For all $0\leq t\leq T$, the following identity holds
\begin{equation}\label{eq:4.3}
||u(t,\,\cdot)||_{L^{2}(\mathbb{R}^{+})}^{2}+\left(1-\rho_{1}^{2}\right)\int_{0}^{t}\sigma_{s}
||u_{x}(s,\,\cdot)||_{L^{2}(\mathbb{R}^{+})}^{2}ds=||u_{0}||_{L^{2}(\mathbb{R}^{+})}^{2}.
\end{equation}
\end{enumerate}
\end{thm}

To prove the above theorem, we need the following convergence result for a sequence of stopped Ito processes, when the corresponding sequence of volatility paths decreases pointwise to some continuous and positive path. 

\begin{lem}{\label{convergence}}
Let $\{\sigma_{t}^{0}:\,0\leq t\leq T\}$ be a continuous and strictly positive path, which is approximated from above by a pointwise decreasing sequence $\{\{\sigma_{t}^{m}:\,0\leq t\leq T\}\}_{m\in\mathbb{N}}$ of positive and bounded paths. For any $m \in \mathbb{N}\cup \{0\}$, denote by $X_{.}^{m}$ the stopped Ito process given by
\begin{equation}\label{eq:4.4}
\begin{cases}
dX_{t}^{m}=\left(r-\frac{\sigma_{t}^{m}}{2}\right)dt+\sqrt{\sigma_{t}^{m}}dW_{t}, & 0\leq t\leq\tau^{m},\\
\\
X_{t}^{m}=0, & t > \tau^{m},
\end{cases}
\end{equation}
where $\tau^{m}=\inf\{t\geq0:\,X_{t}^{m}=0\}$ and $W_{\cdot}$ is a standard Brownian Motion, with the initial condition
\[
X_{0}^{m}=\max\left\{x_{0}-l_{m},\,\frac{x_{0}}{2}\right\},
\]
for $x_0 \geq 0$ and $l_{m}=\left(\left\Vert \sqrt{\sigma_{.}^{m}}-\sqrt{\sigma_{.}^{0}}\right\Vert _{L^{2}\left[0,\,T\right]}\right)^{1/2}$. Then, for a sequence $\{m_{k}:\,k\in\mathbb{N}\}\subset\mathbb{N}$,
we have almost surely: $X_{t}^{m_{k}}\rightarrow X_{t}^{0}$ uniformly
on any compact interval $\left[0,\,T\right]$.
\end{lem}

The proof of the above lemma is quite technical and can be found in the Appendix

\begin{proof}[Proof of Theorem 4.1]
Without loss of generality we can assume $\rho_1 \in [0, 1)$. We will prove first that 2. holds 
for the unique solution $u$ of \eqref{eq:4.2} in the space
$L^{2}\left(\Omega\times\left[0,\,T\right];\,H_{0}^{1}\left(\mathbb{R}^{+} \right)\right)$. Since the existence and uniqueness of this solution follows from the main results of \cite{KR81} (pages 18-20), for any fixed volatility path and any square integrable initial density, we do not need to use 1. The estimate \eqref{eq:4.3} for that $u$ is also obtained without using 1., which means that we can use 2. for proving 1. next. Indeed, applying Ito's formula for the $L^{2}$ norm (Theorem~3.2 in \cite{KR81}
for the triple $H_{0}^{1}\subset L^{2}\subset H^{-1}$) on \eqref{eq:4.2} and observing that $\int_{0}^{\infty}u_{xx}u=-\int_{0}^{\infty}u_{x}^{2}$
(by the definition of the distributional second derivative, since
$u\in H_{0}^{1}$ can be approximated by smooth and compactly supported
functions in that space), we obtain
\[
\left\Vert u(t,\,\cdot,\,W_{\cdot}^{0},\,\mathcal{G})\right \Vert_{L^{2}\left(\mathbb{R}^{+}\right)}^{2} 
+\left(1-\rho_{1}^{2}\right)\int_{0}^{t}\sigma_{s}\left\Vert u_{x}(s,\,\cdot,\,W_{\cdot}^{0},\,\mathcal{G})
\right\Vert _{L^{2}\left(\mathbb{R}^{+}\right)}^{2}ds=\left\Vert u_{0}
\right\Vert _{L^{2}\left(\mathbb{R}^{+}\right)}^{2}
\]
for all $0\leq t\leq T$, for all $\omega\in\Omega^{T}$ with $\mathbb{P}\left(\Omega^{T}\right)=1$.
The desired result follows since $\mathbb{P}\left({\displaystyle \Omega^{\infty}}\right)=1$
for $\Omega^{\infty}={\displaystyle \cap_{N\in\mathbb{N}}\Omega^{N}}$. 

We proceed now to the proof of 1. which will be divided into
3 steps. In what follows, we will be working on the set ${\displaystyle \Omega^{\infty}}$
which is defined above and which is a set of full probability.

\flushleft{1. \emph{The constant volatility case:}}

We assume first that the path $\sigma_{t}$
is constant in $t\geq0$, i.e $\sigma_{t}=\sigma^{2}\;\forall\,t\geq0$.
In that case, $V_{t}$ is the limit empirical process studied in \cite{BHHJR}
and \cite{HL16} (without the compactly supported initial data restriction),
scaled by $\sigma>0$, so it does have a density $u(t, \, \cdot) = u(t,\,\cdot,\,W_{\cdot}^{0})$
which is the unique solution of the SPDE
\[
du(t,\,x)=-\left(r-\frac{\sigma^{2}}{2}\right)u_{x}(t,\,x)dt+\frac{\sigma^{2}}{2}
u_{xx}(t,\,x)dt-\sigma\rho_{1}u_{x}(t,\,x)dW_{t}^{0}
\]
in $L^{2}\left(\Omega\times\left[0,\,T\right];\,H_{0}^{1}\left(\mathbb{R}^{+}\right)\right)$,
under the initial condition $u(0,\cdot)=u_{0}$, which is actually
\eqref{eq:4.2}. It holds also that $xu_{xx}(t,\,x)$
is square integrable.

\flushleft{2. \emph{The piecewise constant volatility case:}}

We assume now that the path
$\sigma_{\cdot}$ is piecewise constant in $0\leq t\leq T$, i.e $\sigma_{t}=\sigma_{i}^{2}>0\;
\forall\,t\in\left[t_{i},\,t_{i+1}\right]$
(almost everywhere), where $0=t_{0}<t_{1}<...<t_{n}=T$ is a partition
of $\left[0,\,T\right]$. We shall prove that 1. holds for $T$ replaced
by $t_{i+1}$ for all $i\leq n-1$, by using induction on $i$. The base
case ($i=0$) follows directly from Step 1 for $T=t_{1}$. Assume
that our desired result holds for $i=j<n-1$ and thus, we have obtained the desired density $u(t, \, \cdot)$ for $0 \leq t \leq t_{j+1}$. For $i=j+1$ now, by
starting our (Markovian) processes at $t=t_{j+1}$ and using Step 1 again,
we have that $V_{t}(A)=\mathbb{P}\left(X_{t}\in A,\tau\geq t\,|\,W_{\cdot}^{0},\,\mathcal{G}\right)$
has a density $\tilde{u}(t, \, \cdot) = \tilde{u}\left(t,\,\cdot\,;\,W_{\cdot}^{0},\,\mathcal{G}\right)$
for all $t_{j+1}\leq t\leq t_{j+2}$, which is the unique solution
of the SPDE
\[
d\tilde{u}(t,\,x)=-\left(r-\frac{\sigma_{j+1}^{2}}{2}\right)\tilde{u}_{x}(t,\,x)dt+\frac{\sigma_{j+1}^{2}}{2}
\tilde{u}_{xx}(t,\,x)dt-\sigma_{j+1}\rho_{1}\tilde{u}_{x}(t,x)dW_{t}^{0}
\]
in $L^{2}\left(\Omega\times\left[t_{j+1},\,t_{j+2}\right];\,H_{0}^{1}\left(\mathbb{R}^{+}\right)\right)$,
i.e an equivalent to \eqref{eq:4.2} in $\left[t_{j+1},\,t_{j+2}\right]$, under the initial condition
\[
\tilde{u}(t_{j+1},\,\cdot) = u\left(t_{j+1},\,\cdot \right)
\]
Therefore, by defining $u(t, \, \cdot) = \tilde{u}(t, \, \cdot)$ in $\left(t_{j+1},\,t_{j+2}\right]$, we see that $u(t, \, \cdot)$ is both the density of $V_{t}\left(\cdot\right)$ and the unique solution
of \eqref{eq:4.2} in $L^{2}\left(\Omega\times\left[0,\,t_{j+2}\right];
\,H_{0}^{1}\left(\mathbb{R}^{+}\right)\right)$
under the initial condition $u(0,\cdot)=u_{0}$. Our induction is now complete.

{\flushleft{3. \emph{The general case:}} }

We assume finally that the path
$\sigma_{\cdot}$ is any continuous and positive
path. Recall Lemma~\ref{convergence} for $\sigma_{\cdot}^{0} = \sigma_{\cdot}$ and for a pointwise decreasing sequence of piecewise constant paths $\{\sigma_{\cdot}^{m}\}_{m\in\mathbb{N}}$ approximating $\sigma_{\cdot}$ from above. This gives a sequence of drifted Brownian Motions $X_{\cdot}^{m_{k}}$ converging $\mathbb{P}$-almost surely to $X_{\cdot}$, with corresponding volatility paths $\{\sigma_{\cdot}^{m_k}\}_{m\in\mathbb{N}}$. If we denote by $u^{k}(t,\,\cdot;\,W_{\cdot}^{0},\,\mathcal{G})$
the density of $X_{t}^{m_{k}}$ for any $t \geq 0$, then Step 2 shows that $u^{k}$ is
the unique solution of \eqref{eq:4.2} with $\sigma_{\cdot}$ replaced by $\sigma_{\cdot}^{m_{k}}$ and with an initial density given by
\[
u_{0}^{k}(x)=u_{0}\left(\min\{2x,\,x+\ell_{m_{k}}\}\right).
\]

Thus, by 2., which has been proven for the unique regular enough solution of \eqref{eq:4.2} (independently from 1. which we are now proving), we have that
\[
\left\Vert u^{k}(t,\,\cdot;\,W_{\cdot}^{0},\,\mathcal{G})\right\Vert _{L^{2}\left(\mathbb{R}^{+}\right)}^{2}+\left(1-\rho_{1}^{2}\right)\int_{0}^{t}\sigma_{s}^{m_{k}}\left\Vert u_{x}^{k}(s,\,\cdot;\,W_{\cdot}^{0},\,\mathcal{G})\right\Vert _{L^{2}\left(\mathbb{R}^{+}\right)}^{2}ds=\left\Vert u_{0}^{k}\right\Vert _{L^{2}\left(\mathbb{R}^{+}\right)}^{2},
\]
where we can take expectations to obtain
\begin{equation}\label{eq:4.8}
\left\Vert u^{k}(t,\,\cdot;\,W_{\cdot}^{0},\,\mathcal{G})\right\Vert _{L^{2}}^{2}+\left(1-\rho_{1}^{2}\right)\int_{0}^{t}\sigma_{s}^{m_{k}}\left\Vert u_{x}^{k}(s,\,\cdot;\,W_{\cdot}^{0},\,\mathcal{G})\right\Vert _{L^{2}}^{2}ds=\left\Vert u_{0}^{k}\right\Vert _{L^{2}}^{2}
\end{equation}
for all $0\leq t\leq T$, where $L^{2}$ stands for $L^{2}\left(\Omega \times \mathbb{R}^{+}\right)$. Moreover, by the choice of the approximating
sequence, we can see that all the volatility paths are bounded below
by $m={\displaystyle \min_{0\leq s\leq T}\{\sigma_{s}\}}$, while
the sequence of the norms of the $u_{0}^{k}$s is bounded. Thus, we
can easily obtain from \eqref{eq:4.8} that the sequence of $L^{2}\left(\Omega\times\left[0,\,T\right];\,H_{0}^{1}\left(\mathbb{R}^{+}\right)\right)$
norms of the $u^{k}$s is bounded. Hence, there exists a $u\in L^{2}\left(\Omega\times\left[0,\,T\right];\,H_{0}^{1}\left(\mathbb{R}^{+}\right)\right)$
such that $u^{n_{k}}\rightarrow u$ weakly in that space, for a sequence
$\left\{ n_{k}:\,k\in\mathbb{N}\right\} $ of positive integers. Given
now the path $W_{\cdot}^{0}$ and given $\mathcal{G}$, since convergence
almost surely implies convergence in distribution, for any smooth
function $f$ defined on $\left[0,\,T\right]$, any open set $A\subset\mathbb{R}^{+}$
and any $B\in\sigma\left(W_{\cdot}^{0},\,\mathcal{G}\right)$ we have
\begin{eqnarray*}
\int_{0}^{T}\mathbb{E}\left[\mathbb{P}\left(X_{t}\in A\,|\,W_{\cdot}^{0},\,\mathcal{G}\right)\mathbb{I}_{B}\right]f(t)dt
&=& \int_{0}^{T}\mathbb{E}\left[\lim_{k\rightarrow\infty}\mathbb{P}\left(X_{t}^{m_{n_k}}\in A\,|\,W_{\cdot}^{0},\,\mathcal{G}\right)\mathbb{I}_{B}\right]f(t)dt \\
&=& \lim_{k\rightarrow\infty}\int_{0}^{T}\mathbb{E}\left[\mathbb{P}\left(X_{t}^{m_{n_k}}\in A\,|\,W_{\cdot}^{0},\,\mathcal{G}\right)\mathbb{I}_{B}\right]f(t)dt \\
& =& \lim_{k\rightarrow\infty}\int_{0}^{T}\mathbb{E}\left[\mathbb{I}_{B}\int_{A}u^{n_k}(t,\,x,\,W_{\cdot}^{0},\,\mathcal{G})dx\right]f(t)dt
\\
&=&\int_{0}^{T}\mathbb{E}\left[\mathbb{I}_{B}\int_{A}u(t,\,x,\,W_{\cdot}^{0},\,\mathcal{G})dx\right]f(t)dt,
\end{eqnarray*}
which implies that $u\in L^{2}\left(\Omega\times[0,\,T];\,H_{0}^{1}\left(\mathbb{R}^{+}\right)\right)$
is the density process of $X_{t}$ given $W_{\cdot}^{0}$ and $\mathcal{G}$.
By applying Ito's formula on any smooth function computed at $X_{t}$
and taking expectations given $W_{\cdot}^{0}$ and $\mathcal{G}$,
we find that $u$ is also a solution of \eqref{eq:4.2} which also satisfies
the initial condition $u(0,\cdot)=u_{0}$. Thus, $u$ is the unique
solution of \eqref{eq:4.2} in $L^{2}\left(\Omega\times\left[0,\,T\right];\,H_{0}^{1}\left(\mathbb{R}^{+}\right)\right)$,
under the initial condition $u(0,\cdot)=u_{0}$. The proof is now
complete.
\end{proof}

We are now ready to prove the first main result of this paper, that is the existence of a regular density for $v_{t, C_1}$.
\begin{thm}\label{thm:4.2}
Suppose that the $\mathcal{F}_{0}$-measurable random vector $C_{1}$ is independent from $\left(X_{0}^{1},\,\sigma_{0}^{1}\right)$, and for some $q>1$ we have $\mathbb{P}$- almost surely $\frac{4k_{1}\theta_{1}}{3\xi_{1}^{2}} > q$ and $\rho_{1,i},\,\rho_{2,i}\in (-1,\,1)$. Suppose also
that given $\mathcal{G}$, $X^{1}_{0}$ has an $L^{2}$-integrable density
$u_{0}(\cdot|\mathcal{G})$ in $\mathbb{R}^{+}$ such that $\mathbb{E}\left[||u_{0}||_{L^{2}\left(\mathbb{R}^{+}\right)}^{2q'}\right]
<\infty$ for $\frac{1}{q}+\frac{1}{q'}=1$. Suppose finally that the initial value assumptions of Theorem 3.1 are satisfied for $\sigma_{0} = \sigma^{1}_{0}$ and $\left(k,\,\theta,\,\xi,\,r,\,\rho_{1},\,\rho_{2}\right) = \left(k_{1},\,\theta_{1},\,\xi_{1},\,r_{1},\,\rho_{1,1},\,\rho_{2,1}\right)$, for any possible value of $C_1 = \left(k_{1},\,\theta_{1},\,\xi_{1},\,r_{1},\,\rho_{1,1},\,\rho_{2,1}\right)$.
Then, for any value of $C_{1}$, the measure-valued stochastic process $v_{t,C_{1}}$ obtained in Section~2 has a two-dimensional density $u_{C_{1}}(t,\,\cdot,\,W_{\cdot}^{0},\,B_{\cdot}^{0},\,\mathcal{G})$
belonging to the spaces
\[
L_{\alpha, \, C_1}=L^{\infty}\left(\left[0,\,T\right];\,L^{2}\left(\left(\Omega, \, \mathcal{F}, \, \mathbb{P}\left(\cdot \, |\,C_1\,\right)\right);\,L_{y^{\alpha}}^{2}\left(\mathbb{R}^{+}\times\mathbb{R}^{+}\right)\right)\right)
\]
and
\[
H_{\alpha, \, C_1}=L^{2}\left(\left(\Omega, \, \mathcal{F}, \, \mathbb{P}\left(\cdot \, |\,C_1\,\right)\right)\times\left[0,\,T\right];\,H_{0}^{1}\left(\mathbb{R}^{+}\right)\times L_{h_{1}(y)y^{\alpha}}^{2}\left(\mathbb{R}^{+}\right)\right)
\]
for any $\alpha\geq0$, where we write $L_{g(y)}^{2}$ for the weighted
$L^{2}$ space with the weight function $\left\{ g(y):y\geq0\right\} $,
and $h_{1}(y) := \min\left\{ h(y),1\right\} \forall y\ge0$.
\end{thm}

\begin{proof}
Let $f$ be a smooth function, compactly supported
in $\mathbb{R}^{2}$, such that $f$ vanishes on the $y$ - axis. Then by Theorem~\ref{thm:3.1} we have
\begin{eqnarray}
v_{t,C_{1}}\left(f \right)&=& \mathbb{E}\left[f\left(X_{t}^{1}, \, \sigma_{t}^{1}\right)\mathbb{I}_{\{T_{1}\geq t\}}\,|W_{\cdot}^{0},B_{\cdot}^{0},C_{1},\mathcal{G}\right] \nonumber \\
&=&\mathbb{E}\left[\mathbb{E}\left[f\left(X_{t}^{1}, \, \sigma_{t}^{1}\right)\mathbb{I}_{\{T_{1}\geq t\}}\,|W_{\cdot}^{0},\sigma_{t}^{1},B_{\cdot}^{0},C_{1},\mathcal{G}\right]\,|W_{\cdot}^{0},B_{\cdot}^{0},C_{1},\mathcal{G}\right]
\nonumber \\
&=&\int_{\mathbb{R}^+}\mathbb{E}\left[f\left(X_{t}^{1}, \, y\right)\mathbb{I}_{\{T_{1}\geq t\}}\,|W_{\cdot}^{0},\sigma_{t}^{1}=y,B_{\cdot}^{0},C_{1},\mathcal{G}\right]p_{t}\left(y|B_{\cdot}^{0},\mathcal{G}\right)dy. \nonumber \\ 
\label{eq:4.9}
\end{eqnarray}

Next we have
\begin{eqnarray}
& & \mathbb{E}\left[f\left(X_{t}^{1}, \, y\right)\mathbb{I}_{\{T_{1}\geq t\}}|W_{\cdot}^{0},
\sigma_{t}^{1}=y,B_{\cdot}^{0},C_{1},\mathcal{G}\right] \nonumber \\
& & \qquad =\mathbb{E}\left[\mathbb{E}\left[f\left(X_{t}^{1}, \, y\right)\mathbb{I}_{\{T_{1}\geq t\}}|
W_{\cdot}^{0},\sigma_{.},C_{1},\mathcal{G}\right]|W_{\cdot}^{0},\sigma_{t}^{1}=y,B_{\cdot}^{0},C_{1},
\mathcal{G}\right] \nonumber \\
& & \qquad =\mathbb{E}\left[\int_{\mathbb{R}^+}f(x, \, y)u\left(t,x,W_{\cdot}^{0},\mathcal{G},C_{1}, 
h\left(\sigma_{.}\right)\right)dx|W_{\cdot}^{0},\sigma_{t}^{1}=y,B_{\cdot}^{0},C_{1},\mathcal{G}\right] 
\nonumber \\
& & \qquad =\int_{\mathbb{R}^+}f(x, \, y)\mathbb{E}\left[u\left(t,x,W_{\cdot}^{0},\mathcal{G},C_{1}, 
h\left(\sigma_{.}\right)\right) |W_{\cdot}^{0},\sigma_{t}^{1}=y,B_{\cdot}^{0},C_{1},\mathcal{G}\right]dx, \nonumber \\
\label{eq:4.10}
\end{eqnarray}
where $u\left(t,x,W_{\cdot}^0,C_{1},\mathcal{G},h\left(\sigma_{.}\right)\right)$
is the $L^{2}\left(\Omega\times[0,\,T];\,H_{0}^{1}\left(\mathbb{R}^{+}\right)\right)$
density given by Theorem~\ref{thm:4.1}, when the coefficient vector $C_{1}$
is given and the volatility path is $h\left(\sigma_{.}\right)$. By \eqref{eq:4.9} and \eqref{eq:4.10} 
we obtain that the desired density exists and is given by
\[
u_{C_{1}}\left(t,x,y,W_{\cdot}^{0},B_{\cdot}^{0},\mathcal{G}\right)=p_{t}\left(y|B_{\cdot}^{0},\mathcal{G}\right)
\mathbb{E}\left[u\left(t,x,W_{\cdot}^{0},\mathcal{G},C_{1},h\left(\sigma_{.}\right)\right)\,|W_{\cdot}^{0},\sigma_{t}^{1}=
y,B_{\cdot}^{0},C_{1},\mathcal{G}\right]
\]
which is obviously supported in $\mathbb{R}^{+}\times\mathbb{R}^{+}$.
Now, recall the estimate from Theorem~\ref{thm:3.1} and the Cauchy-Schwarz inequality
to obtain
\begin{eqnarray*}
& & y^{\alpha}\left(\frac{\partial u_{C_{1}}}{\partial x}\left(t,\,x,\,y,\,W_{\cdot}^{0},\,B_{\cdot}^{0},\,\mathcal{G}\right)\right)^{2} \\
& & \qquad \leq  M_{B_{\cdot}^{0},\alpha} p_{t}\left(y\,|\,B_{\cdot}^{0},\,\mathcal{G}\right) \\
& & \qquad \qquad \times\mathbb{E}^{2}\left[u_{x}\left(t,\,x,\,W_{\cdot}^{0},\,\mathcal{G},\,C_{1},\,h\left(\sigma_{.}\right)\right)\,|\,W_{\cdot}^{0},\,\sigma_{t}^{1}=y,\,B_{\cdot}^{0},\,C_{1},\,\mathcal{G}\right] \\
& & \qquad \leq  M_{B_{\cdot}^{0},\,\alpha} p_{t}\left(y\,|\,B_{\cdot}^{0},\,\mathcal{G}\right)\mathbb{E}\left[u_{x}^{2}\left(t,\,x,\,W_{\cdot}^{0},\,\mathcal{G},\,C_{1},\,h\left(\sigma_{.}\right)\right)\,|\,W_{\cdot}^{0},\,\sigma_{t}^{1}=y,\,B_{\cdot}^{0},\,C_{1},\,\mathcal{G}\right]
\end{eqnarray*}
where $M_{B_{\cdot}^{0},\alpha}={\displaystyle \sup_{0\leq t\leq T}}\left({\displaystyle \sup_{y\geq0}}\left(y^{\alpha}p_{t}\left(y\,|\,B_{\cdot}^{0},\,\mathcal{G}\right)\right)\right)$
for which we have $\mathbb{E}\left[M_{B_{\cdot}^{0},\alpha}^{q}\right]<\infty$
(by Theorem~\ref{thm:3.1}). Multiplying the above by $h_{1}(y)$, integrating
in $y$ and using the law of total expectation, we obtain
\begin{eqnarray*}
& & \int_{\mathbb{R}_+}h_{1}(y)y^{\alpha}\left(\frac{\partial u_{C_{1}}}{\partial x}\left(t,\,x,\,y,\,W_{\cdot}^{0},\,B_{\cdot}^{0},\,\mathcal{G}\right)\right)^{2}dy \\
& & \qquad \qquad \leq M_{B_{\cdot}^{0},\alpha}\times\mathbb{E}\left[h\left(\sigma_{t}^{1}\right)u_{x}^{2}\left(t,\,x,\,W_{\cdot}^{0},\,\mathcal{G},\,C_{1},\,h\left(\sigma_{.}\right)\right)\,|\,W_{\cdot}^{0},\,B_{\cdot}^{0},\,C_{1},\,\mathcal{G}\right]
\end{eqnarray*}
Thus, writing $\mathbb{E}_{C_{1}}$ for the expectation given $C_{1}$,
we have
\begin{eqnarray*}
& & \mathbb{E}_{C_{1}}\left[\int_{0}^{T}\int_{\mathbb{R}^+}\int_{\mathbb{R}^+} h_{1}(y) y^{\alpha} 
\left(\frac{\partial u_{C_{1}}}{\partial x}\left(t,\,x,\,y,\,W_{\cdot}^{0},\,B_{\cdot}^{0},\,
\mathcal{G}\right)\right)^{2}dydxdt\right]
\\
& & \quad \leq\mathbb{E}_{C_{1}}\left[M_{B_{\cdot}^{0},\alpha}\int_{0}^{T}\int_{\mathbb{R}^{+}}
\mathbb{E}\left[h\left(\sigma_{t}^{1}\right)u_{x}^{2}\left(t,x,W_{\cdot}^{0},\mathcal{G},C_{1},h\left(\sigma_{.}\right)\right)\,|
W_{\cdot}^{0},B_{\cdot}^{0},C_{1},\mathcal{G}\right]dxdt\right] \\
& & \quad =\mathbb{E}_{C_{1}}\left[M_{B_{\cdot}^{0},\alpha}\mathbb{E}\left[\int_{0}^{T}h\left(\sigma_{t}^{1}\right)
\int_{\mathbb{R}^{+}}u_{x}^{2}\left(t,x,W_{\cdot}^{0},\mathcal{G},C_{1},h\left(\sigma_{.}\right)\right)dxdt\,|
W_{\cdot}^{0},B_{\cdot}^{0},C_{1},\mathcal{G}\right]\right]
\end{eqnarray*}
by Tonelli's Theorem. Then, by 2. of Theorem~\ref{thm:4.1}, the above quantity
is bounded by a multiple of
\begin{eqnarray*}
\mathbb{E}_{C_{1}}\left[M_{B_{\cdot}^{0},\alpha}\mathbb{E}\left[\int_{\mathbb{R}^{+}}u_{0}^{2}(x)dx\,|\,W_{\cdot}^{0},\,B_{\cdot}^{0},\,C_{1},\,\mathcal{G}\right]\right] 
& =& \mathbb{E}_{C_{1}}\left[M_{B_{\cdot}^{0},\alpha}\int_{\mathbb{R}^{+}}u_{0}^{2}(x)dx\right] \\
& \leq &\mathbb{E}_{C_{1}}\left[M_{B_{\cdot}^{0},\alpha}^{q}\right]\mathbb{E}\left[||u_{0}||_{L^{2}\left(\mathbb{R}^{+}\right)}^{2q'}\right]<\infty
\end{eqnarray*}
by our assumptions. This is the estimate for the $x$-derivative of
$u_{C_{1}}$. To obtain an estimate for the density itself, we follow
the same steps without multiplying by $h_{1}(y)$ and without integrating
in $t$. In that case, when we recall 2. from Theorem~\ref{thm:4.1}, we drop
the integral term of the LHS and our upper bound is again $\mathbb{E}_{C_{1}}\left[M_{B_{\cdot}^{0},\alpha}^{q}\right]\mathbb{E}\left[||u_{0}||_{L^{2}\left(\mathbb{R}^{+}\right)}^{2q'}\right]<\infty$
which is independent of $t\in\left[0,\,T\right]$. Thus we obtain
\begin{eqnarray*}
& & \sup_{0\leq t\leq T}\mathbb{E}_{C_{1}}\left[\int_{\mathbb{R}^+}\int_{\mathbb{R}^+}y^{\alpha}u_{C_{1}}^{2}\left(t,\,x,\,y,\,W_{\cdot}^{0},\,B_{\cdot}^{0},\,\mathcal{G}\right)dydx\right] \\
&& \qquad \leq\mathbb{E}_{C_{1}}\left[M_{B_{\cdot}^{0},\alpha}^{q}\right]\mathbb{E}\left[||u_{0}||_{L^{2}\left(\mathbb{R}^{+}\right)}^{2q'}\right]<\infty
\end{eqnarray*}
Adding our two estimates we can easily deduce that
\[
\left\Vert u_{C_{1}}\right\Vert _{H_{\alpha, \, C_1}\cap L_{\alpha, \, C_1}}^{2}\leq C\mathbb{E}_{C_{1}}\left[M_{B_{\cdot}^{0},\alpha}^{q}\right]\mathbb{E}\left[||u_{0}||_{L^{2}\left(\mathbb{R}^{+}\right)}^{2q'}\right]<\infty
\]
for some $C>0$ (independent of $C_1$) and our proof is complete.
\end{proof}
  
If $C_1$ has a nice distribution such that  $\mathbb{E}\left[\mathbb{E}_{C_{1}}\left[M_{B_{\cdot}^{0},\alpha}^{q}\right]\right] < \infty$, we can take expectations on the last estimate to deduce the existence of a regular density for the limiting empirical measure process, justifying the validity of the approximate computation in \eqref{eq:1.8}. Substituting now $\int_{\mathbb{R}^{2}}f\cdot dv_{t,C_{1}}=\int_{\mathbb{R}^{2}}f(x,y)u_{C_{1}}(t,\,x,\,y)dxdy$
in the distributional SPDE of Theorem~\ref{thm:2.6} and integrating by parts, we obtain the SPDE for the density of $v_{t,C_1}$:
\begin{eqnarray}
u_{C_{1}}(t,\,x,\,y)&=& U_{0}(x, \, y \,|\,G)-r_{1}\int_{0}^{t}\left(u_{C_{1}}(s,\,x,\,y)\right)_{x}ds \nonumber \\
& & \qquad +\frac{1}{2}\int_{0}^{t}h^{2}(y)\left(u_{C_{1}}(s,\,x,\,y)\right)_{x}ds-k_{1}\theta_{1}\int_{0}^{t}
\left(u_{C_{1}}(s,\,x,\,y)\right)_{y}ds \nonumber \\
& & \qquad +k_{1}\int_{0}^{t}\left(yu_{C_{1}}(s,\,x,\,y)\right)_{y}ds+\frac{1}{2}\int_{0}^{t}h^{2}(y)\left(u_{C_{1}}(s,\,x,\,y)\right)_{xx}ds
\nonumber \\
& & \qquad +\xi_{1}\rho_{3}\rho_{1,1}\rho_{2,1}\int_{0}^{t}\left(h\left(y\right)\sqrt{y}u_{C_{1}}(s,\,x,\,y)\right)_{xy}ds \nonumber \\
& & \qquad +\frac{\xi_{1}^{2}}{2}\int_{0}^{t}\left(yu_{C_{1}}(s,\,x,\,y)\right)_{yy}ds-\rho_{1,1}\int_{0}^{t}h(y)\left(u_{C_{1}}(s,\,x,\,y)\right)_{x}dW_{s}^{0}\nonumber \\
& & \qquad -\xi_{1}\rho_{2,1}\int_{0}^{t}\left(\sqrt{y}u_{C_{1}}(s,\,x,\,y)\right)_{y}dB_{s}^{0}, \label{eq:4.11}
\end{eqnarray}
where the derivatives in $y$ and the second derivative in $x$ are considered in the distributional sense (over the test space $C_{0}^{test}$ defined in Theorem~\ref{thm:2.6}), while $U_0(x, \, y \,|\,G)$ stands for the initial density with marginals $u_{0}(x\,|\,G)$ and $p_{0}(y\,|\,\mathcal{G})$.

\section{Using the SPDE to improve the regularity}

In this section we write $\tilde{L}^2_{\delta} = L_{y^{\delta}}^{2}\left(\mathbb{R}^{+}
\times\mathbb{R}^{+}\right)$. Observe that under this notation, the space $L_{\alpha, \, C_1}$ defined in Theorem~\ref{thm:4.2} can be written as
\[
L_{\alpha, \, C_1}=L^{\infty}\left(\left[0,\,T\right];\,L^{2}\left(\left(\Omega, \, \mathcal{F}, \, \mathbb{P}\left(\cdot \, |\,C_1\,\right)\right);\,\tilde{L}^2_{\alpha}\right)\right),
\]
which means that the $L_{\alpha, \, C_1}$ - norm of a function equals the supremum in $t \in \left[0, \, T\right]$ of the $L^{2}\left(\Omega, \, \mathcal{F}, \, \mathbb{P}\left(\cdot \, |\,C_1\,\right)\right)$ norm of its $\tilde{L}^2_{\alpha}$ - norm. All the expectations in this chapter are taken under the conditional probability measure $\mathbb{P}\left(\cdot \, |\,C_1\,\right)$. For any value of the coefficient vector $C_1$, we will write $\Omega$ for $\left(\Omega, \, \mathcal{F}, \, \mathbb{P}\left(\cdot \, |\,C_1\,\right)\right)$ and $\mathbb{P}$ for $\mathbb{P}\left(\cdot \, |\,C_1\,\right)$ for simplicity.  

In this section, we exploit the initial-boundary value
problem satisfied by $u_{C_{1}}$, in order to establish the best possible regularity for our density. First, we need to define the initial-boundary value
problem explicitly. We give the following definition of an $\alpha$-solution to our problem for $\alpha\geq0$, the properties of which
are all satisfied by the density function $u_{C_{1}}$ for all $\alpha\geq 0$
as we have shown in the previous section.
\begin{defn}
For a given value of the coefficient vector $C_{1}$, let $U_{0}\in L^{2}\left(\Omega;\,\tilde{L}_{\alpha}^{2}
\right)$ be a random function which is extended to be zero outside $\mathbb{R}^{+}\times\mathbb{R}^{+}$, $h$ a function having polynomial growth in $\mathbb{R}^{+}$, and $\rho$ a real number.
Given $C_{1}$, $\rho$ and the functions $U_{0}$ and $h$, we say that $u$ is an $\alpha$-solution to our problem when the following are satisfied;

\newpage{}

\begin{enumerate}
\item   $u$ is adapted to the filtration $\{\sigma\left(\mathcal{G},\,W_{t}^{0},\,B_{t}^{0}\right):\,t\geq0\}$
and belongs to the space $H_{\alpha, \, C_1}\cap L_{\alpha, \, C_1}$, where $L_{\alpha, \, C_1}$
and $H_{\alpha, \, C_1}$ are the spaces defined in Theorem~\ref{thm:4.2}.
\item   $u$ vanishes for negative $y$ and satisfies the SPDE
\begin{eqnarray}
u(t,x,y)&=& U_{0}(x,\,y)-r_{1}\int_{0}^{t}\left(u(s,x,y)\right)_{x}ds \nonumber \\
&& \qquad +\frac{1}{2}\int_{0}^{t}h^{2}(y)\left(u(s,x,y)\right)_{x}ds-k_{1}\theta_{1}\int_{0}^{t}\left(u(s,x,y)\right)_{y}ds 
\nonumber \\
&& \qquad +k_{1}\int_{0}^{t}\left(yu(s,\,x,\,y)\right)_{y}ds+\frac{1}{2}\int_{0}^{t}h^{2}(y)\left(u(s,x,y)\right)_{xx}ds
\nonumber \\
& & \qquad +\rho\int_{0}^{t}\left(h\left(y\right)\sqrt{y}u(s,\,x,\,y)\right)_{xy}ds \nonumber \\
& & \qquad +\frac{\xi_{1}^{2}}{2}\int_{0}^{t}\left(yu(s,x,y)\right)_{yy}ds-\rho_{1,1}\int_{0}^{t}h(y)
\left(u(s,x,y)\right)_{x}dW_{s}^{0} \nonumber \\
& & \qquad -\xi_{1}\rho_{2,1}\int_{0}^{t}\left(\sqrt{y}u(s,x,y)\right)_{y}dB_{s}^{0}, \label{eq:5.1}
\end{eqnarray}
for all $x\geq0$ and $y\in\mathbb{R}$, where $u_{y}$, $u_{yy}$
and $u_{xx}$ are considered in the distributional sense over the space of test functions
\[
C_{0}^{test}=\{g\in C_{b}^{2}(\mathbb{R}^{+}\times\mathbb{R}):\,g(0,y)=0, \;\forall y\in\mathbb{R}\}.
\]
\end{enumerate}
\end{defn}

Observe that for $\rho = \xi_{1}\rho_{3}\rho_{1,1}\rho_{2,1}$, where $\rho_3$ is the correlation between $W^0$ and $B^0$ (i.e $dW^0_t \cdot dB^0_t = \rho_3dt$), we obtain the SPDE obtained in the previous section. The main result of this section is given in the following theorem.

\begin{thm}\label{thm:5.2}
Fix the value of the coefficient vector $C_1$, the real number $\rho$ and the initial data function $U_{0}$. Let $u$ be an $\alpha$-solution to our problem, for all $\alpha\geq 0$. Then, the weak derivative $u_{y}$ of $u$ exists and we have
\[
u_{y}\in L^{2}\left(\left[0,\,T\right]\times\Omega;\,\tilde{L}^2_{\alpha}\right)
\]
for all $\alpha\geq 2$.
\end{thm}

To prove the above Theorem, we need to modify appropriately the kernel smoothing
method which has been developed in \cite{BHHJR, KX99, HL16}. The idea is
to test our SPDE against
\[
\phi_{\epsilon}(z,\,y)=\frac{1}{\sqrt{2\pi\epsilon}}e^{-\frac{(\sqrt{z}-y)^{2}}{2\epsilon}},
\;y,\,z\in\mathbb{R},
\]
in order to obtain a smoothed version of it. Keep in mind that we do not have to integrate over the negative numbers, where the square root is not defined, since by definition our solution vanishes there. From the smoothed version of our SPDE, we can obtain for any $\delta > 1$, an identity involving some finite $\tilde{L}^2_{\delta}$, $\tilde{L}^2_{\delta - 1}$ and $\tilde{L}^2_{\delta - 2}$ norms and inner products of smoothed quantities involving the solution and its derivatives (we shall refer to this as the $\delta$-identity). The finiteness of the $\tilde{L}^2$ norms appearing in the $\delta$-identity for any $\delta$, follows from the good global regularity of our functions when they are smoothed with $\phi_{\epsilon}(z,\,y)$. This good smoothing property of $\phi_{\epsilon}(z,\,y)$, which is not obvious and has to be established concretely, follows from the fact that $\phi_{\epsilon}(z,\,y)$ is just the standard heat kernel (used in the standard kernel smoothing method) composed with a square root function. Then, we can obtain the desired result by manipulating appropriately the $\delta$-identity for each $\delta > 1$ and by taking $\epsilon \rightarrow 0^{+}$, provided that $\phi_{\epsilon}(z,\,y)$ has the same convergence properties as the standard heat kernel. The composition with the square root function leads to the elimination of some bad terms in our $\delta$-identities (terms that could explode as $\epsilon \rightarrow 0^{+}$ under our weak regularity assumptions) which would appear if we used the standard kernel smoothing method. The intuition behind the choice of this composition is that our solution is expected to be the density of a law describing a CIR process in the $y$-direction, while the mapping $z \rightarrow \sqrt{z}$ transforms a CIR process into a process of a constant volatility (like the Brownian motion with drift in the constant volatility model, where the standard kernel smoothing method works).

As we have already mentioned, we need to show that $\phi_{\epsilon}(z,\,y)$ introduced above possesses all the nice smoothing and convergence (as $\epsilon \rightarrow 0^{+}$) properties of the standard heat kernel, and also for many different weighted $L^2$ norms. These natural extensions are given in the following technical lemmas, the proofs of which have been put in the Appendix since they are simple modifications of the proofs of the corresponding properties of the standard heat kernel under the standard $L^2$ norm. 

\begin{lem}\label{lem:5.3}
Suppose that $\left(\Lambda, \, \mu \right)$ is a measure space. For any function $u$ supported in $\Lambda \times \mathbb{R}^{+}$ we define
the functions
\[
J_{u,\epsilon}(\lambda,\,y)=\int_{\mathbb{R}}u(\lambda,\,z)\phi_{\epsilon}(z,\,y)dz
\]
and
\[
J_{u}(\lambda,\,y)=2yu(\lambda,\,y^{2}).
\]
Suppose that for all $\delta' >-1$ we have $J_{u} \in L^{2}\left(\Lambda;\,L_{y^{\delta'}}^{2}\left(\mathbb{R}^{+}
\right)\right)$. Then for all $\delta' >-1$ we have the following regularity and convergence results;
\begin{enumerate}
\item $J_{u,\epsilon}(\cdot,\,\cdot)$ is smooth and for all $n \in \mathbb{N}$ it holds that $\frac{\partial^{n}}{\partial y^{n}}J_{u,\epsilon}(\cdot,\,\cdot) \in L^{2}\left(\Lambda;\,L_{y^{\delta'}}^{2}\left(\mathbb{R}^{+}
\right)\right)$.

\item $J_{u,\epsilon}(\cdot,\,\cdot)\rightarrow J_{u}(\cdot,\,\cdot)$
strongly in $L^{2}\left(\Lambda;\,L_{y^{\delta'}}^{2}\left(\mathbb{R}^{+}
\right)\right)$, as $\epsilon\rightarrow0^{+}$.

\end{enumerate}
\end{lem}

\begin{lem}\label{lem:5.4}
In the notation of lemma~\ref{lem:5.3}, assume that for some $\delta' > 0$, there exists a constant $C>0$ and an $n\in\mathbb{N}$ such that for any $\epsilon>0$ we have
\begin{equation}\label{eq:5.5}
\left\Vert \frac{\partial^{l}}{\partial y^{l}}J_{u,\epsilon}\left(s,\,\cdot\right)\right\Vert _{L^{2}\left(\Lambda;\,L_{y^{\delta'}}^{2}\left(\mathbb{R}^{+}
\right)\right)}^{2} \leq C
\end{equation}
for some function $u$ supported in $\Lambda \times \mathbb{R}^{+}$ and all $l\in\{1,\,2,\,...,\,n\}$. Then we have $\frac{\partial^{l}}{\partial y^{l}}J_{u}\in L^{2}\left(\Lambda;\,L_{y^{\delta'}}^{2}\left(\mathbb{R}^{+}
\right)\right)$ and also $\frac{\partial^{l}}{\partial y^{l}}J_{u,\epsilon}\rightarrow\frac{\partial^{l}}{\partial y^{l}}J_{u}$
strongly in $L^{2}\left(\Lambda;\,L_{y^{\delta'}}^{2}\left(\mathbb{R}^{+}
\right)\right)$
as $\epsilon\rightarrow0^{+}$, for all $l\in\{1,\,2,\,...,\,n\}$.
\end{lem}

We will use these two lemmas for $\Lambda = \Omega \times \mathbb{R}^{+}$ and $\Lambda = \left[0, \, t \right] \times \Omega \times \mathbb{R}^{+}$ for $t \geq 0$, with the corresponding product of measures (where $\Omega$ is equipped with the measure $\mathbb{P}\left(\cdot \, |\,C_1\,\right)$ and both $\left[0, \, t \right]$ and $\mathbb{R}^{+}$ are equipped with the standard Lebesgue measure). This means that, in the notation we introduced at the beginning of this section, the two lemmas will be used for functions in the spaces $L^{2}\left(\Omega \times \mathbb{R}^{+};\,L_{y^{\delta'}}^{2}\left(\mathbb{R}^{+}
\right)\right) = L^{2}\left(\Omega;\,\tilde{L}_{\delta'}^{2}\right)$ and $L^{2}\left(\left[0, \, t \right] \times \Omega \times \mathbb{R}^{+};\,L_{y^{\delta'}}^{2}\left(\mathbb{R}^{+}
\right)\right) = L^{2}\left(\left[0, \, t \right] \times \Omega;\,\tilde{L}_{\delta'}^{2}\right)$. 

We fix now a function $u$ which is an $\alpha$-solution to our problem for all $\alpha \geq 0$ and we set:
\[
I_{\epsilon,g(z)}(s,\,x,\,y)=\int_{\mathbb{R}}g(z)u(s,\,x,\,z)\phi_{\epsilon}(z,\,y)dz
\]
for any function $g$ of $z$ and $\epsilon > 0$. Under this notation, the $\delta$-identity (for any $\delta > 1$) for our solution $u$ is given in the following lemma:

\begin{lem}[\textbf{the $\delta$-identity}]

The following estimate holds for any $\delta > 1$,
\begin{eqnarray}
\left\Vert I_{\epsilon,1}(t,\,\cdot)\right\Vert _{L^{2}\left(\Omega;\,\tilde{L}_{\delta}^2\right)}^{2} &=& \left\Vert \int_{\mathbb{R}}U_{0}(\cdot,\,z)\phi_{\epsilon}(z,\,\cdot)dz\right\Vert _{L^{2}\left(\Omega;\,\tilde{L}_{\delta}^2\right)}^{2} \nonumber \\
& & \qquad +\int_{0}^{t}\left\langle \frac{\partial}{\partial x}I_{\epsilon,h^{2}(z)}(s,\,\cdot),\,I_{\epsilon,1}(s,\,\cdot)\right\rangle _{L^{2}\left(\Omega;\,\tilde{L}_{\delta}^2\right)}ds \nonumber \\
& & \qquad +\delta\left(k_{1}\theta_{1}-\frac{\xi_{1}^{2}}{4}\right)\int_{0}^{t}\left\langle I_{\epsilon,z^{-\frac{1}{2}}}(s,\,\cdot),\,I_{\epsilon,1}(s,\,\cdot)\right\rangle _{L^{2}\left(\Omega;\,\tilde{L}^2_{\delta-1}\right)}ds \nonumber \\
& & \qquad +\left(k_{1}\theta_{1}-\frac{\xi_{1}^{2}}{4}\right)\int_{0}^{t}\left\langle I_{\epsilon,z^{-\frac{1}{2}}}(s,\,\cdot),\,\frac{\partial}{\partial y}I_{\epsilon,1}(s,\,\cdot)\right\rangle _{L^{2}\left(\Omega;\,\tilde{L}_{\delta}^2\right)}ds \nonumber \\
& & \qquad -\delta k_{1}\int_{0}^{t}\left\langle I_{\epsilon,z^{\frac{1}{2}}}(s,\,\cdot),\,I_{\epsilon,1}(s,\,\cdot)\right\rangle _{L^{2}\left(\Omega;\,\tilde{L}^2_{\delta-1}\right)}ds \nonumber \\
& & \qquad -k_{1}\int_{0}^{t}\left\langle I_{\epsilon,z^{\frac{1}{2}}}(s,\,\cdot),\,\frac{\partial}{\partial y}I_{\epsilon,1}(s,\,\cdot)\right\rangle _{L^{2}\left(\Omega;\,\tilde{L}_{\delta}^2\right)}ds \nonumber \\
& & \qquad -\int_{0}^{t}\left\langle \frac{\partial}{\partial x}I_{\epsilon,h^{2}(z)}(s,\,\cdot),\,\frac{\partial}{\partial x}I_{\epsilon,1}(s,\,\cdot)\right\rangle _{L^{2}\left(\Omega;\,\tilde{L}_{\delta}^2\right)}ds \nonumber \\
& & \qquad -\delta\rho\int_{0}^{t}\left\langle \frac{\partial}{\partial x}I_{\epsilon,h\left(z\right)}(s,\cdot),\,I_{\epsilon,1}(s,\cdot)\right\rangle _{L^2\left(\Omega;\, \tilde{L}_{\delta-1}^2 \right)}ds \nonumber \\
& & \qquad +\rho_{1,1}^{2}\int_{0}^{t}\left\Vert \frac{\partial}{\partial x}I_{\epsilon,h(z)}(s,\,\cdot)\right\Vert _{L^{2}\left(\Omega;\,\tilde{L}_{\delta}^2\right)}^{2}ds
\nonumber \\
& & \qquad +\delta(\delta-1)\frac{\xi_{1}^{2}}{8}\int_{0}^{t}\left\Vert I_{\epsilon,1}(s,\,\cdot)\right\Vert_{L^2\left(\Omega;\,\tilde{L}^2_{\delta-2}\right)}^{2}ds
\nonumber \\
& & \qquad -\frac{\xi_{1}^{2}}{4}\left(1-\rho_{2,1}^{2}\right)\int_{0}^{t}\left\Vert \frac{\partial}{\partial y}I_{\epsilon,1}(s,\,\cdot)\right\Vert _{L^{2}\left(\Omega;\,\tilde{L}_{\delta}^2\right)}^{2}ds. \nonumber \\
& & \qquad - \left(\rho - \xi_{1}\rho_{3}\rho_{1,1}\rho_{2,1}\right)\int_{0}^{t}\left\langle \frac{\partial}{\partial x}I_{\epsilon,h\left(z\right)}(s,\cdot),\,\frac{\partial}{\partial y}I_{\epsilon,1}(s,\cdot)\right\rangle _{L^2\left(\Omega; \, \tilde{L}_{\delta}^2\right)}ds. \nonumber \\
\label{eq:5.17}
\end{eqnarray} 

All the terms in the above identity are finite. 

\end{lem}

\begin{proof}
Notice first that the finiteness of each term in the identity we need to prove is a consequence of 1. of Lemma~\ref{lem:5.3}. Next, we observe that by definition of $\phi_{\epsilon}$ we have
\begin{eqnarray}
\int_{\mathbb{R}}g(z)u(s,\,x,\,z)\frac{\partial}{\partial z}\phi_{\epsilon}(z,\,y)dz 
&=& -\frac{1}{2}\int_{\mathbb{R}}z^{-\frac{1}{2}}g(z)u(s,\,x,\,z)\frac{\partial}{\partial y}
\phi_{\epsilon}(z,\,y)dz \nonumber \\
&=& -\frac{1}{2}\frac{\partial}{\partial y} I_{\epsilon,g(z)z^{-\frac{1}{2}}}(s,\,x,\,y) \label{eq:5.8}
\end{eqnarray}
and also
\begin{eqnarray}
& &  \int_{\mathbb{R}}g(z)u(s,\,x,\,z)\frac{\partial^{2}}{\partial z^{2}}\phi_{\epsilon}(z,\,y)dz \nonumber \\
& & \quad = -\int_{\mathbb{R}}g(z)u(s,\,x,\,z)\frac{\partial}{\partial z}\left(\frac{\partial}{\partial y}\phi_{\epsilon}(z,\,y)\frac{1}{2\sqrt{z}}\right)dz \nonumber \\
&& \quad = -\int_{\mathbb{R}}g(z)u(s,\,x,\,z)\frac{\partial}{\partial y}\frac{\partial}{\partial z}\phi_{\epsilon}(z,\,y)\frac{1}{2\sqrt{z}}dz +\int_{\mathbb{R}}g(z)u(s,\,x,\,z)\frac{\partial}{\partial y}\phi_{\epsilon}(z,\,y)\frac{1}{4z\sqrt{z}}dz \nonumber \\
& & \quad = -\left(\frac{1}{2}\int_{\mathbb{R}}g(z)z^{-\frac{1}{2}}u(s,\,x,\,z)\frac{\partial}{\partial z}\phi_{\epsilon}(z,\,y)
dz\right)_{y} + \frac{1}{4}\int_{\mathbb{R}}g(z)z^{-\frac{3}{2}}u(s,\,x,\,z)\frac{\partial}{\partial y}\left(\phi_{\epsilon}(z,\,y)\right)dz \nonumber \\
& & \quad =\frac{1}{4}\left(I_{\epsilon,g(z)z^{-1}}(s,\,x,\,y)\right)_{yy}+\frac{1}{4}\left(I_{\epsilon,g(z)z^{-\frac{3}{2}}}(s,\,x,\,y)\right)_{y},
\label{eq:5.9}
\end{eqnarray}
for any $\epsilon>0$ and $\alpha\in\mathbb{R}$. Thus, after testing
\eqref{eq:5.1} against $\phi_{\epsilon}$, by substituting from \eqref{eq:5.8}
and \eqref{eq:5.9}, and by interchanging the $x$-derivatives with the integrals, we obtain
\begin{eqnarray}
I_{\epsilon,1}(t,\,x,\,y)&=& \int_{\mathbb{R}^{+}}U_{0}(x,\,z)\phi_{\epsilon}(z,\,y)dz-r_{1}\int_{0}^{t}\frac{\partial}{\partial x}I_{\epsilon,1}(s,\,x,\,y)ds \nonumber \\
& & \quad +\frac{1}{2}\int_{0}^{t}\frac{\partial}{\partial x}I_{\epsilon,h^{2}(z)}(s,\,x,\,y)ds -\frac{k_{1}\theta_{1}}{2}\int_{0}^{t}\frac{\partial}{\partial y}I_{\epsilon,z^{-\frac{1}{2}}}(s,\,x,\,y)ds \nonumber \\
& & \quad +\frac{k_{1}}{2}\int_{0}^{t}\frac{\partial}{\partial y}I_{\epsilon,z^{\frac{1}{2}}}(s,\,x,\,y)ds +\frac{1}{2}\int_{0}^{t}\frac{\partial^{2}}{\partial x^{2}}I_{\epsilon,h^{2}(z)}(s,\,x,\,y)ds \nonumber \\
& & \quad +\frac{\xi_{1}^{2}}{8}\int_{0}^{t}\frac{\partial^{2}}{\partial y^{2}}I_{\epsilon,1}(s,\,x,\,y)ds +\frac{\xi_{1}^{2}}{8}\int_{0}^{t}\frac{\partial}{\partial y}I_{\epsilon,z^{-\frac{1}{2}}}(s,\,x,\,y)ds \nonumber \\
& & \quad +\frac{\rho}{2}\int_{0}^{t}\frac{\partial^{2}}{\partial x \partial y}I_{\epsilon,h\left(z\right)}(s,\,x,\,y)ds -\rho_{1,1}\int_{0}^{t}\frac{\partial}{\partial x}I_{\epsilon,h(z)}(s,\,x,\,y)dW_{s}^{0} \nonumber \\
& & \quad -\frac{\xi_{1}}{2}\rho_{2,1}\int_{0}^{t}\frac{\partial}{\partial y}I_{\epsilon,1}(s,\,x,\,y)dB_{s}^{0}. \label{eq:5.10}
\end{eqnarray}
By applying Ito's formula for the $L^{2}(\mathbb{R}^{+})$ norm on
\eqref{eq:5.10} (Theorem~3.1 from \cite{KR81} for the triple $H_{0}^{1}\subset L^{2}\subset H^{-1}$),
multiplying by $y^{\delta}$ for $\delta > 1$ and integrating in
$y$ over $\mathbb{R}^{+}$, we obtain the equality
\begin{eqnarray}
& & \left\Vert I_{\epsilon,1}(t,\,\cdot)\right\Vert _{\tilde{L}_{\delta}^2}^{2}=\left\Vert \int_{\mathbb{R}}U_{0}(\cdot,z)
\phi_{\epsilon}(z,\cdot)dz\right\Vert _{\tilde{L}_{\delta}^2}^{2} \nonumber \\
& & \qquad -2r_{1}\int_{0}^{t}\left\langle \frac{\partial}{\partial x}I_{\epsilon,1}(s,\cdot),I_{\epsilon,1}(s,\cdot)\right\rangle _{\tilde{L}_{\delta}^2}ds+\int_{0}^{t}\left\langle \frac{\partial}{\partial x}I_{\epsilon,h^{2}(z)}(s,\cdot),I_{\epsilon,1}(s,\cdot)\right\rangle _{\tilde{L}_{\delta}^2}ds \nonumber \\
& & \qquad -k_{1}\theta_{1}\int_{0}^{t}\left\langle \frac{\partial}{\partial y}I_{\epsilon,z^{-\frac{1}{2}}}(s,\cdot),I_{\epsilon,1}(s,\cdot)\right\rangle _{\tilde{L}_{\delta}^2}ds+k_{1}\int_{0}^{t}\left\langle \frac{\partial}{\partial y}I_{\epsilon,z^{\frac{1}{2}}}(s,\cdot),I_{\epsilon,1}(s,\cdot)\right\rangle _{\tilde{L}_{\delta}^2}ds \nonumber \\
& & \qquad +\int_{0}^{t}\left\langle \frac{\partial^{2}}{\partial x^{2}}I_{\epsilon,h^{2}(z)}(s,\cdot),\,I_{\epsilon,1}(s,\cdot)\right\rangle _{\tilde{L}_{\delta}^2}ds+\frac{\xi_{1}^{2}}{4}\int_{0}^{t}\left\langle \frac{\partial^{2}}{\partial y^{2}}I_{\epsilon,1}(s,\cdot),
I_{\epsilon,1}(s,\cdot)\right\rangle _{\tilde{L}_{\delta}^2}ds \nonumber \\
& & \qquad +\rho\int_{0}^{t}\left\langle \frac{\partial^{2}}{\partial x \partial y}I_{\epsilon,h\left(z\right)}(s,\cdot),\,I_{\epsilon,1}(s,\cdot)\right\rangle _{\tilde{L}_{\delta}^2}ds \nonumber \\
& & \qquad +\xi_{1}\rho_{3}\rho_{1,1}\rho_{2,1}\int_{0}^{t}\left\langle \frac{\partial}{\partial x}I_{\epsilon,h\left(z\right)}(s,\cdot),\,\frac{\partial}{\partial y}I_{\epsilon,1}(s,\cdot)\right\rangle _{\tilde{L}_{\delta}^2}ds \nonumber \\
& & \qquad +\frac{\xi_{1}^{2}}{4}\int_{0}^{t}\left\langle \frac{\partial}{\partial y}I_{\epsilon,z^{-\frac{1}{2}}}(s,\cdot),I_{\epsilon,1}(s,\,\cdot)\right\rangle _{\tilde{L}_{\delta}^2}ds+\rho_{1,1}^{2}\int_{0}^{t}\left\Vert \frac{\partial}{\partial x}
I_{\epsilon,h(z)}(s,\cdot)\right\Vert _{\tilde{L}_{\delta}^2}^{2}ds \nonumber \\
& & \qquad +\frac{\xi_{1}^{2}}{4}\rho_{2,1}^{2}\int_{0}^{t}\left\Vert \frac{\partial}{\partial y}
I_{\epsilon,1}(s,\cdot)\right\Vert _{\tilde{L}_{\delta}^2}^{2}ds
\nonumber \\
& & \qquad -2\rho_{1,1}\int_{0}^{t}\left\langle \frac{\partial}{\partial x}I_{\epsilon,h(z)}(s,\cdot),I_{\epsilon,1}(s,\cdot)\right\rangle _{\tilde{L}_{\delta}^2}dW_{s}^{0} \nonumber \\
& &\qquad -\xi_{1}\rho_{2,1}\int_{0}^{t}\left\langle \frac{\partial}{\partial y}I_{\epsilon,1}(s,\cdot),I_{\epsilon,1}(s,\cdot)\right\rangle _{\tilde{L}_{\delta}^2}dB_{s}^{0}. \label{eq:5.11}
\end{eqnarray}
Observe now that by the definition of $u_{xx}$ in our SPDE, we have
\begin{eqnarray}
& & \int_{\mathbb{R}^+}\int_{\mathbb{R}}u_{xx}(s,\,x,\,z)\phi_{\epsilon}(z,\,y)f(x)dzdx=\int_{\mathbb{R}^+}\int_{\mathbb{R}}u(s,\,x,\,z)\phi_{\epsilon}(z,\,y)f_{xx}(x)dzdx \nonumber \\
& & \qquad =-\int_{\mathbb{R}^+}\int_{\mathbb{R}}u_{x}(s,\,x,\,z)\phi_{\epsilon}(z,\,y)f_{x}(x)dzdx \label{eq:5.12}
\end{eqnarray}
for any smooth $f$ vanishing at zero. Since $u\in H_{0}^{1}$, this
mapping over all such functions $f$ defines a distribution in $H^{-1}$,
and since those test functions are dense in $H_{0}^{1}$, we have
that \eqref{eq:5.12} holds for any $f\in H_{0}^{1}$. In particular, for
$f=I_{\epsilon,1}(s,\,\cdot,\,y)$, multiplying \eqref{eq:5.12} by $y^{\delta}$ and then integrating in $y$ and $t$ over $\mathbb{R}^{+}$, we obtain
\begin{equation}
\int_{0}^{t}\left\langle \frac{\partial^{2}}{\partial x^{2}}I_{\epsilon,h^{2}(z)}(s,\,\cdot),\,I_{\epsilon,1}(s,\,\cdot)\right\rangle _{\tilde{L}_{\delta}^2}ds
=-\int_{0}^{t}\left\langle \frac{\partial}{\partial x}I_{\epsilon,h^{2}(z)}(s,\,\cdot),\,\frac{\partial}{\partial x}I_{\epsilon,1}(s,\,\cdot)\right\rangle _{\tilde{L}_{\delta}^2}ds. \\
\label{eq:5.13}
\end{equation}
Next, integration by parts implies
\begin{eqnarray}
& & \int_{0}^{t}\left\langle \frac{\partial^{2}}{\partial y^{2}}I_{\epsilon,1}(s,\,\cdot),\,I_{\epsilon,1}(s,\,\cdot)\right\rangle _{\tilde{L}_{\delta}^2}ds
\nonumber \\
& & \qquad =\delta(\delta-1)\frac{1}{2}\int_{0}^{t}\left\Vert I_{\epsilon,1}(s,\,\cdot)\right\Vert _{\tilde{L}^2_{\delta-2}}^{2}
ds-\int_{0}^{t}\left\Vert \frac{\partial}{\partial y}I_{\epsilon,1}(s,\,\cdot)\right\Vert _{\tilde{L}_{\delta}^2}^{2}ds \nonumber \\
\label{eq:5.14}
\end{eqnarray}
and
\begin{eqnarray}
& & \int_{0}^{t}\left\langle \frac{\partial^{2}}{\partial x \partial y}I_{\epsilon,h\left(z\right)}(s,\cdot),\,I_{\epsilon,1}(s,\cdot)\right\rangle _{\tilde{L}_{\delta}^2}ds  \nonumber \\
& & \;\; = - \int_{0}^{t}\left\langle \frac{\partial}{\partial x}I_{\epsilon,h\left(z\right)}(s,\cdot),\,\frac{\partial}{\partial y}I_{\epsilon,1}(s,\cdot)\right\rangle _{\tilde{L}_{\delta}^2}ds  -\delta\int_{0}^{t}\left\langle \frac{\partial}{\partial x}I_{\epsilon,h\left(z\right)}(s,\cdot),\,I_{\epsilon,1}(s,\cdot)\right\rangle _{\tilde{L}_{\delta-1}^2}ds \nonumber \\
\label{eq:5.15}
\end{eqnarray}
and also
\begin{eqnarray}
& & \int_{0}^{t}\left\langle \frac{\partial}{\partial y}I_{\epsilon,z^{\alpha}}(s,\,\cdot),\,I_{\epsilon,1}(s,\,\cdot)
\right\rangle _{\tilde{L}_{\delta}^2}ds \nonumber \\
& & \qquad = -\delta\int_{0}^{t}\left\langle I_{\epsilon,z^{\alpha}}(s,\,\cdot),\,I_{\epsilon,1}(s,\,\cdot)\right\rangle _{\tilde{L}^2_{\delta-1}}ds -\int_{0}^{t}\left\langle I_{\epsilon,z^{\alpha}}(s,\,\cdot),\,\frac{\partial}{\partial y}I_{\epsilon,1}(s,\,\cdot)
\right\rangle_{\tilde{L}_{\delta}^2}ds, \nonumber \\
\label{eq:5.16}
\end{eqnarray}
for any $\alpha$ for which the above quantities are regular enough. Note that integrating by parts in the $y$-direction is possible without leaving any boundary term at infinity, since all the terms inside the inner products are rapidly decreasing in $y$. This is also a consequence of 1. of Lemma~\ref{lem:5.3}, since for any $n \in \mathbb{N}$ and any function $f$ having derivatives in polynomially weighted $L^2$ spaces, by Morrey's inequality we have:
\begin{eqnarray*}
&& y^nf(y) \leq \frac{1}{y}\sup_{z \in \mathbb{R}}{|z^{n+1}f(z)|} \nonumber \\
& &\qquad \qquad \leq \frac{1}{y}\left(\int_{\mathbb{R}}z^{2(n+1)}f^2(z)dz+\int_{\mathbb{R}}z^{2n}f^2(z)dz+\int_{\mathbb{R}}z^{2(n+1)}(f'(z))^2dz\right)^\frac{1}{2} \rightarrow 0
\end{eqnarray*} as $y \rightarrow \infty$. Of course, we do not have boundary terms at zero either, due to the weight function $y^{\delta}$.

We will use \eqref{eq:5.13} - \eqref{eq:5.15} to get rid of second order derivative terms in our estimate. Here, it becomes clear why we have chosen to compose the standard heat kernel with $\sqrt{z}$: In \eqref{eq:5.11}, substituting the second term in the fourth row from \eqref{eq:5.14} gives again the term of the eighth row but with a negative coefficient of a bigger absolute value, which allows us to control $y$-derivative terms. It is not hard to check that that this wouldn't have been the case if we had composed the standard heat kernel with another function, when the existence of $u_{y}$ is not assumed (as in our case). By observing now that the first inner product
of the RHS of \eqref{eq:5.11} is zero, substituting also \eqref{eq:5.13}, \eqref{eq:5.15}
and \eqref{eq:5.16} for $\alpha=\pm\frac{1}{2}$ in \eqref{eq:5.11} and
taking expectations, we obtain the desired.
\end{proof}

Now that we have obtained the $\delta$-identity for all $\delta > 1$, we can proceed to the proof of our main Theorem. Our strategy is 
to establish the regularity result by controlling the derivative terms in the $\delta$-identity for all $\delta > 1$, by taking
$\epsilon \rightarrow 0^{+}$ and by using Lemma~\ref{lem:5.3} and Lemma~\ref{lem:5.4} (which gives the regularity of the limits). 

\begin{proof}[Proof of Theorem 5.2]

For all the inner products in the $\delta$-identity except the first and the seventh, we
can use the Cauchy-Schwartz inequality in the form
\[
\left\langle u_{1},\,u_{2}\right\rangle _{\tilde{L}_{\delta}^{2}}\leq\left\Vert u_{1}\right\Vert _{\tilde{L}_{2\delta_{1}}^{2}}\left\Vert u_{2}\right\Vert _{\tilde{L}_{2\delta_{2}}}
\]
for $\delta=\delta_{1}+\delta_{2}$, and then the AM-GM inequality
($ab\leq\frac{a^{2}}{4C}+Cb^{2}$) for the products of norms to obtain

\newpage{}

\begin{eqnarray}
&& \left\Vert I_{\epsilon,1}(t,\,\cdot)\right\Vert _{L^{2}\left(\Omega;\,\tilde{L}_{\delta}^2\right)}^{2} \nonumber \\
& & \qquad\leq\left\Vert \int_{\mathbb{R}}U_{0}(\cdot,\,z)\phi_{\epsilon}(z,\,\cdot)dz\right\Vert _{L^{2}\left(\Omega;\,\tilde{L}_{\delta}^2\right)}^{2} \nonumber \\
& & \qquad\qquad +\int_{0}^{t}\left\langle \frac{\partial}{\partial x}I_{\epsilon,h^{2}(z)}(s,\,\cdot),\,I_{\epsilon,1}(s,\,\cdot)\right\rangle _{L^{2}\left(\Omega;\,\tilde{L}_{\delta}^2\right)}ds \nonumber \\
& & \qquad\qquad -\delta\rho\int_{0}^{t}\left\langle \frac{\partial}{\partial x}I_{\epsilon,h\left(z\right)}(s,\cdot),\,I_{\epsilon,1}(s,\cdot)\right\rangle _{L^2\left(\Omega;\, \tilde{L}_{\delta-1}^2 \right)}ds \nonumber \\
&  & \qquad\qquad +\left|k_{1}\theta_{1}-\frac{\xi_{1}^{2}}{4}\right|\int_{0}^{t}C_1\left\Vert I_{\epsilon,z^{-\frac{1}{2}}}(s,\,\cdot)\right\Vert _{L^{2}\left(\Omega;\,\tilde{L}_{\delta}^2\right)}^{2}ds \nonumber \\
&  & \qquad\qquad +\left|k_{1}\theta_{1}-\frac{\xi_{1}^{2}}{4}\right|\int_{0}^{t}\frac{1}{4C_1}\left\Vert \frac{\partial}{\partial y}I_{\epsilon,1}(s,\,\cdot)\right\Vert _{L^{2}\left(\Omega;\,\tilde{L}_{\delta}^2\right)}^{2}ds \nonumber \\
&  & \qquad\qquad+\delta\left|k_{1}\theta_{1}-\frac{\xi_{1}^{2}}{4}\right|\int_{0}^{t}C_2\left\Vert I_{\epsilon,z^{-\frac{1}{2}}}(s,\,\cdot)\right\Vert _{L^{2}\left(\Omega;\,\tilde{L}_{\delta}^2\right)}^{2}ds
\nonumber \\
&  & \qquad\qquad +\delta\left|k_{1}\theta_{1}-\frac{\xi_{1}^{2}}{4}\right|\int_{0}^{t}\frac{1}{4C_2}\left\Vert I_{\epsilon,1}(s,\,\cdot)\right\Vert _{L^{2}\left(\Omega;\,\tilde{L}_{\delta-2}^{2}\right)}^{2}ds
\nonumber \\
&  & \qquad\qquad +k_{1}\left(C_1\int_{0}^{t}\left\Vert I_{\epsilon,z^{\frac{1}{2}}}(s,\,\cdot)\right\Vert _{L^{2}\left(\Omega;\,\tilde{L}_{\delta}^2\right)}^{2}ds+\frac{1}{4C_1}\int_{0}^{t}\left\Vert \frac{\partial}{\partial y}I_{\epsilon,1}(s,\,\cdot)\right\Vert _{L^{2}\left(\Omega;\,\tilde{L}_{\delta}^2\right)}^{2}ds\right)
\nonumber \\
&  & \qquad\qquad +\delta k_{1}\left(C_2\int_{0}^{t}\left\Vert I_{\epsilon,z^{\frac{1}{2}}}(s,\,\cdot)\right\Vert _{L^{2}\left(\Omega;\,\tilde{L}_{\delta-1}^{2}\right)}^{2}ds+\frac{1}{4C_2}\int_{0}^{t}\left\Vert I_{\epsilon,1}(s,\,\cdot)\right\Vert _{L^{2}\left(\Omega;\,\tilde{L}_{\delta-1}^{2}\right)}^{2}ds\right) \nonumber \\
&  & \qquad\qquad
-\int_{0}^{t}\left\langle \frac{\partial}{\partial x}I_{\epsilon,h^{2}(z)}(s,\,\cdot),\,\frac{\partial}{\partial x}I_{\epsilon,1}(s,\,\cdot)\right\rangle _{L^{2}\left(\Omega;\,\tilde{L}_{\delta}^2\right)}ds \nonumber \\
&  & \qquad\qquad +\rho_{1,1}^{2}\int_{0}^{t}\left\Vert \frac{\partial}{\partial x}I_{\epsilon,h(z)}(s,\,\cdot)\right\Vert _{L^{2}\left(\Omega;\,\tilde{L}_{\delta}^2\right)}^{2}ds \nonumber \\
&  & \qquad\qquad+\delta(\delta-1)\frac{\xi_{1}^{2}}{8}\int_{0}^{t}\left\Vert I_{\epsilon,1}(s,\,\cdot)\right\Vert _{L^{2}
\left(\Omega;\,\tilde{L}_{\delta-2}^{2}\right)}^{2}ds  \nonumber \\
&  & \qquad\qquad -\frac{\xi_{1}^{2}}{4}\left(1-\rho_{2,1}^{2}\right)\int_{0}^{t}\left\Vert \frac{\partial}{\partial y}I_{\epsilon,1}(s,\,\cdot)\right\Vert _{L^{2}\left(\Omega;\,\tilde{L}_{\delta}^2\right)}^{2}ds \nonumber \\
& & \qquad \qquad + C_1\left|\rho - \xi_{1}\rho_{3}\rho_{1,1}\rho_{2,1}\right| \int_{0}^{t}\left\Vert \frac{\partial}{\partial x}I_{\epsilon,h(z)}(s,\,\cdot)\right\Vert _{L^{2}\left(\Omega;\,\tilde{L}_{\delta}^2\right)}^{2}ds \nonumber \\
& & \qquad \qquad + \frac{1}{4C_1}\left|\rho - \xi_{1}\rho_{3}\rho_{1,1}\rho_{2,1}\right| \int_{0}^{t}\left\Vert \frac{\partial}{\partial y}I_{\epsilon,1}(s,\,\cdot)\right\Vert _{L^{2}\left(\Omega;\,\tilde{L}_{\delta}^2\right)}^{2}ds \label{eq:5.18}
\end{eqnarray}
and this for any $C_1,\, C_2>0$. If we choose $C_2=1$ and a large enough $C_1$ to have \begin{eqnarray*}\left(\left|k_{1}\theta_{1}-\frac{\xi_{1}^{2}}{4}\right|+\left|\rho - \xi_{1}\rho_{3}\rho_{1,1}\rho_{2,1}\right|+k_{1}\right)\frac{1}{4C_1}<\frac{\xi_{1}^{2}}{4}\left(1-\rho_{2,1}^{2}\right),
\end{eqnarray*}
then from \eqref{eq:5.18} we can obtain the following estimate
\begin{eqnarray}
& & \left\Vert I_{\epsilon,1}(t,\,\cdot)\right\Vert _{L^{2}\left(\Omega;\,\tilde{L}_{\delta}^2\right)}^{2}+M_1\int_{0}^{t}\left\Vert \frac{\partial}{\partial y}I_{\epsilon,1}(s,\,\cdot)\right\Vert _{L^{2}\left(\Omega;\,\tilde{L}_{\delta}^2\right)}^{2}ds \nonumber \\
& & \quad \leq\left\Vert \int_{\mathbb{R}}U_{0}(\cdot,\,z)\phi_{\epsilon}(z,\,\cdot)dz\right\Vert _{L^{2}\left(\Omega;\,\tilde{L}_{\delta}^2\right)}^{2}+M_2\sum_{\alpha\in\left\{ 0,-\frac{1}{2},\frac{1}{2}\right\} }\int_{0}^{t}\left\Vert I_{\epsilon,z^{\alpha}}(s,\,\cdot)\right\Vert _{L^{2}\left(\Omega;\,\tilde{L}_{\delta}^2\right)}^{2}ds \nonumber \\
& & \qquad +\delta M_2\sum_{\alpha\in\{0,\frac{1}{2}\}}\int_{0}^{t}\left\Vert I_{\epsilon,z^{\alpha}}(s,\,\cdot)\right\Vert _{L^{2}\left(\Omega;\,\tilde{L}_{\delta-1}^{2}\right)}^{2}ds \nonumber \\
& & \qquad +\delta M_2\int_{0}^{t}\left\Vert I_{\epsilon,1}(s,\,\cdot)\right\Vert _{L^{2}\left(\Omega;\,\tilde{L}_{\delta-2}^{2}\right)}^{2}ds
+ M_2\int_{0}^{t}\left\Vert \frac{\partial}{\partial x}I_{\epsilon,h(z)}(s,\,\cdot)\right\Vert _{L^{2}\left(\Omega;\,\tilde{L}_{\delta}^2\right)}^{2}ds \nonumber \\
& & \qquad +\int_{0}^{t}\left\langle \frac{\partial}{\partial x}I_{\epsilon,h^{2}(z)}(s,\,\cdot),\,I_{\epsilon,1}(s,\,\cdot)\right\rangle _{L^{2}\left(\Omega;\,\tilde{L}_{\delta}^2\right)}ds \nonumber \\
& & \qquad -\delta\rho\int_{0}^{t}\left\langle \frac{\partial}{\partial x}I_{\epsilon,h\left(z\right)}(s,\cdot),\,I_{\epsilon,1}(s,\cdot)\right\rangle _{L^2\left(\Omega;\, \tilde{L}_{\delta-1}^2 \right)}ds \nonumber \\
& & \qquad -\int_{0}^{t}\left\langle \frac{\partial}{\partial x}I_{\epsilon,h^{2}(z)}(s,\,\cdot),\,\frac{\partial}{\partial x}I_{\epsilon,1}(s,\,\cdot)\right\rangle _{L^{2}\left(\Omega;\,\tilde{L}_{\delta}^2\right)}ds, \label{eq:5.19}
\end{eqnarray}
for some positive constants $M_1$ and $M_2$. Now, for any function $g$, it is easy to check that with the notation of Lemma~\ref{lem:5.3} we have $I_{\epsilon,g(z)}=J_{g\cdot u,\epsilon}$ and $\frac{\partial}{\partial x}I_{\epsilon,g(z)}=J_{g\cdot u_{x},\epsilon}$.
Then, for any $\delta' \in \{\delta, \, \delta - 1, \, \delta - 2\}$ (since $\delta - 2 > -1$), by 2. of Lemma~\ref{lem:5.3} we can compute the limits
of these quantities in $L^{2}\left(\left[0, \, t \right] \times \Omega;\,\tilde{L}_{\delta}^{2}\right)$, which are equal to
\[
J_{g\cdot u}(s,\,x,\,v)=2vg(v^{2})u(s,\,x,\,v^{2})
\]
and 
\[
J_{g\cdot u_{x}}(s,\,x,\,v)=2vg(v^{2})u_{x}(s,\,x,\,v^{2})
\]
respectively, provided that they belong to $L^{2}\left(\left[0, \, t \right] \times \Omega;\,\tilde{L}_{\delta}^{2}\right)$. This can be verified by computing their norms in that space. For $g(z)=g_{1}(z)=z^{\alpha}$ and $g(z)=g_{2}(z)=h^{\beta}(z)$ and for all $0\leq t\leq T$, this computation gives
\begin{eqnarray}
 \left\Vert J_{g_{1}\cdot u}\right\Vert _{L_{y^{\delta'}}^{2}\left(\left[0,\,t\right]\times\Omega\times\mathbb{R}^{+}\times\mathbb{R}^{+}\right)}^{2}&=& 4\int_{0}^{t}\mathbb{E}\left[\int_{\mathbb{R}^+}\int_{\mathbb{R}^{+}}v^{4\alpha+2+\delta'}u^{2}(s,\,x,\,v^{2})dvdx\right]ds \nonumber \\
&=& 2\int_{0}^{t}\mathbb{E}\left[\int_{\mathbb{R}^+}\int_{\mathbb{R}^{+}}y^{2\alpha+\frac{1+\delta'}{2}}u^{2}(s,\,x,\,y)dydx\right]ds
\label{eq:5.20}
\end{eqnarray}
and
\begin{eqnarray}
\left\Vert J_{g_{2}\cdot u_{x}}\right\Vert _{L_{y^{\delta'}}^{2}\left(\left[0,\,t\right]\times\Omega\times\mathbb{R}^{+}\times\mathbb{R}^{+}\right)}^{2}&=& 4\int_{0}^{t}\mathbb{E}\left[\int_{\mathbb{R}^+}\int_{\mathbb{R}^{+}}v^{2+\delta'}h^{2\beta}\left(v^{2}\right)u^{2}(s,\,x,\,v^{2})dvdx\right]ds \nonumber \\
&=& 2\int_{0}^{t}\mathbb{E}\left[\int_{\mathbb{R}^+}\int_{\mathbb{R}^{+}}y^{\frac{1+\delta'}{2}}h^{2\beta}\left(y\right)u^{2}(s,\,x,\,y)dydx\right]ds, \nonumber \\
\label{eq:5.21}
\end{eqnarray}
which are both finite for all the combinations of $\alpha$, $\beta$ and $\delta'$ appearing in the norm terms of the RHS of \eqref{eq:5.19}, since it is easy to verify then that the exponent of $y$ in the RHS of both \eqref{eq:5.20} and \eqref{eq:5.21} is positive, and since $h$ has polynomial growth. Hence, 2. of Lemma~\ref{lem:5.3} and the continuity of the inner products imply that all the terms in the RHS of \eqref{eq:5.19} are convergent as $\epsilon\rightarrow 0^{+}$. Therefore, the RHS of \eqref{eq:5.19} is also bounded in $\epsilon$ and thus, Lemma~\ref{lem:5.4} applied on the $y$-derivative term in the LHS of that estimate implies that $\frac{\partial}{\partial v}J_{u}=\left(2vu(s,\,x,\,v^{2})\right)_{v}$
exists in $L^{2}\left(\left[0, \, t \right] \times \Omega;\,\tilde{L}_{\delta}^{2}\right)$ and that in this space we have $\frac{\partial}{\partial v}I_{\epsilon,1} = \frac{\partial}{\partial v}J_{u,\epsilon} \rightarrow \frac{\partial}{\partial v}J_u$ as $\epsilon \rightarrow 0^{+}$.

From the above we can easily deduce that $\left(u\left(s,\,x,\,v^{2}\right)\right)_{v}$
exists, which implies that the weak derivative $u_{y}(s,\,x,\,y)=\frac{1}{2\sqrt{y}}\frac{\partial}{\partial v}u\left(s,\,x,\,v^{2}\right)|{}_{v=\sqrt{y}}$ also
exists. Moreover, by using the standard inequality $(a - b)^2 \leq 2(a^2 + b^2)$ we have
\begin{eqnarray}
& & \int_{0}^{t}\mathbb{E}\left[\int_{\mathbb{R}^+}\int_{\mathbb{R}^{+}}y^{\frac{\delta+3}{2}}u_{y}^{2}(s,\,x,\,y)dydx\right]ds \nonumber \\
&& \qquad=\frac{1}{2}\int_{0}^{t}\mathbb{E}\left[\int_{\mathbb{R}^+}\int_{\mathbb{R}^{+}}v^{\delta+2}\left(\left(u\left(s,\,x,\,v^{2}\right)\right)_{v}\right)^{2}dvdx\right]ds \nonumber \\
&&  \qquad = \frac{1}{2}\int_{0}^{t}\mathbb{E}\left[\int_{\mathbb{R}^+}\int_{\mathbb{R}^{+}}v^{\delta}\left(\left(vu\left(s,\,x,\,v^{2}\right)\right)_{v}-u\left(s,\,x,\,v^{2}\right)\right)^{2}dvdx\right]ds \nonumber \\
& & \qquad \leq \frac{1}{2}\int_{0}^{t}\left\Vert \frac{\partial}{\partial y}J_{u}(s,\,\cdot)\right\Vert _{L_{y^{\delta}}^{2}\left(\Omega\times\mathbb{R}^{+}\times\mathbb{R}^{+}\right)}^{2}ds \nonumber \\
& & \qquad\qquad +\frac{1}{2}\int_{0}^{t}\mathbb{E}\left[\int_{\mathbb{R}^+}\int_{\mathbb{R}^{+}}y^{\frac{\delta-1}{2}}u^{2}\left(s,\,x,\,y\right)dydx\right]ds, \label{eq:5.22}
\end{eqnarray}
which is clearly finite since $\frac{\partial}{\partial v}J_{u} \in L^{2}\left(\left[0, \, t \right] \times \Omega;\,\tilde{L}_{\delta}^{2}\right)$. This gives the regularity result for weight exponents $\alpha = \frac{\delta+3}{2} \in \left(2, +\infty \right)$. Observe however that the limit of the RHS of \eqref{eq:5.19} as $\epsilon \rightarrow 0^{+}$ gives a bound for the first summand of the RHS of \eqref{eq:5.22} consisting of weighted $L^2$ norms of $u$ with weights $y^{\delta}$, $y^{\delta - 1}$ and $y^{\delta - 2}$ for $\delta > 1$. Then, by our regularity assumptions for $u$, we see that each of these norms is also finite for $\delta = 1$, and the same holds for the second summand of the RHS of \eqref{eq:5.22}. Hence, by using the Dominated Convergence Theorem, we see that the LHS of \eqref{eq:5.22} is bounded as $\delta \rightarrow 1^{+}$ and thus, it is also finite for $\delta = 1$ (by Fatou's lemma for example), which implies the desired regularity result also for $\alpha = 2$.
\end{proof}

\vspace*{.05in}

\begin{rem}
Observe that the smoothed quantities $J_{u,\epsilon}(\lambda,y)$ do not have to vanish as $y\to 0$, even 
though their limits $J_u(\lambda, y)$ decay linearly in $y$ near zero, so the integral of $J_{u,\epsilon}(\lambda,y)$ against 
$y^{\delta'}$ can explode at zero for $\delta' \leq -1$. It follows that Lemma~\ref{lem:5.3} does not work 
for $\delta' = -1$, since the weighted norms of the smoothed quantities can be infinite, while those 
of their limits are finite. This is why we had to work with the $\delta$-identity for $\delta > 1$ 
(implying $\delta' \geq \delta - 2 > -1$ wherever Lemma~\ref{lem:5.3} is used) and take $\delta 
\rightarrow 1^{+}$ only after taking $\epsilon \rightarrow 0^{+}$.
\end{rem}

\begin{rem}\label{rem:5.7} The flexibility in the choice of $\rho$ allows us to extend our results to the case where the idiosyncratic noises have nonzero correlation. Indeed, suppose that for any $i \geq 1$ we have $W^{i}_t = w_i\tilde{W}^{i}_t + \sqrt{1 - w^2_i}Z^i_t$ and $B^{i}_t = b_i\tilde{B}^{i}_t + \sqrt{1 - b^2_i}Z^i_t$, where $\tilde{W}_{\cdot}^{i}$, $\tilde{B}_{\cdot}^{i}$ and $Z_{\cdot}^i$ are pairwise independent standard Brownian Motions, and $w_i, b_i \in [-1, 0) \cup (0, 1]$. Then, we can obtain the convergence results of section 2 in exactly the same way, and the SPDE we obtain is the one treated in the previous section with 
\begin{eqnarray}
\rho = \xi_{1}\rho_{3}\rho_{1,1}\rho_{2,1} + \xi_{1}\sqrt{1 - \rho^{2}_{1,1}}\sqrt{1 - \rho^{2}_{2,1}}\sqrt{1 - w^{2}_{1}}\sqrt{1 - b^{2}_{1}} \label{eq:2017}
\end{eqnarray}
The extension will be complete if we manage to embed the measure-valued process $v_{t,C_1}$ in $L_{\alpha, \, C_1} \cap H_{\alpha, \, C_1}$ for all $\alpha \geq 0$ and for a given value of $C_1$, as we have done in Section~4 for the zero correlation case. Since $v_{t,C_1}$ can be expressed as a conditional law of the pair $(X_{\cdot}^1, \, \sigma_{\cdot}^1)$ as in the zero correlation case, this embedding can be done by conditioning on $Z_{\cdot}^1$ to reduce the problem to the zero correlation case, with $\sqrt{1 - \rho^{2}_{1,1}}W_{\cdot}^1$, $\sqrt{1 - \rho^{2}_{2,1}}B_{\cdot}^1$, $\rho_{1,1}W_{\cdot}^0$ and $\rho_{2,1}B_{\cdot}^0$ replaced by $w_1\sqrt{1 - \rho^{2}_{1,1}}\tilde{W}_{\cdot}^{1}$, $b_1\sqrt{1 - \rho^{2}_{2,1}}\tilde{B}_{\cdot}^{1}$, $\rho_{1,1}W_{\cdot}^0 + \sqrt{1 - w^{2}_1}Z_{\cdot}^1$ and $\rho_{2,1}B_{\cdot}^0 + \sqrt{1 - b^{2}_1}Z_{\cdot}^1$ respectively. This approach obviously fails when $w_1 = 0$ or $b_1 = 0$.
\end{rem}


\section{Discussion of Uniqueness}

The previous sections have established existence and regularity results for this class of stochastic volatility models arising from large portfolios. We would also like to prove that our problem has always a unique solution for a fixed coefficient vector $C_{1}$ and a fixed initial data function $U_{0}$. However the bad behaviour of the coefficients of the SPDE near zero render all the standard approaches to the question of uniqueness inapplicable.

Indeed, by following the same steps as in the proof of the $\delta$-identity but for $\delta = 0$, without integrating in $y$, and by using the product rule instead of integrating by parts, we can obtain $L_{\epsilon}(t, \, y) = R_{\epsilon}(t, \, y)$ where
\begin{eqnarray}
& & L_{\epsilon}(t, \, y) \nonumber \\
& & \qquad = \frac{\xi_{1}^{2}}{8} \left( \int_{0}^{t}\left\Vert I_{\epsilon,1}(s, \, \cdot, \, y)\right\Vert _{L^2 \left( \Omega \times \mathbb{R}^{+} \right)}^{2}ds \right)_{yy} \nonumber \\
& & \qquad \qquad +\rho \left( \int_{0}^{t}\left\langle \frac{\partial}{\partial x}I_{\epsilon,h\left(z\right)}(s, \, \cdot, \, y),\,I_{\epsilon,1}(s, \, \cdot, \, y)\right\rangle _{L^2 \left( \Omega \times \mathbb{R}^{+} \right)}ds \right)_{y} \nonumber \\
& & \qquad \qquad -\left(k_{1}\theta_{1}-\frac{\xi_{1}^{2}}{4}\right)\left(\int_{0}^{t}\left\langle I_{\epsilon,z^{-\frac{1}{2}}}(s, \, \cdot, \, y),I_{\epsilon,1}(s, \, \cdot, \, y)\right\rangle _{L^2 \left( \Omega \times \mathbb{R}^{+} \right)}ds\right)_y \nonumber \\
& & \qquad \qquad +k_{1}\left(\int_{0}^{t}\left\langle I_{\epsilon,z^{\frac{1}{2}}}(s, \, \cdot, \, y),I_{\epsilon,1}(s, \, \cdot, \, y)\right\rangle _{L^2 \left( \Omega \times \mathbb{R}^{+} \right)}ds\right)_y \label{eq:5.11u2}
\end{eqnarray}
and
\begin{eqnarray}
& & R_{\epsilon}(t, \, y) \nonumber \\
& & \qquad = \left\Vert I_{\epsilon,1}(t,\,\cdot, \, y)\right\Vert _{L^2 \left( \Omega \times \mathbb{R}^{+} \right)}^{2} - \left\Vert \int_{\mathbb{R}}U_{0}(\cdot,z)
\phi_{\epsilon}(z,\, y)dz\right\Vert _{L^2 \left( \Omega \times \mathbb{R}^{+} \right)}^{2} \nonumber \\
& & \qquad \qquad -\int_{0}^{t}\left\langle \frac{\partial}{\partial x}I_{\epsilon,h^{2}(z)}(s, \, \cdot, \, y),I_{\epsilon,1}(s,\, \cdot, \, y)\right\rangle _{L^2 \left( \Omega \times \mathbb{R}^{+} \right)}ds \nonumber \\
& & \qquad \qquad -\left(k_{1}\theta_{1}-\frac{\xi_{1}^{2}}{4}\right)\int_{0}^{t}\left\langle I_{\epsilon,z^{-\frac{1}{2}}}(s, \, \cdot, \, y),\frac{\partial}{\partial y}I_{\epsilon,1}(s, \, \cdot, \, y)\right\rangle _{L^2 \left( \Omega \times \mathbb{R}^{+} \right)}ds \nonumber \\
& & \qquad \qquad +k_{1}\int_{0}^{t}\left\langle I_{\epsilon,z^{\frac{1}{2}}}(s, \, \cdot, \, y),\frac{\partial}{\partial y}I_{\epsilon,1}(s, \, \cdot, \, y)\right\rangle _{L^2 \left( \Omega \times \mathbb{R}^{+} \right)}ds \nonumber \\
& & \qquad \qquad +\int_{0}^{t}\left\langle \frac{\partial}{\partial x}I_{\epsilon,h^{2}(z)}(s, \, \cdot, y),\,\frac{\partial}{\partial x}I_{\epsilon,1}(s,\, \cdot, y)\right\rangle _{L^2 \left( \Omega \times \mathbb{R}^{+} \right)}ds \nonumber \\
& & \qquad \qquad -\left(\xi_{1}\rho_{3}\rho_{1,1}\rho_{2,1} - \rho \right)\int_{0}^{t}\left\langle \frac{\partial}{\partial x}I_{\epsilon,h\left(z\right)}(s, \, \cdot, \, y),\,\frac{\partial}{\partial y}I_{\epsilon,1}(s, \, \cdot, \, y)\right\rangle _{L^2 \left( \Omega \times \mathbb{R}^{+} \right)}ds \nonumber \\
& & \qquad \qquad -\rho_{1,1}^{2}\int_{0}^{t}\left\Vert \frac{\partial}{\partial x}
I_{\epsilon,h(z)}(s, \, \cdot, \, y)\right\Vert _{L^2 \left( \Omega \times \mathbb{R}^{+} \right)}^{2}ds \nonumber \\ 
& & \qquad \qquad -\frac{\xi_{1}^{2}}{4}\left(\rho_{2,1}^{2} - 1 \right)\int_{0}^{t}\left\Vert \frac{\partial}{\partial y}
I_{\epsilon,1}(s, \, \cdot, \, y)\right\Vert _{L^2 \left( \Omega \times \mathbb{R}^{+} \right)}^{2}ds. \label{eq:5.11u3}
\end{eqnarray}
where we can use our regularity result to compute the limit of each term in $R_{\epsilon}(t, \, y)$, in an $L_{loc}^1$ sense as a function of $y$ and for any $t \geq 0$. Using this, we can deduce the convergence of each term in $L_{\epsilon}(t, \, y)$ in the same sense, which implies that the function $E(t, \, y) := \int_{0}^{t}\left\Vert 2yu(s, \, \cdot, \, y^2)\right\Vert _{L^2 \left( \Omega \times \mathbb{R}^{+} \right)}^{2}ds$ also has a locally integrable second derivative in $y$. Then, we can take $\epsilon \rightarrow 0^{+}$ on $L_{\epsilon}(t, \, y) = R_{\epsilon}(t, \, y)$ and substitute the limit of each term to obtain 
\begin{eqnarray}
E(t, \, y) &=& E(0, \, y) + \frac{{\xi_{1}}^2}{8}E_{yy}(t, \, y) - \left(k_{1}\theta_{1}-\frac{\xi_{1}^{2}}{4}\right)\frac{1}{2y}E_y(t, \, y) + \frac{k_1y}{2}E_y(t, \, y) \nonumber \\
& & \qquad + \left(k_{1}\theta_{1}-\frac{\xi_{1}^{2}}{4}\right)\frac{1}{y^2}E(t, \, y) + k_1E(t, \, y) \nonumber \\
& & \qquad -\left(1 - \rho_{1,1}^{2} \right)h^2(y^2)\int_{0}^{t}\left\Vert \frac{\partial}{\partial x}\left(2yu(s,\, \cdot, y^2)\right)\right\Vert _{L^2 \left( \Omega \times \mathbb{R}^{+} \right)}^{2}ds \nonumber \\
& & \qquad +\left(\xi_{1}\rho_{3}\rho_{1,1}\rho_{2,1} - \rho \right)h(y^2) \nonumber \\
& & \qquad \qquad \times \int_{0}^{t}\left\langle \frac{\partial}{\partial x}\left(2yu(s, \, \cdot, \, y^2)\right),\,\frac{\partial}{\partial y}\left(2yu(s, \, \cdot, \, y^2)\right)\right\rangle _{L^2 \left( \Omega \times \mathbb{R}^{+} \right)}ds \nonumber \\
& & \qquad -\frac{\xi_{1}^{2}}{4}\left(1 - \rho_{2,1}^{2} \right)\int_{0}^{t}\left\Vert \frac{\partial}{\partial y}
\left(2yu(s, \, \cdot, \, y^2)\right)\right\Vert _{L^2 \left( \Omega \times \mathbb{R}^{+} \right)}^{2}ds \label{eq:5.11u8}
\end{eqnarray}
where we can assume that $|\rho - \xi_{1}\rho_{3}\rho_{1,1}\rho_{2,1}| \leq \xi_{1}\sqrt{1 - \rho^{2}_{1,1}}\sqrt{1 - \rho^{2}_{2,1}}$, a condition obviously satisfied when $\rho = \xi_{1}\rho_{3}\rho_{1,1}\rho_{2,1}$, and then apply the Cauchy-Schwarz and AM-GM inequalities to show that the sum of the last three terms is negative. This implies that $E(t, \, y)$ satisfies:
\begin{eqnarray}
E_t(t, \, y) &\leq& E_t(0, \, y) + \frac{{\xi_{1}}^2}{8}E_{yy}(t, \, y) - \left(k_{1}\theta_{1}-\frac{\xi_{1}^{2}}{4}\right)\frac{1}{2y}E_y(t, \, y) + \frac{k_1y}{2}E_y(t, \, y) \nonumber \\
& & \qquad + \left(k_{1}\theta_{1}-\frac{\xi_{1}^{2}}{4}\right)\frac{1}{y^2}E(t, \, y) + k_1E(t, \, y) \label{eq:5.11u10}
\end{eqnarray}
under the boundary condition $E(t, \, 0) = 0$ for all $t \geq 0$, with the first order derivatives in $y$ being continuous classical derivatives (this follows from standard 1-dimensional Sobolev embeddings). This seems to be the best possible result we can have for the energy of a solution to our initial-boundary value problem, since all the norm estimates that can be obtained from Theorem~\ref{thm:5.2} can also be obtained by integrating \eqref{eq:5.11u10} against some power of $y$. Since the problem is linear, uniqueness follows if we can show that $E$ must vanish everywhere when we have zero initial data (which is equivalent to $E_t(0, \, y) = 0$ for all $y \geq 0$). However, this is an open problem as standard approaches to problems of this kind fail due to the unboundedness of the coefficient of the non-derivative term.

\vspace*{.05in}

A possible approach to the above problem would be to multiply \eqref{eq:5.11u10} (for $E_t(0, \, \cdot) = 0$) by some positive function of $y$ and integrate in $[0, \, +\infty)$, hoping to obtain an estimate where Gronwall's Lemma can be applied to give the desired result. This seems to fail since it leads to estimates involving different weighted norms of $E$, which are always non-equivalent due to the unboundedness of $\frac{1}{y^2}$ near zero. 

Another approach would be to try to use an argument like the standard parabolic maximum principle, i.e to obtain a zero maximum for the positive function $e^{g(y)t}E(t, \, y)$ by choosing a function $g$ that helps in the elimination of non-derivative terms in \eqref{eq:5.11u10}, and by recalling that when the maximum of a function is not attained at the boundary, the first order derivatives vanish and the second order ones are non-positive. Once more, the unboundedness of $\frac{1}{y^2}$ near zero does not allow for $g$ to be bounded, which causes extra problems as one can easily check.

Finally, if we try to implement either of the above approaches in the domain $[\epsilon, \, +\infty)$ for small $\epsilon > 0$, where the coefficient of the non-derivative term in \eqref{eq:5.11u10} is bounded, and then try to take $\epsilon \rightarrow 0^{+}$, we will see that the desired result can be obtained only when $E(t, \, y) = \mathcal{O}(e^{-\frac{c}{y^2}})$ near zero, for some $c > 0$. Of course, this is something we cannot expect since our CIR density does not vanish faster than $y^{\frac{2k_1\theta_{1}}{\xi_1^{2}}-1}$ as $y \rightarrow 0^{+}$.

\vspace*{.05in}

\begin{rem} 
The estimate \eqref{eq:5.11u10} can also be obtained in the case where the idiosyncratic Brownian Motions have correlation as in Remark~\ref{rem:5.7}, since by \eqref{eq:2017} we have
\begin{eqnarray*}
|\rho - \xi_{1}\rho_{3}\rho_{1,1}\rho_{2,1}| &=& \xi_{1}\sqrt{1 - \rho^{2}_{1,1}}\sqrt{1 - \rho^{2}_{2,1}}\sqrt{1 - w^{2}_{1}}\sqrt{1 - b^{2}_{1}} \nonumber \\
&& \leq \xi_{1}\sqrt{1 - \rho^{2}_{1,1}}\sqrt{1 - \rho^{2}_{2,1}}.
\end{eqnarray*}
a condition necessary for obtaining that estimate.
\end{rem}

\vspace*{.05in}

{\flushleft{\textbf{Acknowledgement}}\\[.1in]
The second author's work was supported financially by the United Kingdom Engineering and Physical Sciences Research Council {[}EP/L015811/1{]}}, and by the Foundation for Education and European Culture in Greece (founded by Nikos \& Lydia Tricha).

\appendix
\section{APPENDIX: Proofs of standard and technical results}

\begin{proof}[\textbf{Proof of Theorem 2.1}]
First, we consider each pair $\left(X_{.}^{i},\,\sigma_{.}^{i}\right)$
as a random variable taking values in the probability space $\left(C\left(\left[0,\,T\right];\,\mathbb{R}^{2}\right),\,||\centerdot||_{\infty},\,\mathcal{B}\right)$,
which is the space of continuous $\mathbb{R}^{2}$-valued functions
defined on $\left[0,\,T\right]$, equipped with the supremum norm
$||\centerdot||$ and the appropriate $\sigma$-algebra $\mathcal{B}$.

Since $\left[0,\,T\right]$ is a compact subinterval of
$\mathbb{R}$, $\mathcal{B}$ coincides with the usual
$\sigma$-algebra for the law of the $\mathcal{F}_{t}$-adapted process
$\left(X_{.}^{i},\,\sigma_{.}^{i}\right)$. Moreover, there is a function $S$ such that for each $\omega\in\Omega$
we can write
\[
\left(X_{.}^{i}(\omega),\,\sigma_{.}^{i}(\omega)\right)=S\left(B_{.}^{i}(\omega),\,W_{.}^{i}(\omega),\,B_{.}^{0}(\omega),\,W_{.}^{0}(\omega),\,x^{i}(\omega),\,\sigma_0^{i}(\omega),\,C_{i}(\omega)\right),
\]
since $\left(X_{.}^{i},\,\sigma_{.}^{i}\right)$ is obviously a strong
solution to \eqref{eq:model}.

For a permutation $\pi:\{ 1,2,...,N\} \to \{ 1,2,...,N\} $
and a collection $\left\{ G_{1},G_{2},...,G_{N}\right\} $ of
$\mathcal{B}$-measurable sets, the event   
\[
\left\{ \left(B_{.}^{\pi(i)},\,W_{.}^{\pi(i)},\,B_{.}^{0},\,W_{.}^{0},\,x^{\pi(i)},\,\sigma^{\pi(i)},\,C_{\pi(i)}\right)\in S^{-1}(G_{i}),\:\forall\,1\leq i\leq N\right\} 
\]
has a probability which is equal to 
\[
\mathbb{P}\left(\omega\in\Omega:\:\left(X_{.}^{\pi(i)}(\omega),\,\sigma_{.}^{\pi(i)}(\omega)\right)\in G_{i},\:\forall\,1\leq i\leq N\right).
\]
We claim that the law
\begin{eqnarray*}
&& \mathbb{P}\left(\omega\in\Omega:\:\left(X_{.}^{\pi(i)}(\omega),\,\sigma_{.}^{\pi(i)}(\omega)\right)\in G_{i},\:\forall\,1\leq i\leq N\right)\\
&& \qquad = \mathbb{P}\left(\left(B_{.}^{\pi(i)},\,W_{.}^{\pi(i)},\,B_{.}^{0},\,W_{.}^{0},\,x^{\pi(i)},\,\sigma_0^{\pi(i)},\,C_{\pi(i)}\right)\in S^{-1}\left(G_{i}\right),\:\forall\,1\leq i\leq N\right)
\end{eqnarray*}
is independent of the permutation $\pi$. Indeed, by a linear inversion
it is enough to show that the joint law of $\left\{ B_{.}^{\pi(i)},\,W_{.}^{\pi(i)},\,x^{\pi(i)},\,\sigma_0^{\pi(i)},\,C_{\pi(i)}:\:1\leq i\leq N\right\} 
\cup\left\{ B_{.}^{0},\,W_{.}^{0}\right\} $
is independent of the permutation $\pi\left(\cdot\right)$, which is a consequence of our exchangeability assumptions. 
As a result the set $\left\{ \left(X_{.}^{i},\,\sigma_{.}^{i}\right):\:1\leq i\leq N\right\} $
is an exchangeable set of $C\left(\left[0,\,T\right];\,\mathbb{R}^{2}\right)$-valued
random variables. Hence, by de Finetti's Theorem (see Theorem 4.1 in \cite{KKP}, but it can also be found in \cite{Aldous}), we
obtain that the sequence of measure-valued processes
\[
v_{*}^{N}={\displaystyle \frac{1}{N}\sum_{i=1}^{N}\delta_{X_{.}^{i},\,\sigma_{.}^{i}}}
\]
converges weakly to some probability measure $v_{*}$ (which is defined
on $\mathcal{B}$), $\mathbb{P}$-almost surely. Thus there is a set $\Omega'\subset\Omega$ where the
convergence is valid for any $\omega\in\Omega'$, where
$\mathbb{P}(\Omega')=1$.

Let $P_{t,s}:\,\left(C\left(\left[0,\,T\right];\,\mathbb{R}^{2}\right),\,||\centerdot||_{\infty},\,\mathcal{B}\right)\longrightarrow\mathbb{R}^{3}$
be an evaluation functional at some $(t,\,s)\in\left[0,\,T\right]^{2}$,
which maps $\left(f(\cdot),\,g(\cdot)\right)$ to $\left(f(t),\,g(t),\,g(s)\right)$
and which is obviously continuous. We fix an $\omega\in\Omega'$ and
we define $v_{3,t,s}=v_{*}\circ P_{t,s}^{-1}$ for all $\left(t,\,s\right)\in[0,\,T]^{2}$.
Then, for this $\omega$ and for any Borel set $A\subset\mathbb{R}^{3}$
we have
\begin{eqnarray*}
v_{3,t,s}^{N}(A) &=& \frac{1}{N}\#\left\{ 1\leq i\leq N:\,\left(X_{t}^{i},\,\sigma_{t}^{i},\,\sigma_{s}^{i}\right)\in A\right\} 
\\
&=& \frac{1}{N}\#\left\{ 1\leq i\leq N:\,P_{t,s}\left(X_{.}^{i},\,\sigma_{.}^{i}\right)\in A\right\} \\
&=& \frac{1}{N}\#\left\{ 1\leq i\leq N:\,\left(X_{.}^{i},\,\sigma_{.}^{i}\right)\in P_{t,s}^{-1}(A)\right\} =v_{*}^{N}(P_{t,s}^{-1}(A))
\end{eqnarray*}
for all $N\in\mathbb{N}$ and all $\left(t,\,s\right)\in[0,\,T]^{2}$. This means that for this $\omega$ and
for any $f\in C_{b}\left(\mathbb{R}^{3};\,\mathbb{R}\right)$ we have
\begin{equation}\label{eq:A.1}
\int_{\mathbb{R}^{2}}f dv_{3,t,s}^{N}=\int_{C\left(\left[0,\,T\right];\,\mathbb{R}^{2}\right)}f\circ P_{t,s} dv_{*}^{N},
\end{equation}
since we can easily show this for a sequence of simple functions approximating
$f$ from below and conclude then by the Monotone Convergence Theorem.
Taking now $N\rightarrow\infty$ we find
\begin{equation}\label{eq:A.2}
\int_{\mathbb{R}^{2}}fdv_{3,t,s}^{N}\rightarrow\int_{C\left(\left[0,\,T\right];\,\mathbb{R}^{2}\right)}f\circ P_{t,s}\cdot dv_{*}=\int_{\mathbb{R}^{2}}f dv_{3,t,s},
\end{equation}
for any $f\in C_{b}\left(\mathbb{R}^{3};\,\mathbb{R}\right)$, where
the last equality in \eqref{eq:A.2} is obtained exactly as \eqref{eq:A.1}. Since
this holds for any $f\in C_{b}\left(\mathbb{R}^{3};\,\mathbb{R}\right)$,
we have the desired convergence. 

Finally, to show continuity under the weak topology for
a given $\omega\in\Omega'$, we shall invoke the Portmanteau Theorem,
according to which we only need to show that
\[ \liminf_{n\rightarrow\infty}v_{3,t_{n},s_{n}}(A)\geq v_{3,t,s}(A) \]
whenever $\left(t_{n},\,s_{n}\right)\rightarrow\left(t,\,s\right)\in\left[0,\,T\right]^{2}$
and for any open $A\subset\mathbb{R}^{3}$. This is obtained by observing
that
\begin{eqnarray*}
v_{3,t,s}(A) &=& v_{*}\left(P_{t,s}^{-1}(A)\right) \\
&=& v_{*}\left(\left\{ \left(Y_{.},\,Z_{.}\right)\in C\left(\left[0,\,T\right];\,\mathbb{R}^{2}\right):\,\left(Y_{t},\,Z_{t},\,Z_{s}\right)\in A\right\} \right) \\
&=&  v_{*}\left( \cup_{k=1}^{\infty}\cap_{n=k}^{\infty}\left\{ \left(Y_{.},\,Z_{.}\right)\in C\left(\left[0,\,T\right];\,\mathbb{R}^{2}\right):\,\left(Y_{t_{n}},\,Z_{t_{n}},\,Z_{s_{n}}\right)\in A\right\} \right),
\end{eqnarray*}
which holds because $\left(Y_{t_{n}},\,Z_{t_{n}},\,Z_{s_{n}}\right)\rightarrow\left(Y_{t},\,Z_{t},\,Z_{s}\right)$
by the continuity of the path $\left(Y_{.},\,Z_{.}\right)$, and hence
$\left(Y_{t_{n}},\,Z_{t_{n}},\,Z_{s_{n}}\right)$ is finally contained
in any open set containing $\left(Y_{t},\,Z_{t},\,Z_{s}\right)$.
Then, the last quantity is equal to
\begin{eqnarray*}
& & \lim_{k\rightarrow\infty} v_{*}\left(\cap_{n=k}^{\infty}\left\{ \left(Y_{.},\,Z_{.}\right)\in C\left(\left[0,\,T\right];\,\mathbb{R}^{2}\right):\,\left(Y_{t_{n}},\,Z_{t_{n}},\,Z_{s_{n}}\right)\in A\right\} \right) \\
& & \qquad\leq \lim_{k\rightarrow\infty}\inf_{n\geq k} v_{*}\left(\left\{ \left(Y_{.},\,Z_{.}\right)\in C\left(\left[0,\,T\right];\,\mathbb{R}^{2}
\right):\,\left(Y_{t_{n}},\,Z_{t_{n}},\,Z_{s_{n}}\right)\in A\right\} \right) \\
&& \qquad=\liminf_{n\rightarrow\infty}v_{*}\left(P_{t_{n},s_{n}}^{-1}(A)\right)\\
&& \qquad=\liminf_{n\rightarrow\infty}v_{3,t_{n},s_{n}}(A)
\end{eqnarray*}
and the desired continuity has been proven. Since this continuous
limit process of measures is obtained almost surely for $(t,\,s)\in[0,\,n]^{2}$,
for any $n\in\mathbb{N}$, with $\mathbb{N}$ being countable, it
is actually obtained almost surely for all $t,\,s\geq0$. The proof
of the Theorem is now complete.
\end{proof}

\begin{proof}[\textbf{Proof of Lemma 3.2}]
Observe that we only need to prove our claim for $p=n\in\mathbb{N}$. By Ito's formula we have
\begin{eqnarray*}
\sigma_{t}^{n}&=& \sigma_{0}^{n}+\int_{0}^{t}n\left(\sigma_{s}^{n-1}k\left(\theta-\sigma_{s}\right)+\frac{\xi^{2}}{2}\left(n-1\right)\sigma_{s}^{n-1}\right)ds
\\
&&\qquad +n\xi\int_{0}^{t}\sigma_{s}^{n-1/2}d\left(\sqrt{1-\rho_{2}^{2}}B_{s}^{1}+\rho_{2}B_{s}^{0}\right)
\\
&\leq& \sigma_{0}^{n}+C_{1}T+n\xi\int_{0}^{t}\sigma_{s}^{n-1/2}d\left(\sqrt{1-\rho_{2}^{2}}B_{s}^{1}+\rho_{2}B_{s}^{0}\right)
\end{eqnarray*}
for some $C_{1}>0$ when $t\leq T$, since the quantity within the
Riemann integral is a polynomial of a negative leading coefficient (thus upper bounded in the positive reals), computed at the CIR process $\sigma_{s}$ which is always non-negative (as we pointed out before introducing our model).
Taking supremum for $t\leq T$, then taking expectations and finally
using Cauchy-Schwartz and Doob's inequalities, we obtain:
\begin{eqnarray*}
\mathbb{E}\left[M_{T}^{n}\right] &\leq & \mathbb{E}\left[\sigma_{0}^{n}\right]+C_{1}T
+n\xi\mathbb{E}^{\frac{1}{2}}\left[\sup_{0\leq t\leq T}\left(\int_{0}^{t}\sigma_{s}^{n-1/2}d\left(\sqrt{1-\rho_{2}^{2}}B_{s}^{1}+\rho_{2}B_{s}^{0}\right)\right)^{2}\right] \\
&\leq &\mathbb{E}\left[\sigma_{0}^{n}\right]+C_{1}T+n\xi\mathbb{E}^{\frac{1}{2}}\left[\int_{0}^{T}\sigma_{s}^{2n-1} ds\right] \\
&\leq &\mathbb{E}\left[\sigma_{0}^{n}\right]+C_{1}T+C_{2}\int_{0}^{T}\mathbb{E}^{\frac{1}{2}}\left[\sigma_{s}^{2n-1}\right]ds,
\end{eqnarray*}
for some $C_{2}>0$, where we have set $M_{T}={\displaystyle \sup_{t\leq T}\sigma_{t}^{n}}$.
The first expectation of the RHS of the last equation is finite by
our assumptions for the initial data. To obtain the desired result
for the CIR process, it suffices to show that the expectation within
the last Riemann integral is bounded for $0\leq s\leq T$. For this,
we recall Theorem~\ref{thm:3.1} and Remark 2 from pages 8-9 in \cite{HK06}, from
which we can easily obtain (after conditioning on the initial value)
\[
\mathbb{E}\left[\sigma_{s}^{2n-1}\right]\leq C_{3}\sum_{k=0}^{2n-1}\gamma_{s}^{k-2n+1}\mathbb{E}\left[\sigma_{0}^{k}\right],
\]
for all $0\leq s\leq T$ and some $C_{3}>0$, where $\gamma_{s}=\frac{2k}{\xi^{2}}\left(1-e^{-ks}\right)^{-1}$.
The RHS of the above inequality is bounded for $0<s\leq T$, since
$\sigma_{0}$ has bounded moments and since $\gamma_{s}>\frac{2k}{\xi^{2}}>0$
for all $0<s\leq T$.
 
Finally, the desired result for $\left\{ u_{t}^{2}:\,t\geq0\right\}$
can be obtained in a much easier way, since we have an explicit formula for the Ornstein-Uhlenbeck process. Indeed, by using this formula we can control the maximum of the process by $\sqrt{\sigma^{0}}$ an by the maximum of a Brownian Motion in $\left[0, \, T\right]$ (up to a constant factor), where the last is normally distributed and thus it has a finite second moment. The proof of the Lemma is now complete.
\end{proof}

\begin{proof}[\textbf{Proof of Lemma 3.3}]
First we set $v_{t}=\sqrt{\sigma_{t}}$ and, as our assumptions ensure that 
$\sigma$ does not hit 0, we can apply Ito's formula to equation \eqref{eq:3.1} to obtain
\begin{equation}\label{eq:3.3}
dv_{t}=\left[\left(\frac{k\theta}{2}-\frac{\xi^{2}}{8}\right)\frac{1}{v_{t}}-\frac{k}{2}v_{t}\right]dt+\frac{\xi}{2}dB_{t},
\end{equation}
where $B_{t}:=\sqrt{1-\rho_{2}^{2}}B_{t}^{1}+\rho_{2}B_{t}^{0}$ is
a standard Brownian Motion. Since the CIR process is an $L^{1}$-integrable
process (this follows from Lemma~\ref{lem:3.2}), $v_{t}$ is an $L^{2}$-integrable
process. Consider now for any $\epsilon>0$, a twice continuously
differentiable and increasing cut-off function $\Phi^{\epsilon}(x)$
satisfying
\[
\Phi^{\epsilon}(x)=\begin{cases}
1 & \mbox{if} \:\:x\geq2\epsilon,\\
0 & \mbox{if} \:\:x<\epsilon.
\end{cases}
\]
Then the derivative satisfies
\[
\frac{\partial}{\partial x}\Phi^{\epsilon}(x)=\begin{cases}
0 & \mbox{if}\:\:x\geq2\epsilon,\\
0 & \mbox{if}\:\:x<\epsilon.
\end{cases}
\]
Moreover, we define: $J^{\epsilon}(x)=\frac{\Phi^{\epsilon}(x)}{x}$
for $x>0$ and $J^{\epsilon}(0)=0$, and we observe that this function
is bounded and continuously differentiable with
\[
\frac{\partial}{\partial x}J^{\epsilon}(x)=\begin{cases}
-\frac{1}{x^{2}} & \:\:x\geq2\epsilon,\\
0 & \:\:x<\epsilon,
\end{cases}
\]
which is also bounded and non-positive for any $\epsilon>0$.

Let $\{v_{t}^{\epsilon}:\,t\geq0\}$ be the unique solution
to the SDE
\begin{equation}\label{eq:3.4}
dv_{t}^{\epsilon}=\left[\left(\frac{k\theta}{2}-\frac{\xi^{2}}{8}\right)J^{\epsilon}(v_{t}^{\epsilon})-\frac{k}{2}v_{t}^{\epsilon}\right]dt
+\frac{\xi}{2}dB_{t}
\end{equation}
for an arbitrary $\epsilon>0$, where $B_{\cdot}$ is the same Brownian
motion as in \eqref{eq:3.3}. For any $\epsilon>0$, Theorem 2.2.1 from page
102 of \cite{Nualart} implies that $v_{t}^{\epsilon}$ is Malliavin differentiable
with respect to the Brownian motion $B_{\cdot}^{1}$. By looking at
the proof of that Theorem, we can see that the underlying probability
measure does not play any role, as long as we are differentiating
with respect to the path of a Brownian motion, which means that here
we always have Malliavin differentiability under the probability measure
$\mathbb{P}(\cdot\,|\,B_{\cdot}^{0},\,\mathcal{G})$. Under that conditional
probability measure, by the same Theorem and the remark after its
proof we have that the Malliavin derivative of $v_{t}^{\epsilon}$
(with respect to $B_{\cdot}^{1}$) satisfies the integral equation
\[
D_{t'}v_{t}^{\epsilon}=\frac{\xi\sqrt{1-\rho^{2}_{2}}}{2}+\int_{t'}^{t}\left[\left(\frac{k\theta}{2}-\frac{\xi^{2}}{8}\right)\frac{\partial}{\partial x}J^{\epsilon}(v_{s}^{\epsilon})-\frac{k}{2}\right]D_{t'}v_{s}^{\epsilon}ds, \:\:\forall\,t\geq t'\geq0, 
\]
This can be solved in $t$ to give
\begin{equation}\label{eq:3.5}
D_{t'}v_{t}^{\epsilon}=\frac{\xi\sqrt{1-\rho^{2}_{2}}}{2}e^{\int_{t'}^{t}\left[\left(\frac{k\theta}{2}-\frac{\xi^{2}}{8}\right)\frac{\partial}{\partial x}J^{\epsilon}(v_{s}^{\epsilon})-\frac{k}{2}\right]ds},\:\:\forall\,t\geq t'\geq 0.
\end{equation}

As mentioned in \cite{AE2}, the stopping time $\tau_{\epsilon}=\inf\{t>0:\,v_{t}\leq\epsilon\}$
tends to $\infty$ as $\epsilon\rightarrow0$ and we also have $v_{t}^{\epsilon}=v_{t}^{\tau_{2\epsilon}}=v_{t}\geq2\epsilon, \:
\forall\,t\leq\tau_{2\epsilon}$
(this can be seen by observing that when we stop \eqref{eq:3.3} at $\tau_{2\epsilon}$, we can substitute the $\frac{1}{v_t^{\tau_{2\epsilon}}}$ term by the equal $J^{\epsilon}\left(v_t^{\tau_{2\epsilon}}\right)$ and obtain exactly \eqref{eq:3.4} stopped at $\tau_{2\epsilon}$ but for $v_t$ instead of $v_t^{\epsilon}$, so since this stopped SDE has a pathwise unique solution, $v_t$ and $v_t^{\epsilon}$ must coincide up to time $\tau_{2\epsilon}$),
$\mathbb{P}$-almost surely. It follows then that $v_{t}^{\epsilon}\rightarrow v_{t}$
and also $\int_{t'}^{t}\frac{\partial}{\partial x}J^{\epsilon}(v_{s}^{\epsilon})ds\rightarrow-\int_{t'}^{t}\frac{1}{v_{s}^{2}}ds$
for all $t\geq t'\geq0$ as $\epsilon\rightarrow0^{+}$, $\mathbb{P}$-almost
surely. Hence we have
\[
\mathbb{E}\left[\mathbb{P}\left(v_{t}^{\epsilon}\rightarrow v_{t},\:\forall\,t\geq0\,|\,B_{\cdot}^{0},\,\mathcal{G}\right)\right]=\mathbb{P}\left(v_{t}^{\epsilon}\rightarrow v_{t},\:\forall\,t\geq0\right)=1,
\]
which implies that 
\begin{equation}\label{eq:3.6}
\mathbb{P}\left(v_{t}^{\epsilon}\rightarrow v_{t},\:\forall\,t\geq0\,|\,B_{\cdot}^{0},\,\mathcal{G}\right)=1
\end{equation}
$\mathbb{P}$-almost surely. Similarly, we can deduce that 
\begin{equation}\label{eq:3.7}
\mathbb{P}\left(\int_{t'}^{t}\frac{\partial}{\partial x}J^{\epsilon}(v_{s}^{\epsilon})ds\rightarrow-\int_{t'}^{t}\frac{1}{v_{s}^{2}}ds,
\:\forall\,t\geq0\,|\,B_{\cdot}^{0},\,\mathcal{G}\right)=1
\end{equation}
$\mathbb{P}$-almost surely.

Furthermore, it is shown in \cite{AE2} that
\[
\mathbb{E}\left[\mathbb{P}\left(|v_{t}^{\epsilon}|\leq|u_{t}|+|v_{t}|,\:\forall\,t\geq0\,|\,B_{\cdot}^{0},\,\mathcal{G}\right)\right]=\mathbb{P}\left(|v_{t}^{\epsilon}|\leq|v_{t}|+|u_{t}|,\:\forall\,t\ge0\right)=1,
\]
where $u_{t}$ is the Ornstein - Uhlenbeck process of Lemma~\ref{lem:3.2}, while we also have
\begin{eqnarray*}
\mathbb{E}\left[\sup_{0\leq t\leq T}\mathbb{E}\left[(|v_{t}|+|u_{t}|)^{2}\,|\,B_{\cdot}^{0},\,\mathcal{G}\right]\right] &\leq & \mathbb{E}\left[\sup_{0\leq t\leq T}(|v_{t}|+|u_{t}|)^{2}\right] \\
&\leq & 2(||\sup_{0\leq t\leq T}v_{t}||_{L^{2}\left(\Omega\right)}^{2}+||\sup_{0\leq t\leq T}u_{t}||_{L^{2}\left(\Omega\right)}^{2}),
\end{eqnarray*}
which is finite by the two results of Lemma~\ref{lem:3.2}. This means that $\mathbb{P}$-almost
surely we have also
\begin{equation}\label{eq:3.8}
\mathbb{P}\left(|v_{t}^{\epsilon}|\leq|u_{t}|+|v_{t}|\:\forall\,t\geq0\,|\,B_{\cdot}^{0},\,\mathcal{G}\right)=1
\end{equation}
and
\begin{equation}\label{eq:3.9}
\mathbb{E}\left[(|v_{t}|+|u_{t}|)^{2}|\,B_{\cdot}^{0},\,\mathcal{G}\right]<\infty
\end{equation}
for all $0\leq t\leq T$.
By \eqref{eq:3.6}, \eqref{eq:3.7}, \eqref{eq:3.8} and \eqref{eq:3.9}, we have that there exists
an $\Omega_{0}$ of full probability such that for all $\omega\in\Omega_{0}$ and all $0\leq t'\leq t\leq T$, both $v_{t}^{\epsilon}$
and $D_{t'}v_{t}^{\epsilon}$ converge $\mathbb{P}(\cdot\,|\,B_{\cdot}^{0},\,\mathcal{G})$-almost surely to $v_{t}$ and 
\[
V_{t,t'}=\frac{\xi\sqrt{1-\rho^{2}_{2}}}{2}e^{-\int_{t'}^{t}\left[\left(\frac{k\theta}{2}-\frac{\xi^{2}}{8}\right)\frac{1}{v_{t}^{2}}+\frac{k}{2}\right]ds}
\]
respectively as $\epsilon\rightarrow0$, while $\left\{ v_{t}^{\epsilon}:\,t\geq0\right\} $
is dominated by an $L_{B_{\cdot}^{0},\,\mathcal{G}}^{2}$-integrable process and $D_{t'}v_{t}^{\epsilon}\leq\frac{\xi\sqrt{1-\rho^{2}_{2}}}{2}$ for
all $\epsilon>0$ and all $0\leq t'\leq t\leq T$.
Thus, we can apply the Dominated Convergence Theorem to deduce that
the last two convergences hold also in $L_{B_{\cdot}^{0},\,\mathcal{G}}^{2}$
and $L_{B_{\cdot}^{0},\mathcal{G}}^{2}\left(\left[0,\,t\right]\times\Omega\right)$ respectively,
for all $0\leq t\leq T$ and all $\omega\in\Omega_{0}$. Then, by
Lemma~1.2.3 from \cite{Nualart} (page 30) we obtain that $D_{t'}v_{t}$ exists
and is equal to $V_{t,t'}$, 
\begin{equation}\label{eq:3.10}
D_{t'}v_{t}=\frac{\xi\sqrt{1-\rho^{2}_{2}}}{2}e^{-\int_{t'}^{t}\left[\left(\frac{k\theta}{2}-\frac{\xi^{2}}{8}\right)\frac{1}{\sigma_{s}^{1}}-\frac{k}{2}\right]ds}
\end{equation}
for all $0\leq t'\leq t\leq T$ and all $\omega\in\Omega_{0}$.

Finally, for any $n\in\mathbb{N},$ let $f_{n}$ be a smooth
and compactly supported function such that $f_{n}(x)=x$ for all $x\leq n$
and $\left\Vert \frac{\partial}{\partial x}f_{n}\right\Vert _{\infty}=1$.
By Lemma~\ref{lem:3.2}, we have
\[
\mathbb{E}\left[\sup_{0\leq t\leq T}\mathbb{E}\left[\sigma_{t}^{2}\,|\,B_{\cdot}^{0},\,\mathcal{G}\right]\right]\leq\mathbb{E}\left[\sup_{0\leq t\leq T}\sigma_{t}^{2}\right]<\infty,
\]
which implies that $\sigma_{t}\in L_{B_{\cdot}^{0},\,\mathcal{G}}^{2}$ for
all $\omega\in\Omega_{1}$ and all $0\leq t\leq T$, where $\Omega_{1}\subset\Omega_{0}$
is a set of full probability. Then, for all $\omega\in\Omega_{1}$ and all $0\leq t\leq T$, the Dominated Convergence Theorem implies that $f_{n}^{2}(v_{t})\rightarrow\sigma_{t}$
in $L_{B_{\cdot}^{0},\,\mathcal{G}}^{2}$ as $n\rightarrow\infty$ (since
we obviously have $\mathbb{P}(\cdot\,|\,B_{\cdot}^{0},\,\mathcal{G})$-almost
sure convergence and domination by $v_{t}^{2}=\sigma_{t}$). Moreover, for all $\omega\in\Omega_{1}$ and all $0\leq t' \leq t \leq T$, the standard Malliavin chain rule implies that
\[
D_{t'}f_{n}^{2}(v_{t})=2f_{n}(v_{t})f_{n}'(v_{t})D_{t'}v_{t}\rightarrow2v_{t}D_{t'}v_{t}
\]
$\mathbb{P}(\cdot\,|\,B_{\cdot}^{0},\,\mathcal{G})$ - almost surely as $n\rightarrow\infty$,
while we have also domination by $2v_{t}D_{t'}v_{t}\leq\xi\rho_{2}v_{t}\in L_{B_{\cdot}^{0},\,\mathcal{G}}^{2}$. Thus, we can use the Dominated Convergence
Theorem once more to see that the last convergence holds also in $L_{B_{\cdot}^{0},\mathcal{G}}^{2}\left(\left[0,\,t\right]\times\Omega\right)$, for all $\omega\in\Omega_{1}$ and all $0 \leq t \leq T$.
Recalling now Lemma~1.2.3 from \cite{Nualart} again, we deduce that $D_{t'}\sigma_{t}$
exists in $L_{B_{\cdot}^{0},\mathcal{G}}^{2}\left(\left[0,\,t\right]\times\Omega\right)$
and it is equal to $2v_{t}D_{t'}v_{t}$, thus
\[
D_{t'}\sigma_{t}=\xi\sqrt{1-\rho^{2}_{2}}e^{-\int_{t'}^{t}\left[\left(\frac{k\theta}{2}-\frac{\xi^{2}}{8}\right)\frac{1}{\sigma_{s}^{1}}-\frac{k}{2}\right]ds}\sqrt{\sigma_{t}}
\]
which is exactly \eqref{eq:3.2}. The proof is now complete.
\end{proof}

\begin{proof}[\textbf{Proof of Lemma 3.4}]
Fix $t>0$. Consider the sequence of stochastic processes (in $t'\in\left[0,\,t\right]$)
\[
v_{t,t'}^{n}=\xi\sqrt{1-\rho^{2}_{2}}e^{-\int_{t'}^{t}\left[\left(\frac{k\theta}{2}-\frac{\xi^{2}}{8}\right)\frac{1}{\sigma_{s}+\frac{1}{n}}+\frac{k}{2}\right]ds}g_{n}\left(\sqrt{\sigma_{t}+\frac{1}{n}}\right),\;\:\forall\,n\in\mathbb{N},
\]
where the smooth and increasing cut-off function $g_{n}$ satisfies
\[
g_{n}(x)=\begin{cases}
x, & \;\;0\leq x\leq n,\\
-1, & \;\;x\leq-2,\\
n+1, & \;\;x\geq n+2,
\end{cases}
\]
and has a derivative which is bounded by $1$. This process is uniformly
bounded by $\xi\sqrt{1-\rho^{2}_{2}}\sqrt{\sigma_{s}+1}\in L^{p}\left(\Omega\right),\:\forall\,p\geq1$,
since 
\[
\sqrt{\sigma_{s}+1}\leq\frac{1}{2}{\displaystyle \sup_{0\leq s\leq T}}\left(\sigma_{s}+2\right),
\]
which has finite moments by Lemma~\ref{lem:3.2}. Thus we have also $\xi\sqrt{1-\rho^{2}_{2}}\sqrt{\sigma_{s}+1}\in L_{B_{\cdot}^{0},\,\mathcal{G}}^{p}\left(\Omega\right)$ for all $p>1$,
$\mathbb{P}$-almost surely. Moreover, by the Monotone Convergence
Theorem, $v_{t,t'}^{n}$ converges pointwise to $D_{t'}\sigma_{t}$
as $n\rightarrow\infty$, so by the Dominated Convergence Theorem
we see that this convergence holds also in $L_{B_{\cdot}^{0},\,\mathcal{G}}^{p}\left(\Omega\right)$
for any $t'<t$, and also in $L_{B_{\cdot}^{0},\,\mathcal{G}}^{p}\left(\Omega;\,L^{2}\left(\left[0,\,t\right]\right)\right)$,
for any $p\geq1$, $\mathbb{P}$-almost surely.

Next, observe that
\begin{equation}\label{eq:3.14}
v_{t,t'}^{n}=f\left(\int_{t'}^{t}h_{n}\left(\sigma_{s}\right)ds\right)g_{n}\left(\sqrt{\sigma_{t}+\frac{1}{n}}\right),
\end{equation}
where $f,\,h_{n}$ are sufficiently smooth functions with bounded
first derivatives, such that $f(x)=\sqrt{1-\rho^{2}_{2}}\xi e^{-x}$ and $h_{n}(x)=\frac{k}{2}+\left(\frac{k\theta}{2}-\frac{\xi^{2}}{8}\right)
\frac{1}{x+\frac{1}{n}}$ for $x>0$ and $n\in\mathbb{N}$, and $f(x)=h_{n}(x)=0$ for $x<-1$
and $n\in\mathbb{N}$. Now we recall the standard Malliavin chain
rule, so almost surely, under the probability measure $\mathbb{P}(\cdot\,|\,B_{\cdot}^{0},\,\mathcal{G})$,
we have
\[
D_{t''}h_{n}\left(\sigma_{s}\right)=h'_{n}\left(\sigma_{s}\right)D_{t''}\sigma_{s},
\]
which is bounded for any $n\in\mathbb{N}$, so we can integrate in
$s$ and intechange the integral with the derivative to obtain
\begin{equation}\label{eq:3.15}
D_{t''}\int_{t'}^{t}h_{n}\left(\sigma_{s}\right)ds=\int_{t'}^{t}h'_{n}\left(\sigma_{s}\right)D_{t''}\sigma_{s}ds.
\end{equation}

Next, observe that all the arguments in \eqref{eq:3.14} are positive,
so by applying the same Malliavin chain rule and by substituting from \eqref{eq:3.15} we obtain
\begin{eqnarray}
D_{t''}v_{t,t'}^{n}&=& f'\left(\int_{t'}^{t}h_{n}\left(\sigma_{s}\right)ds\right)\int_{t'}^{t}h'_{n}\left(\sigma_{s}\right)D_{t''}\sigma_{s}ds g_{n}\left(\sqrt{\sigma_{t}+\frac{1}{n}}\right) \nonumber \\
& & \qquad +f\left(\int_{t'}^{t}h_{n}\left(\sigma_{s}\right)ds\right)g'_{n}\left(\sqrt{\sigma_{t}+\frac{1}{n}}\right)\frac{D_{t''}\sigma_{t}}{2\sqrt{\sigma_{t}+\frac{1}{n}}} \nonumber \\
&=&\sqrt{1-\rho^{2}_{2}}\xi e^{-\int_{t'}^{t}\left(\frac{k}{2}+\left(\frac{k\theta}{2}-\frac{\xi^{2}}{2}\right)\frac{1}{\sigma_{s}+\frac{1}{n}}\right)ds}\int_{t'}^{t}\frac{\left(\frac{k\theta}{2}-\frac{\xi^{2}}{8}\right)D_{t''}\sigma_{s}ds}{\left(\sigma_{s}+\frac{1}{n}\right)^{2}}g_{n}\left(\sqrt{\sigma_{t}+\frac{1}{n}}\right) \nonumber \\
&& \qquad +\sqrt{1-\rho^{2}_{2}}\xi e^{-\int_{t'}^{t}\left(\frac{k}{2}+\left(\frac{k\theta}{2}-\frac{\xi^{2}}{2}\right)\frac{1}{\sigma_{s}+\frac{1}{n}}\right)ds}g'_{n}\left(\sqrt{\sigma_{t}+\frac{1}{n}}\right)\frac{D_{t''}\sigma_{t}}{2\sqrt{\sigma_{t}+\frac{1}{n}}}. \nonumber \\
\label{eq:3.16}
\end{eqnarray}

Now we want to bound the above quantity by some process
in $L_{B_{\cdot}^{0},\,\mathcal{G}}^{q'}\left(\Omega;\,L^{2}\left[\left[0,\,t\right]^{2}\right]\right)$,
uniformly in $n\in\mathbb{N}$, so we can apply again the Dominated
Convergence Theorem, for some $q'>1$. Observe that $D_{t''}\sigma_{s}\leq\xi\sqrt{1-\rho^{2}_{2}}\sqrt{\sigma_{s}}$
(by \eqref{eq:3.2}) and that $0\leq\frac{d}{dx}g_{n}(x)\leq1\Rightarrow g_{n}(x)\leq x$
for all $n\in\mathbb{N}$, so if we drop the summand$-\int_{t'}^{t}\frac{k}{2}$
from the exponents in \eqref{eq:3.16} we obtain
\begin{eqnarray*}
D_{t''}v_{t,t'}^{n} &\leq & \xi\sqrt{1-\rho^{2}_{2}}e^{-\int_{t'}^{t}\left(\frac{k\theta}{2}-\frac{\xi^{2}}{2}\right)\frac{1}{\sigma_{s}+\frac{1}{n}}ds}\int_{t'}^{t}\frac{\sqrt{\sigma_{s}}ds}{\left(\sigma_{s}+\frac{1}{n}\right)^{2}}\sqrt{\sigma_{t}+\frac{1}{n}} \\
& & \qquad +\xi\sqrt{1-\rho^{2}_{2}}e^{-\int_{t'}^{t}\left(\frac{k\theta}{2}-\frac{\xi^{2}}{2}\right)\frac{1}{\sigma_{s}+\frac{1}{n}}ds}\frac{\sqrt{\sigma_{t}}}{2\sqrt{\sigma_{t}+\frac{1}{n}}} \nonumber \\ 
&<& \xi\sqrt{1-{\rho_{2}}^{2}}\left[\int_{0}^{T}\frac{ds}{\sigma_{s}^{\frac{3}{2}}}\sup_{0\leq s\leq T}\sqrt{\sigma_{s}+1}+1\right]
\end{eqnarray*}
whose $L_{B_{\cdot}^{0},\,\mathcal{G}}^{q'}\left(\Omega;\,L^{2}\left[\left[0,\,t\right]^{2}\right]\right)$
norm is bounded by
\begin{eqnarray}
& & \xi\sqrt{1-\rho^{2}_{2}}t\left(\mathbb{E}\left[\left(\int_{0}^{T}\frac{ds}{\sigma_{s}^{\frac{3}{2}}}\sup_{0<s\leq T}\sqrt{\sigma_{s}+1}\right)^{q'}\right]+1\right)^{\frac{1}{q'}} \nonumber \\
& & \qquad \leq  \xi\sqrt{1-\rho^{2}_{2}}t\left(\mathbb{E}^{\frac{q'}{p}}\left[\left(\int_{0}^{T}\frac{ds}{\sigma_{s}^{\frac{3}{2}}}\right)^{p}\right]\mathbb{E}^{\frac{q'}{p'}}\left[\left(\sup_{0<s\leq T}\sqrt{\sigma_{s}+1}\right)^{p'}\right]+1\right)^{\frac{1}{q'}}\nonumber  \\
& & \qquad \leq \xi\sqrt{1-\rho^{2}_{2}}t\left(T^{\frac{q'}{p'}}\mathbb{E}^{\frac{q'}{p}}\left[\int_{0}^{T}\frac{ds}{\sigma_{s}^{\frac{3p}{2}}}\right]\mathbb{E}^{\frac{q'}{p'}}\left[\left(\sup_{0<s\leq T}\sqrt{\sigma_{s}+1}\right)^{p'}\right]+1\right)^{\frac{1}{q'}}, \nonumber \\
\label{eq:3.17}
\end{eqnarray}
where $\frac{1}{p}+\frac{1}{p'}=\frac{1}{q'}$. 

The second expectation of \eqref{eq:3.17} is finite for all $p'<\infty\Leftrightarrow p>q'$
because of the estimate $\left(\sigma_{s}+1\right)^{\frac{p'}{2}}\leq C_{1}\left(\sigma_{s}^{p'}+1\right)$
for some $C_{1}>0$ and Lemma~\ref{lem:3.2}. On the other hand, if $\frac{2k\theta}{\xi^{2}}>\frac{3p}{2}$,
the first expectation of \eqref{eq:3.17} can be computed by recalling Theorem~3.1 from \cite{HK06} as follows
\begin{equation}\label{eq:3.18}
\mathbb{E}\left[\int_{0}^{T}\frac{ds}{\sigma_{s}^{\frac{3p}{2}}}\right]=\mathbb{E}\left[\int_{0}^{T}\mathbb{E}\left[\sigma_{s}^{-\frac{3p}{2}}\,|\,\sigma_{0}\right]ds\right]=\lambda_{1}\mathbb{E}\left[\int_{0}^{T}\gamma_{s}^{\frac{3p}{2}}H\left(-\gamma_{s}\sigma_{0}e^{-ks}\right)ds\right]
\end{equation}
where $\lambda_{1}>0$, $\gamma_{s}=\frac{2k}{\xi^{2}}\left(1-e^{-ks}\right)^{-1}>\frac{2k}{\xi^{2}}$
for all $s\geq0$, and $H$ is a hypergeometric function for which
we have the asymptotic estimate of page 17 in \cite{HK06}. That estimate
(for $N=0$) easily gives $H(-z)\leq\lambda_{2}|z|^{-\frac{3p}{2}}$
for some $\lambda_{2}>0$ and all $z\geq0$. Thus, by \eqref{eq:3.18} we
find
\[
\mathbb{E}\left[\int_{0}^{T}\frac{ds}{\sigma_{s}^{\frac{3p}{2}}}\right]\leq\lambda_{1}\lambda_{2}\mathbb{E}\left[\int_{0}^{T}e^{\frac{3kps}{2}}\sigma_{0}^{-\frac{3p}{2}}\right]ds=\lambda_{1}\lambda_{2}\int_{0}^{T}e^{\frac{3kps}{s}}ds\mathbb{E}\left[\sigma_{0}^{-\frac{3p}{2}}\right]
\]
which is finite by our initial data assumptions if and only
if $\frac{2k\theta}{\xi^{2}}>\frac{3p}{2}$. Thus, the RHS of \eqref{eq:3.17}
if finite iff $\frac{2k\theta}{\xi^{2}}>\frac{3p}{2}$. This can be
achieved by making $p$ sufficiently close to $q'$, provided that:
$\frac{2k\theta}{\xi^{2}}>\frac{3q'}{2}$ which is equivalent to $q'<\frac{4k\theta}{3\xi^{2}}$.
We can choose such a $q'>1$ since we have $\frac{4k\theta}{3\xi^{2}}>1$.
Observe that the same condition is assumed in \cite{AE1} to obtain $L^{1}$
regularity, but for our purpose, we are going to need this $L^{q}$ regularity
for some $q$ strictly bigger than $1$. Moreover,
we need to have a finite $L_{B_{\cdot}^{0},\,\mathcal{G}}^{q'}\left(\Omega;\,L^{2}\left[\left[0,\,t\right]^{2}\right]\right)$
norm, $\mathbb{P}$-almost surely, and this is obtained by
the law of total expectation as follows
\[
\mathbb{E}\left[\mathbb{E}\left[\left(\int_{0}^{T}\frac{ds}{\sigma_{s}^{\frac{3}{2}}}\sup_{0<s\leq T}\sqrt{\sigma_{s}+1}\right)^{q'}+1\,|\,B_{\cdot}^{0},\,\mathcal{G}\right]\right]
=\mathbb{E}\left[\left(\int_{0}^{T}\frac{ds}{\sigma_{s}^{\frac{3}{2}}}\sup_{0<s\leq T}\sqrt{\sigma_{s}+1}\right)^{q'}+1\right], \]
so we have
\[
\mathbb{E}\left[\left(\int_{0}^{T}\frac{ds}{\sigma_{s}^{\frac{3}{2}}}\sup_{0<s\leq T}\sqrt{\sigma_{s}+1}\right)^{q'}+1\,|\,B_{\cdot}^{0},\,\mathcal{G}\right]<\infty
\]
for all $\omega$ in some $\Omega'\subset\Omega$ of full probability.
Thus, the pointwise convergence of $D_{t''}v_{t,t'}^{n}$ to the RHS
of \eqref{eq:3.11} and the Dominated Convergence Theorem imply that we have
the same convergence in $L_{B_{\cdot}^{0},G}^{q}\left(\Omega;\,L^{2}\left[\left[0,\,t\right]^{2}\right]\right)$
for all $\omega\in\Omega'$. Then, since $v_{t,t'}^{n}$ converges
to $D_{t'}\sigma_{t}$ in $L_{B_{\cdot}^{0},\,\mathcal{G}}^{q'}\left(\Omega\right)$
for any $t'<t$ and also in $L_{B_{\cdot}^{0},\,\mathcal{G}}^{q'}\left(\Omega;\,L^{2}\left(\left[0,\,t\right]\right)\right)$,
for any $p\geq1$ and any $\omega\in\Omega_{2}\subset\Omega'$, we
deduce that $\sigma_{t}\in\mathbb{D}^{2,q'}$ under the probability
measure $\mathbb{P}(\cdot\,|\,B_{\cdot}^{0},\,\mathcal{G})$) with respect
to $B_{\cdot}^{1}$, with the second Malliavin Derivative being given
by \eqref{eq:3.11}, $\mathbb{P}$-almost surely. It follows also that \eqref{eq:3.12}
holds, since the sequence converging pointwise to the RHS of \eqref{eq:3.11}
is dominated by a random quantity of finite positive moments, uniformly
in $t,\,t',\,t''\in[0,\,T]$. The proof for $v_{t}=\sqrt{\sigma_{t}}$
is similar, the only difference is the absence of the functions $g_{n}$
and the terminal value term.
\end{proof}

\begin{proof}[\textbf{Proof of Lemma 3.5}]
For any $a\geq b$, we define $\psi(y)=\mathbb{I}_{[b,\,a]}(y)$
and $\phi(y)=\int_{-\infty}^{y}\psi(z)dz$. The standard Malliavin
Chain rule implies that $\phi(F)\in\mathbb{D}^{1,2}$ and moreover
\begin{eqnarray*}
\left\langle u_{.},\,D.\phi(F)\right\rangle _{L^{2}}&=&\left\langle u_{.},\,\psi(F)D_{.}F\right\rangle _{L^{2}}  \\
&=& \psi(F)\left\langle u_{.},\,D_{.}F\right\rangle _{L^{2}}
\end{eqnarray*}
 and dividing by $<u_{.},\,D_{.}F>_{L^{2}}$ yields
\begin{equation}
\psi(F)=\left\langle \frac{u_{.}}{\left\langle u_{.},\,D_{.}F\right\rangle _{L^{2}}},\,D.\phi(F)\right\rangle _{L^{2}} 
\label{eq:3.19}
\end{equation}

Next, by Proposition~1.5.4 of \cite{Nualart} (page 69), $\frac{u_{t}}{<u_{.},\,D_{.}F>_{L^{2}}}$
belongs to the domain of $\delta$, the adjoint of the Malliavin derivative operator, and there exists a constant $C>0$
such that
\begin{equation}\label{eq:3.20}
\mathbb{E}^{\frac{1}{q'}}\left[\left|\delta\left(\frac{u_{.}}{\left\langle u_{.},\,D_{.}F\right\rangle }\right)\right|^{q'}\right]\leq C\mathbb{E}^{\frac{1}{q'}}\left[\left(\left\Vert D_{.}\frac{u_{.}}{\left\langle u_{.},\,D_{.}F\right\rangle _{L^{2}}}\right\Vert _{L^{2}\left(\left[0,\,T\right]^{2}\right)}\right)^{q'}\right]<\infty.
\end{equation}
Hence \eqref{eq:3.19} implies
\begin{eqnarray}
\mathbb{P}\left(b\leq F\leq a\right)&=& \mathbb{E}\left(\psi(F)\right)=\mathbb{E}\left[\delta\left(\frac{u_{.}}{\left\langle u_{.},\,D_{.}F\right\rangle _{L^{2}}}\right)\phi(F)\right] \nonumber \\
&=& \mathbb{E}\left[\int_{-\infty}^{F}\mathbb{I}_{[b,\,a]}(z)\delta\left(\frac{u_{.}}{\left\langle u_{.},\,D_{.}F\right\rangle _{L^{2}}}\right)dz\right] \nonumber \\
&=& \mathbb{E}\left[\int_{b}^{a}\mathbb{I}_{\{z\leq F\}}\delta\left(\frac{u_{.}}{\left\langle u_{.},\,D_{.}F\right\rangle _{L^{2}}}\right)dz\right] \label{eq:3.21}
\end{eqnarray}
Now, by Holder's inequality and \eqref{eq:3.20} we have that the quantity
\[
p(z)=\mathbb{E}\left[\mathbb{I}_{\{z\leq F\}}\delta\left(\frac{u_{.}}{\left\langle u_{.},\,D_{.}F\right\rangle _{L^{2}}}\right)\right]
\]
is bounded, thus by Fubini's Theorem and \eqref{eq:3.21} we obtain
\[
\mathbb{P}\left(b\leq F\leq a\right)=\int_{b}^{a}\mathbb{E}\left[\mathbb{I}_{\{z\leq F\}}\delta\left(\frac{u_{.}}{\left\langle u_{.},\,D_{.}F\right\rangle _{L^{2}}}\right)\right]dz
\]
Therefore, the probability density exists and is equal to
$p(z)$, which is bounded as mentioned above. Moreover, since the
quantity within the expectation is dominated by $\delta\left(\frac{u_{.}}{\left\langle u_{.},\,D_{.}F\right\rangle _{L^{2}}}\right)$,
which is in $L^{q'}$ by \eqref{eq:3.20}, the Dominated Convergence Theorem
implies that the density is also continuous. Furthermore, for $\alpha\geq0$,
by Holder's inequality and \eqref{eq:3.20}, we have
\begin{eqnarray*}
y^{\alpha}p(y)&=& \mathbb{E}\left[y^{\alpha}\mathbb{I}_{\{y\leq F\}}\delta\left(\frac{u_{.}}{\left\langle u_{.},\,D_{.}F\right\rangle _{L^{2}}}\right)\right] \\
&\leq &\mathbb{E}^{\frac{q'-1}{q'}}\left[y^{\frac{\alpha q'}{q'-1}}\mathbb{I}_{y\leq F}\right]\mathbb{E}^{\frac{1}{q'}}\left[\left|\delta\left(\frac{u_{.}}{\left\langle u_{.},\,D_{.}F\right\rangle }\right)\right|^{q'}\right] \\
&\leq & C\mathbb{E}\left[F^{\frac{\alpha q'}{q'-1}}\right]\mathbb{E}^{\frac{1}{q'}}\left[\left(\left\Vert D_{.}\frac{u_{.}}{\left\langle u_{.},\,D_{.}F\right\rangle _{L^{2}}}\right\Vert _{L^{2}\left(\left[0,\,T\right]^{2}\right)}\right)^{q'}\right]<\infty,
\end{eqnarray*}
for any $y>0$ and the desired estimate follows.
\end{proof}

\begin{proof}[\textbf{Proof of Lemma 4.2}]
By using the
bounded convergence theorem we can easily obtain $\sigma_{.}^{m}\rightarrow\sigma_{.}^{0}$
and $\sqrt{\sigma_{.}^{m}}\rightarrow\sqrt{\sigma_{.}^{0}}$ in $L^{2}$,
as $m\rightarrow\infty$. 

To prove our claim, we consider the process $Y_{.}^{m}$ satisfying the same
SDE and initial condition as $X_{.}^{m}$ for any $m \in \mathbb{N}\cup \{0\}$,
but without being stopped when it hits zero. We will show first that for a
subsequence $\{m_{k}:\,k\in\mathbb{N}\}\subset\mathbb{N}$, we have
almost surely: $Y_{.}^{m_{k}}\rightarrow Y_{.}^{0}$ uniformly on any
compact interval $\left[0,\,T\right]$, and there exists a $k_{0}\in\mathbb{N}$
such that $Y_{t}^{m_{k}}<Y_{t}^{0}\:\forall\,k\geq k_{0}$ and all $t\leq T$.
Indeed, we have
\[
\sup_{t\leq T}\left|Y_{t}^{m}-Y_{t}^{0}\right|=\sup_{t\leq T}\left|X_{0}^{m}-X_{0}^{0}-\frac{1}{2}\int_{0}^{t} 
\left(\sigma_{s}^{m}-\sigma_{s}^{0}\right)ds+\int_{0}^{t}\left(\sqrt{\sigma_{s}^{m}}-\sqrt{\sigma_{s}^{0}}\right)
dW_{s}\right|, \]
This is bounded by
\begin{eqnarray*}
&& \left|X_{0}^{m}-X_{0}^{0}\right|+\frac{1}{2}\int_{0}^{T}\left|\sigma_{s}^{m}-\sigma_{s}^{0}\right|ds+
\sup_{t\leq T}\left|\int_{0}^{t}\left(\sqrt{\sigma_{s}^{m}}-\sqrt{\sigma_{s}^{0}}\right)dW_{s}\right|
\\
&& \qquad =\min\left\{\frac{x_{0}}{2},\,l_{m}\right\}+\frac{1}{2}\left\Vert \sigma_{\cdot}^{m}-\sigma_{\cdot}^{0}
\right\Vert _{L^{1}\left[0,T\right]}+\sup_{t\leq T}\left|\int_{0}^{t}\left(\sqrt{\sigma_{s}^{m}} 
-\sqrt{\sigma_{s}^{0}}\right)dW_{s}\right|,
\end{eqnarray*}
where the first two terms tend obviously to zero, while the last term tends to zero in probability due to 
Doob's Martingale inequality for $p=2$ and Ito's isometry, so along a subsequence, the whole quantity 
tends almost surely to zero. Next, we have
\begin{eqnarray}
Y_{t}^{m}-Y_{t}^{0} &=& -\min\left\{\frac{x_{0}}{2},\,l_{m}\right\}-\frac{1}{2}\int_{0}^{t}\left(\sigma_{s}^{m}-\sigma_{s}^{0} 
\right)ds +\int_{0}^{t}\left(\sqrt{\sigma_{s}^{m}}-\sqrt{\sigma_{s}^{0}}\right)dW_{s} \nonumber \\
& \leq & -\min\left\{\frac{x_{0}}{2},\,l_{m}\right\}+\sup_{t\leq T}\left|\int_{0}^{t} 
\left(\sqrt{\sigma_{s}^{m}}-\sqrt{\sigma_{s}^{0}}\right)dW_{s}\right|,
\label{eq:4.5}
\end{eqnarray}
and once more, by Doob's Martingale inequality for $p=2$ and 
$l_{m}^{0}=\frac{\min\{\frac{x_{0}}{2},\,l_{m}\}}{2}$, we have
\begin{eqnarray*}
 & & \mathbb{P}\left(\sup_{t\leq T}\left|\int_{0}^{t}\left(\sqrt{\sigma_{s}^{m}}-\sqrt{\sigma_{s}^{0}}\right)
dW_{s}\right|>l_{m}^{0}\,|\,\mathcal{F}_{0}\right) \\
& & \qquad \qquad \leq  \frac{1}{\left(l_{m}^{0}\right)^{2}} 
\mathbb{E}\left(\left(\int_{0}^{T}\left(\sqrt{\sigma_{s}^{m}}-\sqrt{\sigma_{s}^{0}}\right)dW_{s}\right)^{2}\,|
\,\mathcal{F}_{0}\right)\\
& &\qquad \qquad=\frac{\left(l_{m}\right)^{4}}{\left(l_{m}^{0}\right)^{2}},
\end{eqnarray*}
where $l_{m}^{0}=\frac{l_{m}}{2}\rightarrow0$ for large $m$. Thus,
there exists a subsequence $\{m_{k}:\,k\in\mathbb{N}\}\subset\mathbb{N}$
such that
\[
\sum_{k=1}^{\infty}\frac{\left(l_{m_{k}}\right)^{4}}{\left(l_{m_{k}}^{0}\right)^{2}}<\infty
\]
which implies that almost surely, $\sup_{t\leq T}\left|\int_{0}^{t}\left(\sqrt{\sigma_{s}^{m_{k}}} 
-\sqrt{\sigma_{s}^{0}}\right)dW_{s}\right|<l_{m_{k}}^{0}=\frac{l_{m_{k}}}{2}$
for all large $k$ (by the Borel-Cantelli lemma). Therefore, by \eqref{eq:4.5}
we obtain
\begin{equation}\label{eq:4.6}
\sup_{t\leq T}\left(Y_{t}^{m_{k}}-Y_{t}^{0}\right)\leq-l_{m_{k}}+\frac{l_{m_{k}}}{2}=-\frac{l_{m_{k}}}{2}<0
\end{equation}
almost surely for all large $k$.

We are ready now to prove the uniform convergence of the
stopped processes. For a fixed event, the stopping times $\tau^{m_{k}},\,\tau^{0}$
are given and \eqref{eq:4.6} implies that $\tau^{m_{k}}\leq\tau^{0}$ for all
large $k$. Moreover, we have $\tau^{m_{k}}\rightarrow\tau^{0}$ as $k\rightarrow\infty$.
Indeed, for any $\epsilon>0$, $Y^{m_{k}}$ is lower bounded by a
positive constant in $\left[0,\,\tau^{0}-\epsilon\right]$ for all $k$
bigger than some $k^{0}(\epsilon)$ (since the same holds for the
continuous process $Y_{\cdot}^{0}$ by the definition of $\tau^{0}$, and
since $Y_{\cdot}^{m_{k}}$ tends uniformly to $Y_{\cdot}^{0}$), which
implies that $\tau^{0}-\epsilon<\tau^{m_{k}}\leq\tau^{0}$ for all $k\geq k^{0}(\epsilon)$.
Now, if we have $\tau^{0}>T$, then we have also $\tau^{m_{k}}>T$ for
all big enough $k$, which gives
\[
\lim_{k\rightarrow\infty}\sup_{0\leq t\leq T}\left|X_{t}^{m_{k}}-X_{t}^{0}\right|=\lim_{k\rightarrow\infty}\sup_{0\leq t\leq T}
\left|Y_{t}^{m_{k}}-Y_{t}^{0}\right|=0.
\]
On the other hand, if $\tau^{0}\leq T$ we have
\begin{equation}\label{eq:4.7}
\sup_{t\leq T}\left|X_{t}^{m_{k}}-X_{t}^{0}\right|=\max\left\{ \sup_{0\leq t\leq\tau^{m_{k}}}\left|X_{t}^{m_{k}}-X_{t}^{0}\right|,\,\sup_{\tau^{m_{k}}\leq t\leq\tau^{0}}\left|X_{t}^{m_{k}}-X_{t}^{0}\right|,\,\sup_{\tau^{0} \leq t\leq T}\left|X_{t}^{m_{k}}-X_{t}^{0}\right|\right\}. 
\end{equation}

The first supremum of the RHS of the above is equal to the
supremum of $\left|Y_{t}^{m_{k}}-Y_{t}^{0}\right|$ for $t\leq\tau^{m_{k}}$,
which tends to zero since $Y_{\cdot}^{m_{k}}\rightarrow Y_{\cdot}^{0}$ uniformly in $\left[0,\,T\right]$,
while the third one is always equal to $0$. Hence, we only need to
show that the second supremum of the RHS of \eqref{eq:4.7} tends also to
$0$ as $k\rightarrow\infty$. Indeed, for some $\tau^{m_{k}}\leq t_{k}\leq\tau^{0}$,
we have
\[
\sup_{\tau^{m_{k}}\leq t\leq\tau^{0}}\left|X_{t}^{m_{k}}-X_{t}^{0}\right|=\left|Y_{t_{k}}^{0}\right|\rightarrow\left|Y_{\tau^{0}}^{0}\right|=0,
\]
as $k\rightarrow\infty$ (by the continuity of $Y_{\cdot}^{0}$) so the
desired result follows.
\end{proof}

\begin{proof}[\textbf{Proof of Lemma 5.3}]
Observe that by setting $z = v^2, v \in \mathbb{R}^{+}$, any integration against $\phi_{\epsilon}$ can be written as an integration against the standard heat kernel, i.e
\begin{eqnarray*}
J_{u,\epsilon}(\lambda,\,y)&=& \mathbb{\int_{\mathbb{R}^{+}}}u(\lambda,\,z)\frac{1}{\sqrt{2\pi\epsilon}}e^{-\frac{(\sqrt{z}-y)^{2}}{2\epsilon}}dz
\\
&=& \int_{\mathbb{R}}2vu\left(\lambda,\,v^{2}\right)\mathbb{I}_{\mathbb{R}^{+}}(v)\frac{1}{\sqrt{2\pi\epsilon}}e^{-\frac{(v-y)^{2}}{2\epsilon}}dv
\end{eqnarray*}
We are going to prove 1. first. Observe that by our regularity assumptions and the properties of the standard heat kernel, $J_{u,\epsilon}(\lambda,\,y)$ is smooth and it's $n$-th derivative in $y$ equals 
\begin{eqnarray*}
\int_{\mathbb{R}^{+}}2vu\left(\lambda,\,v^{2}\right)\frac{1}{\sqrt{2\pi\epsilon}}P(v - y)e^{-\frac{(v-y)^{2}}{2\epsilon}}dv
\end{eqnarray*}
where $P$ is some polynomial of degree $n$. Thus we need to show that for any $\delta > 0$ and $n \in \mathbb{N}$ we have
\begin{eqnarray*}
\int_{\Lambda}\int_{\mathbb{R}^{+}}y^{\delta'}\left(\int_{\mathbb{R}^{+}}2vu\left(\lambda,\,v^{2}\right)(v - y)^{n}\frac{e^{-\frac{(v-y)^{2}}{2\epsilon}}}{\sqrt{2\pi\epsilon}}dv\right)^{2}dyd\mu(\lambda) < \infty
\end{eqnarray*}
 
By Cauchy-Schwartz, the above quantity is bounded by:

\begin{eqnarray}
& & \int_{\Lambda}\int_{\mathbb{R}^{+}}y^{\delta'}\left(\int_{\mathbb{R}^{+}}4v^{2}u^{2}\left(\lambda,\,v^{2}\right)(v - y)^{2n}\frac{e^{-\frac{(v-y)^{2}}{2\epsilon}}}{\sqrt{2\pi\epsilon}}dv\right)\left(\int_{\mathbb{R}}\frac{e^{-\frac{(v-y)^{2}}{2\epsilon}}}{\sqrt{2\pi\epsilon}}dv\right)dyd\mu(\lambda) \nonumber \\
&&\qquad = \int_{\Lambda}\int_{\mathbb{R}^{+}}\left(\int_{\mathbb{R}^{+}}y^{\delta'}4v^{2}u^{2}\left(\lambda,\,v^{2}\right)(y - v)^{2n}\frac{e^{-\frac{(v-y)^{2}}{2\epsilon}}}{\sqrt{2\pi\epsilon}}dv\right)dyd\mu(\lambda) \nonumber \\
\end{eqnarray}
and thus, by Fubini's Theorem, we only need to show that

\begin{eqnarray*}
\int_{\Lambda}\int_{\mathbb{R}^{+}}4v^{2}u^{2}\left(\lambda, \, v^{2}\right)\left(\int_{\mathbb{R}^{+}}y^{\delta'}(y - v)^{2n}\frac{e^{-\frac{(v-y)^{2}}{2\epsilon}}}{\sqrt{2\pi\epsilon}}dy\right)dvd\mu(\lambda) < \infty \nonumber \\
\end{eqnarray*}
for which it suffices to show that 
\begin{eqnarray*}
&& \int_{\mathbb{R}^{+}}y^{\delta'}(y - v)^{2n} \frac{e^{-\frac{(v-y)^{2}}{2\epsilon}}}{\sqrt{2\pi\epsilon}}dy = \mathcal{O}\left(v^{\delta'} + 1\right) 
\end{eqnarray*}
due to our integrability assumptions for $J_u\left(\lambda, \, v\right) = 2vu\left(\lambda, \, v^2\right)$.

For $\delta' \geq 0$, we use the well known estimate $(a + b)^{\delta'} \leq C(|a|^{\delta'} + |b|^{\delta'})$ to obtain
\begin{eqnarray}
&& \int_{\mathbb{R}^{+}}y^{\delta'}(y - v)^{2n} \frac{e^{-\frac{(v-y)^{2}}{2\epsilon}}}{\sqrt{2\pi\epsilon}}dy \nonumber \\
& & \qquad \leq Cv^{\delta'}\int_{\mathbb{R}}(y - v)^{2n} \frac{e^{-\frac{(v-y)^{2}}{2\epsilon}}}{\sqrt{2\pi\epsilon}}dy 
+ C\int_{\mathbb{R}}(|y - v|)^{2n+\delta'} \frac{e^{-\frac{(v-y)^{2}}{2\epsilon}}}{\sqrt{2\pi\epsilon}}dy \nonumber \\
& & \qquad \leq C{\epsilon}^{n}v^{\delta'}\int_{\mathbb{R}}w^{2n} \frac{e^{-\frac{w^{2}}{2}}}{\sqrt{2\pi}}dw 
+ C{\epsilon}^{n+\frac{\delta'}{2}}\int_{\mathbb{R}}(|w|)^{2n+\delta'} \frac{e^{-\frac{w^2}{2}}}{\sqrt{2\pi}}dw \nonumber \\
\label{eq:more}
\end{eqnarray} 
which is exactly what we wanted.

On the other hand, for $\delta' \in \left(-1, 0 \right]$, we have
\begin{eqnarray}
&& \int_{\mathbb{R}^{+}}y^{\delta'}(y - v)^{2n} \frac{e^{-\frac{(v-y)^{2}}{2\epsilon}}}{\sqrt{2\pi\epsilon}}dy \nonumber \\
& & \qquad = \int_{0}^{\frac{v}{2}}y^{\delta'}(v - y)^{2n} \frac{e^{-\frac{(v-y)^{2}}{2\epsilon}}}{\sqrt{2\pi\epsilon}}dy 
+ \int_{\frac{v}{2}}^{+\infty}y^{\delta'}(v - y)^{2n} \frac{e^{-\frac{(v-y)^{2}}{2\epsilon}}}{\sqrt{2\pi\epsilon}}dy \nonumber \\
& & \qquad \leq \frac{{(2\epsilon)}^{n}}{\sqrt{\pi}}\int_{0}^{\frac{v}{2}}y^{\delta'}\left(\frac{(v - y)^2}{2\epsilon}\right)^{n+\frac{1}{2}} \frac{e^{-\frac{(v-y)^{2}}{2\epsilon}}}{|v - y|}dy
+ {\left(\frac{v}{2}\right)}^{\delta'}\int_{\frac{v}{2}}^{+\infty}(v - y)^{2n} \frac{e^{-\frac{(v-y)^{2}}{2\epsilon}}}{\sqrt{2\pi\epsilon}}dy \nonumber \\
& & \qquad \leq C(n){(2\epsilon)}^{n}\int_{0}^{\frac{v}{2}}y^{\delta'}\frac{1}{v - y}dy
+ {\left(\frac{v}{2}\right)}^{\delta'}\int_{\mathbb{R}}(v - y)^{2n} \frac{e^{-\frac{(v-y)^{2}}{2\epsilon}}}{\sqrt{2\pi\epsilon}}dy \nonumber \\
& & \qquad \leq C(n){(2\epsilon)}^{n}{\left(\frac{v}{2}\right)}^{-1}\int_{0}^{\frac{v}{2}}y^{\delta'}dy
+ {\left(\frac{v}{2}\right)}^{\delta'}{\epsilon}^{n}\int_{\mathbb{R}}w^{2n} \frac{e^{-\frac{w^{2}}{2}}}{\sqrt{2\pi}}dy \nonumber \\
& & \qquad \leq \frac{C(n){(2\epsilon)}^{n}}{\delta'+1}{\left(\frac{v}{2}\right)}^{\delta'}
+ {\left(\frac{v}{2}\right)}^{\delta'}{\epsilon}^{n}\int_{\mathbb{R}}w^{2n} \frac{e^{-\frac{w^{2}}{2}}}{\sqrt{2\pi}}dy \nonumber \\
\label{eq:less}
\end{eqnarray}
which is again what we needed and thus the proof of 1. is complete. We proceed now to the proof of 2.. 

By the Cauchy-Schwarz inquality and Fubini's Theorem we have
\begin{eqnarray}
& & \left\Vert J_{u,\epsilon}(\cdot,\,\cdot)\right\Vert _{L^{2}\left(\Lambda;\,\tilde{L}_{\delta'}^2\right)}^{2} \nonumber \\
&&\qquad = \int_{\Lambda}\int_{\mathbb{R}^{+}}y^{\delta'}\left(\int_{\mathbb{R}^{+}}2vu\left(\lambda,\,v^{2}\right)\frac{e^{-\frac{(v-y)^{2}}{4\epsilon}}}{\sqrt[4]{2\pi\epsilon}}\frac{e^{-\frac{(v-y)^{2}}{4\epsilon}}}{\sqrt[4]{2\pi\epsilon}}dv\right)^{2}dyd\mu(\lambda) \nonumber \\
&&\qquad \leq \int_{\Lambda}\int_{\mathbb{R}^{+}}y^{\delta'}\left(\int_{\mathbb{R}^{+}}4v^{2}u^{2}\left(\lambda,\,v^{2}\right)\frac{e^{-\frac{(v-y)^{2}}{2\epsilon}}}{\sqrt{2\pi\epsilon}}dv\right)\left(\int_{\mathbb{R}}\frac{e^{-\frac{(v-y)^{2}}{2\epsilon}}}{\sqrt{2\pi\epsilon}}dv\right)dyd\mu(\lambda)
\nonumber \\
&& \qquad = \int_{\Lambda}\int_{\mathbb{R}^{+}}4v^{2}u^{2}\left(\lambda,\,v^{2}\right)\int_{\mathbb{R}^{+}}y^{\delta'}\frac{e^{-\frac{(v-y)^{2}}{2\epsilon}}}{\sqrt{2\pi\epsilon}}dydvd\mu(\lambda)
\nonumber \\
\label{eq:5.2}
\end{eqnarray}

Next, we see that

\begin{eqnarray*}
4v^{2}u^{2}\left(\lambda,\,v^{2}\right)\int_{\mathbb{R}^{+}}y^{\delta'}\frac{e^{-\frac{(v-y)^{2}}{2\epsilon}}}{\sqrt{2\pi\epsilon}}dy \rightarrow 4v^{2+\delta'}u^2\left(\lambda, \, v^2\right)
\end{eqnarray*} 
as $\epsilon \rightarrow 0^{+}$ for $v \geq 0$, and it can also be bounded by something integrable, uniformly in $\epsilon > 0$ (this can be seen by recalling \eqref{eq:more} and \eqref{eq:less} for $n=0$). Thus, by the Dominated Convergence Theorem, the RHS of \eqref{eq:5.2} converges to
\begin{eqnarray*}
\int_{\Lambda}\int_{\mathbb{R}^{+}}4v^{\delta'+2}u^{2}\left(\lambda,\,v^{2}\right)dvd\mu(\lambda)=\left\Vert J_{u}(\cdot,\,\cdot)\right\Vert _{L^{2}\left(\Lambda;\,L_{y^{\delta'}}^{2}\left(\mathbb{R}^{+}
\right)\right)}^{2}
\end{eqnarray*}
as $\epsilon\rightarrow0^{+}$. Therefore, we obtain
\begin{equation}
\limsup_{\epsilon\rightarrow0^{+}}\left\Vert J_{u,\epsilon}(\cdot,\,\cdot)
\right\Vert _{L^{2}\left(\Lambda;\,L_{y^{\delta'}}^{2}\left(\mathbb{R}^{+}
\right)\right)}^{2}
\leq\left\Vert J_{u}(\cdot,\,\cdot)\right\Vert _{L^{2}\left(\Lambda;\,L_{y^{\delta'}}^{2}\left(\mathbb{R}^{+}
\right)\right)}^{2}. \label{eq:5.3}
\end{equation}

Next, fix a measurable $A \subset \Lambda$ with $\mu(A) < +\infty$ and a smooth function $f:\,\mathbb{R}^{+}\rightarrow\mathbb{R}$
supported in some interval $\left[M_{1},\,M_{2}\right]$, where $0<M_1<M_2$. Then it holds that
\[
\int_{\mathbb{R}^{+}}f(y)\frac{e^{-\frac{(v-y)^{2}}{2\epsilon}}}{\sqrt{2\pi\epsilon}}dy\rightarrow f(v)
\]
pointwise as $\epsilon\rightarrow0^{+}$. Furthermore we have
\[
\int_{\mathbb{R}^{+}}f(y)\frac{e^{-\frac{(v-y)^{2}}{2\epsilon_{k_{m}}}}}{\sqrt{2\pi\epsilon_{k_{m}}}}dy
\leq C\sup_{y\in\mathbb{R}}\left|f(y)\right|\begin{cases} 1\,, & v\leq2M_{2}\\
\\
\frac{2M_{2}}{\left|v-M_{2}\right|}\,, & v>2M_{2}
\end{cases}
\]
for all $m\in\mathbb{N}$ and for some constant $C>0$ and thus, by applying Cauchy-Schwartz we obtain
\begin{eqnarray*}
& & \int_{\Lambda}\int_{M_2}^{+\infty}\mathbb{I}_{A}(\lambda)
2vu(\lambda,\,v^{2})\frac{2M_{2}}{\left|v-M_{2}\right|}dvd\mu(\lambda)
\\
& & \qquad \leq 2M_{2}\left\Vert J_{u}(\cdot,\,\cdot)\right\Vert _{L^{2}
\left(\Lambda;\,L_{y^{\delta'}}^{2}\left(\mathbb{R}^{+}\right)\right)}
\left(\mu(A)\int_{v>2M_{2}}\frac{1}{v^{\delta'}\left|v-M_{2}\right|^{2}}dv\right)^{1/2}
\end{eqnarray*}
which is finite. This means that we can apply the Dominated Convergence Theorem to obtain

\begin{eqnarray}
&& \lim_{\epsilon \rightarrow 0^{+}}\int_{\Lambda}\int_{\mathbb{R}^{+}}J_{u,\epsilon}(\lambda,\,y)f(y)\mathbb{I}_{A}(\lambda)dyd\mu(\lambda) \nonumber \\
& & \qquad =\lim_{\epsilon \rightarrow 0^{+}}\int_{\Lambda}\int_{\mathbb{R}^{+}}f(y)\left(\int_{\mathbb{R}}u(\lambda,\,z)\frac{e^{-\frac{(\sqrt{z}-y)^{2}}{2\epsilon}}}{\sqrt{2\pi\epsilon}}\mathbb{I}_{A}(\lambda)dz\right)dyd\mu(\lambda) \nonumber \\
& & \qquad =\lim_{\epsilon \rightarrow 0^{+}}\int_{\Lambda}\int_{\mathbb{R}^{+}}
f(y)\left(\int_{\mathbb{R}^{+}}2vu(\lambda,\,v^{2})\frac{e^{-\frac{(v-y)^{2}}{2\epsilon}}}{\sqrt{2\pi\epsilon}}\mathbb{I}_{A}(\lambda)dv\right)dyd\mu(\lambda)
\nonumber \\
& & \qquad =\lim_{\epsilon \rightarrow 0^{+}}\int_{\Lambda}\int_{\mathbb{R}^{+}}\mathbb{I}_{A}(\lambda)
2vu(\lambda,\,v^{2})\left(\int_{\mathbb{R}^{+}}f(y)\frac{e^{-\frac{(v-y)^{2}}{2\epsilon}}}{\sqrt{2\pi\epsilon}}dy\right)dvd\mu(\lambda)
\nonumber \\
& & \qquad = \int_{\Lambda}\int_{\mathbb{R}^{+}}\mathbb{I}_{A}(\lambda)
2vu(\lambda,\,v^{2})f(v)dvd\mu(\lambda) \label{eq:5.4}
\end{eqnarray}
so we deduce that $J_{u,\epsilon}(\cdot,\,\cdot)\rightarrow J_{u}(\cdot,\,\cdot)$
weakly in the Hilbert space $L^{2}\left(\Lambda;\,L_{y^{\delta'}}^{2}\left(\mathbb{R}^{+}
\right)\right)$ (since $A$ and $f$ are arbitrary).
Since a Hilbert space is always a uniformly convex space, by recalling
\eqref{eq:5.3} and Proposition~III.30 from \cite{Brezis} (page 75), we deduce
that $J_{u,\epsilon}(\cdot,\,\cdot)\rightarrow I_{u}(\cdot,\,\cdot)$
strongly in $L^{2}\left(\Lambda;\,L_{y^{\delta'}}^{2}\left(\mathbb{R}^{+}
\right)\right)$,
which implies 2.
\end{proof}

\begin{proof}[\textbf{Proof of Lemma 5.4}]
First, by our boundedness assumption we have that for any sequence $\left\{ \epsilon_{k}\right\}_{k\in\mathbb{N}}$
converging to $0^{+}$, there exists a decreasing subsequence $\left\{ \epsilon_{k_{m}}\right\} 
_{m\in\mathbb{N}}$ and an element $J_{u}^{l}\in L^{2}\left(\Lambda; \, L_{y^{\delta'}}^{2}\left(\mathbb{R}^{+}
\right)\right)$ for all $l\in\{1,\,2,\,...,\,n\}$, such that
\[
\frac{\partial^{l}}{\partial y^{l}}J_{u,\epsilon_{k_{m}}}\rightarrow J_{u}^{l},
\]
for all $l\in\{1,\,2,\,...,\,n\}$, weakly in $L^{2}\left(\Lambda; \, L_{y^{\delta'}}^{2}\left(\mathbb{R}^{+}
\right)\right)$
as $m\rightarrow +\infty$. Then, for any measurable $A \subset \Lambda$ with $\mu(A) < +\infty$ and any smooth and compactly supported function $f(z)$, we have
\[
\int_{\mathbb{R}}\frac{\partial^{l}}{\partial z^{l}}f(z)\frac{e^{-\frac{\left(z-y\right)^{2}}{2\epsilon}}}{\sqrt{2\pi\epsilon}}
dz\rightarrow\frac{\partial^{l}}{\partial y^{l}}f(y)
\]
pointwise as $\epsilon\rightarrow0^{+}$. Hence, we can
use Fubini's Theorem and the Dominated Convergence Theorem as we did
in the proof of Lemma~\ref{lem:5.3} (but for a partial derivative of $f$)
to obtain

\begin{eqnarray*}
& & \lim_{m\rightarrow\infty}\int_{\Lambda}
\int_{\mathbb{R}^{+}}\mathbb{I}_{A}(\lambda)J_{u,\epsilon_{k_{m}}}(\lambda,\,y)\frac{\partial^{l}}{\partial y^{l}}
f(y)dyd\mu(\lambda) \\
&& \qquad = \lim_{m\rightarrow\infty}\int_{\Lambda}\int_{\mathbb{R}}\mathbb{I}_{A}(\lambda)\frac{\partial^{l}}{\partial y^{l}}f(y)\left(\int_{\mathbb{R}}u(\lambda, \, z)\frac{e^{-\frac{(\sqrt{z}-y)^{2}}{2\epsilon_{k_{m}}}}}{\sqrt{2\pi\epsilon_{k_{m}}}}dz\right)dyd\mu(\lambda)
\\
& & \qquad =\lim_{m\rightarrow\infty}\int_{\Lambda}\int_{\mathbb{R}^{+}}\mathbb{I}_{A}(\lambda)\frac{\partial^{l}}{\partial y^{l}}
f(y)\left(\int_{\mathbb{R}^{+}}2vu(\lambda,\,v^{2})\frac{e^{-\frac{(v-y)^{2}}{2\epsilon_{k_{m}}}}}{\sqrt{2\pi\epsilon_{k_{m}}}}dv\right)dyd\mu(\lambda) \\
& & \qquad =\lim_{m\rightarrow\infty}\int_{\Lambda}\int_{\mathbb{R}^{+}}\mathbb{I}_{A}(\lambda)J_{u}(\lambda,\,v)\left(\int_{\mathbb{R}}\frac{\partial^{l}}{\partial y^{l}}f(y)\frac{e^{-\frac{(v-y)^{2}}{2\epsilon_{k_{m}}}}}{\sqrt{2\pi\epsilon_{k_{m}}}}dy\right)dvd\mu(\lambda)
\\
& & \qquad = \int_{\Lambda}\int_{\mathbb{R}^{+}}\mathbb{I}_{A}(\lambda)J_{u}(\lambda,\,v)\frac{\partial^{l}}{\partial y^{l}}f(v)dvd\mu(\lambda)
\end{eqnarray*}
and thus we have
\begin{eqnarray*} 
& & \int_{\Lambda}\int_{\mathbb{R}^{+}}\mathbb{I}_{A}(\lambda)J_{u}^{l}
(\lambda,\,y)f(y)dyd\mu(\lambda) \\
& & \qquad =\lim_{m\rightarrow\infty}\int_{\Lambda}
\int_{\mathbb{R}^{+}}\mathbb{I}_{A}(\lambda)\frac{\partial^{l}}{\partial y^{l}} J_{u,\epsilon_{k_{m}}}(\lambda,\,y) 
f(y)dyd\mu(\lambda) \\
& & \qquad = \lim_{m\rightarrow\infty} \left(-1\right)^{l}\int_{\Lambda}
\int_{\mathbb{R}^{+}}\mathbb{I}_{A}(\lambda)J_{u,\epsilon_{k_{m}}}(\lambda,\,y)\frac{\partial^{l}}{\partial y^{l}}
f(y)dyd\mu(\lambda) \\
& & \qquad =\left(-1\right)^{l}\int_{\Lambda}\int_{\mathbb{R}} 
\mathbb{I}_{A}(\lambda)J_{u}(\lambda,\,v)\frac{\partial^{l}}{\partial v^{l}}f(v)dvd\mu(\lambda),
\end{eqnarray*}
which means that $J_{u}^{l}$ is the $l$-th weak derivative
of $J_{u}$. Next, for any $l\leq n$ we have
\begin{eqnarray}
& & \left\Vert \frac{\partial^{l}}{\partial y^{l}}J_{u,\epsilon}(\cdot,\,\cdot)
\right\Vert _{L^{2}\left(\Lambda;\,L_{y^{\delta'}}^{2}\left(\mathbb{R}^{+}
\right)\right)}^{2} \nonumber \\
& & \qquad =\int_{\Lambda}\int_{\mathbb{R}^{+}}y^{\delta'}
\left(\frac{\partial^{l}}{\partial y^{l}}\int_{\mathbb{R}^{+}}2vu\left(\lambda,\,v^{2}\right)
\frac{e^{-\frac{(v-y)^{2}}{2\epsilon}}}{\sqrt{2\pi\epsilon}}dv\right)^{2}dyd\mu(\lambda)
\nonumber \\
& & \qquad =\int_{\Lambda}\int_{\mathbb{R}^{+}}y^{\delta'}
\left(\int_{\mathbb{R}^{+}}2vu\left(\lambda,\,v^{2}\right)\frac{\partial^{l}}{\partial y^{l}} 
\left(\frac{e^{-\frac{(v-y)^{2}}{2\epsilon}}}{\sqrt{2\pi\epsilon}}\right)dv\right)^{2}dyd\mu(\lambda) \nonumber \\
& & \qquad =\int_{\Lambda}\int_{\mathbb{R}^{+}}y^{\delta'}\left(\int_{\mathbb{R}^{+}}
J_{u}^{l}\left(\lambda,\,v\right)\frac{e^{-\frac{(v-y)^{2}}{4\epsilon}}}{\sqrt[4]{2\pi\epsilon}}
\frac{e^{-\frac{(v-y)^{2}}{4\epsilon}}}{\sqrt[4]{2\pi\epsilon}}dv\right)^{2}dyd\mu(\lambda) \nonumber \\
& & \qquad \leq \int_{\Lambda}\int_{\mathbb{R}^{+}}y^{\delta'}\left(\int_{\mathbb{R}^{+}}
\left(J_{u}^{l}\left(\lambda,\,v\right)\right)^{2}\frac{e^{-\frac{(v-y)^{2}}{2\epsilon}}}
{\sqrt{2\pi\epsilon}}dv\right)\left(\int_{\mathbb{R}}\frac{e^{-\frac{(v-y)^{2}}{2\epsilon}}}
{\sqrt{2\pi\epsilon}}dv\right)dyd\mu(\lambda) \nonumber \\
& & \qquad = \int_{\Lambda}\int_{\mathbb{R}^{+}}\left(J_{u}^{l}\left(\lambda,\,v\right)\right)^{2}\int_{\mathbb{R}^{+}}y^{\delta'}\frac{e^{-\frac{(v-y)^{2}}{2\epsilon}}}{\sqrt{2\pi\epsilon}}dydvd\mu(\lambda) \nonumber \\
\end{eqnarray}
which converges (by the same argument as in \eqref{eq:5.2} in the proof of Lemma~\ref{lem:5.3}) to
\begin{eqnarray*}
\int_{\Lambda}\int_{\mathbb{R}^{+}}v^{\delta'}\left(J_{u}^{l}\left(\lambda,\,v\right)\right)^{2}dvd\mu(\lambda) = \left\Vert J_{u}^{l}(\cdot,\,\cdot)\right\Vert _{L^{2}\left(\Lambda;\,L_{y^{\delta'}}^{2}\left(\mathbb{R}^{+}
\right)\right)}^{2}
\end{eqnarray*}
and thus we have
\begin{equation}
 \limsup_{\epsilon\rightarrow0^{+}}\left\Vert \frac{\partial^{l}}{\partial y^{l}}J_{u,\epsilon}(\cdot,\,\cdot)\right\Vert _{L^{2}\left(\Lambda;\,L_{y^{\delta'}}^{2}\left(\mathbb{R}^{+}
\right)\right)}^{2} \\
 \leq  \left\Vert J_{u}^{l}(\cdot,\,\cdot)\right\Vert _{L^{2}\left(\Lambda;\,L_{y^{\delta'}}^{2}\left(\mathbb{R}^{+}
\right)\right)}^{2}. \label{eq:5.7}
\end{equation}
 
Hence, by recalling Proposition~III.30 from \cite{Brezis} (as we did in the proof of Lemma~\ref{lem:5.3}), we can conclude that $\frac{\partial^{l}}{\partial z^{l}}J_{u,\epsilon_{k_{m}}}\rightarrow J_{u}^{l}$ as 
$m\rightarrow +\infty$,
strongly in the uniformly convex space $L^{2}\left(\Lambda;\,L_{y^{\delta'}}^{2}\left(\mathbb{R}^{+}
\right)\right)$, for 
all $l\in\{1,\,2,\,...,\,n\}$. The desired
result follows since the sequence $\left\{ \epsilon_{m}\right\} _{m\in\mathbb{N}}$
is arbitrary and since a weak derivative is always unique.
\end{proof}

\end{document}